\documentclass{article}
\usepackage{latexsym, amssymb, amsmath, amscd, amsthm, verbatim, a4, mathdots}
\usepackage{amsfonts,cmmib57,color,graphics,bm}
\usepackage[pdftex]{graphicx}
\usepackage{dsfont}
\usepackage{url}
\usepackage[latin1]{inputenc}
\usepackage[active]{srcltx}
\usepackage[linesnumbered,algosection,ruled]{algorithm2e}
\usepackage{subfigure}
\usepackage{multicol}
\usepackage{cancel}
\usepackage{xfrac}
\usepackage{tikz}
\usetikzlibrary{decorations.pathreplacing}
\usetikzlibrary{arrows,matrix,positioning}

\newtheorem{teor}{Theorem}[section]
\newtheorem*{teor*}{Theorem}
\newtheorem{theorem}[teor]{Theorem}
\newtheorem{lema}[teor]{Lemma}
\newtheorem{lemma}[teor]{Lemma}
\theoremstyle{definition}
\newtheorem{defi}[teor]{Definition}
\newtheorem{definition}[teor]{Definition}
\newtheorem{ejem}[teor]{Example}

\theoremstyle{prop}
\newtheorem{prop}[teor]{Proposition}

\theoremstyle{coro}
\newtheorem{coro}[teor]{Corollary}
\newtheorem{remark}[teor]{Remark}

\numberwithin{equation}{section}
\def\kbyk{k \times k}
\def\mbym{m \times m}
\def\mbyn{m \times n}
\def\nbyn{n \times n}

\newcommand {\mat}  [1] {\left[\begin{array}{#1}}
	\newcommand {\rix}      {\end{array}\right]}

\newcommand{\la}{\ensuremath{\lambda}}

\def\absF#1{\lvert\,#1\,\rvert_{{}_{\scriptstyle \cF}}}

\def\scF{{}_{\scriptstyle \cF}}
\def\scG{{}_{\scriptstyle \cG}}

\def\scFo{{}_{\scriptstyle \cF_1}}
\def\scFt{{}_{\scriptstyle \cF_2}}
\def\scFth{{}_{\scriptstyle \cF_3}}
\def\scFf{{}_{\scriptstyle \cF_4}}

\def\scFj{{}_{\scriptstyle \cF_j}}
\def\scFk{{}_{\scriptstyle \cF_k}}
\def\scFell{{}_{\scriptstyle \cF_{\ell}}}

\def\max{\mathop{\rm max}}
\def\rank{\mathop{\rm rank}}
\def\min{\mathop{\rm min}}
\def\rev{\mathop{\rm rev}}
\def\diag{\mathop{\rm diag}}
\def\det{\mathop{\rm det}}
\def\grade{\mathop{\rm grade}}

\def\mf#1{\mathsf{m}_{#1}^{}}

\def\rom#1{{\upshape #1}}
\newcommand{\parens}[1]{\rom{(}#1\rom{)}}

\def\mystrut#1{\rule{0cm}{#1}}  % puts vspace after rule

\newcommand{\und}{\mbox{$\quad\mbox{ and }\quad$}}

\newcommand{\cC}{{\cal C}}

\newcommand{\cF}{{\cal F}}
\newcommand{\cG}{{\cal G}}

\newcommand{\cI}{{\cal I}}
\newcommand{\cJ}{{\cal J}}

\newcommand{\cL}{{\cal L}}
\newcommand{\cM}{{\cal M}}

\newcommand{\cP}{{\cal P}}

\newcommand{\cS}{{\cal S}}

\newcommand{\bF}{\mathbb{F}}

\newcommand{\bN}{\mathbb{N}}

\newcommand{\bQ}{\mathbb{Q}}
\newcommand{\bR}{\mathbb{R}}

\newcommand{\bZ}{\mathbb{Z}}

\newcommand {\basicbox} 		{\begin{tcolorbox}[width=\textwidth,colback={black!1}]}

\newcommand {\overbox}      {\end{tcolorbox}}

\newcommand {\rbasicbox} 		
{\begin{tcolorbox}[width=\textwidth,colback={red!5}]}

\newcommand {\bbasicbox} 		
{\begin{tcolorbox}[width=\textwidth,colback={blue!5}]}

\newcommand {\gbasicbox} 		
{\begin{tcolorbox}[width=\textwidth,colback={green!10}]}

\newcommand {\ybasicbox} 		
{\begin{tcolorbox}[width=\textwidth,colback={yellow!10}]}

\def\mf#1{\mathsf{m}_{#1}^{}}

\def\MT{\mathbf{M}}

\def\PM{\cP\!\cM}
\def\PMg{\PM_{\text{given}}}

\usepackage{colortbl}
\definecolor{lightgrey}{rgb}{.9,.9,.9}
\definecolor{mediumgrey}{rgb}{.6,.6,.6}
\definecolor{darkgrey}{rgb}{.3,.3,.3}
\definecolor{lightgreen}{rgb}{0.7,1.0,0.7}
\definecolor{mediumgreen}{rgb}{0.3,1.0,0.3}
\definecolor{darkgreen}{rgb}{0.0,0.7,0.0}
\definecolor{lightred}{rgb}{1.0,0.6,0.6}
\definecolor{mediumred}{rgb}{1.0,0.3,0.3}
\definecolor{darkred}{rgb}{0.8,0.0,0.0}
\definecolor{lightblue}{rgb}{0.8,0.8,1.0}
\definecolor{mediumblue}{rgb}{0.5,0.5,1.0}
\definecolor{darkblue}{rgb}{0.1,0.1,0.9}
\definecolor{green}{rgb}{0.0,0.7,0.0}
\definecolor{red}{rgb}{1.0,0.0,0.0}
\definecolor{magenta}{rgb}{.75,0,.75}
\definecolor{hotpink}{rgb}{0.9,0,0.5}
\definecolor{cyan}{rgb}{0,0.8, .8}
\def\tcr#1{\textcolor{red}{#1}}
\def\tcb#1{\textcolor{blue}{#1}}
\def\tcg#1{\textcolor{darkgreen}{#1}}

\mathcode`@="8000 % Make @ behave as per catcode 13 (active).  TeXbook p. 155.
{\catcode`\@=\active\gdef@{\mkern1mu}}
\def\wt{\widetilde}
\def\wh{\widehat}
\def\ov{\overline}

\numberwithin{equation}{section}

\title{Quasi-Triangularization of Matrix Polynomials \\ over Arbitrary Fields
\thanks{This publication is part of the ``Proyecto de I+D+i PID2019-106362GB-I00 financiado
        por MCIN/AEI/10.13039/501100011033''. It has been also funded by
        ``Ministerio de Econom\'ia, Industria y Competitividad (MINECO)'' of Spain
        through grants MTM-2015-65798-P and BES-2013-065688.}}

\author{L.M. Anguas \thanks{Departamento de Matem\'atica e Inform\'atica Aplicadas a las Ingenier\'ias
                            Civil y Naval, Escuela T\'ecnica Superior de Ingenieros de Caminos,
                            Canales y Puertos, Universidad Polit\'ecnica de Madrid, Profesor
                            Aranguren, 3, 28040 Madrid, Spain ({\tt l.anguas@upm.es})}
        \and F. M. Dopico \thanks{Departamento de Matem\'aticas, Universidad Carlos III de Madrid,
		  Avda.\ Universidad 30, 28911 Legan\'es, Spain ({\tt dopico@math.uc3m.es})}
		\and R. Hollister\thanks{Department of Mathematics,
          University at Buffalo, Buffalo NY 14260 USA ({\tt rahollis@buffalo.edu})}
        \and D.S. Mackey\thanks{Department of Mathematics,
          Western Michigan University, Kalamazoo MI 49008-5248 USA ({\tt steve.mackey@wmich.edu})}}

\begin{document}
	
	\maketitle
	
	\begin{abstract}
		In \cite{TasTisZab}, Taslaman, Tisseur, and Zaballa
        show that any regular matrix polynomial $P(\la)$ over an algebraically closed field
        is spectrally equivalent to a triangular matrix polynomial of the same degree.
        When $P(\la)$ is real and regular,
        they also show that there is a real quasi-triangular matrix polynomial of the same degree
        that is spectrally equivalent to $P(\la)$,
        in which the diagonal blocks are of size at most $2 \times 2$.
        This paper generalizes these results to regular matrix polynomials $P(\la)$
        over \emph{arbitrary fields} $\bF$,
        showing that any such $P(\la)$ can be quasi-triangularized
        to a spectrally equivalent matrix polynomial over $\bF$ of the same degree,
        in which the largest diagonal block size
        is bounded by the highest degree appearing among
        all of the $\bF$-irreducible factors in the Smith form for $P(\la)$.
	\end{abstract}
	
	%	\begin{keywords}
	%	\end{keywords}
	%	
	%	\begin{AMS}
	%		
	%	\end{AMS}

{\small
{\bf Key words.} matrix polynomials, triangularization, arbitrary field, majorization, inverse problem, Mobius transformation.
 \\

{\bf AMS subject classification.} 15A18, 15A21, 15A54.}

	\pagestyle{myheadings}
	\thispagestyle{plain}
	
%%%%%%%%%%%%%%%%%%%%%%%%%%%%%%%%%%%%%%%%%%%%%%%%%%%%%%%%%%%%%%%%%%%%%%%%%%%%	
\section{Introduction} \label{sect.intro}

Triangularizations of matrix polynomials
that preserve degree as well as the finite and infinite spectral structure
via unimodular equivalence
are essentially Schur-like forms for matrices
whose entries are polynomials.
% under the relation of unimodular equivalence.
% Triangularizations
These have been achieved
over algebraically closed fields
for regular quadratic matrix polynomials in~\cite{TisZab},
and for regular matrix polynomials of %any
arbitrary degree in~\cite{TasTisZab}.
There are also some %limited
results in~\cite{TasTisZab}
on singular matrix polynomials, but these will not be addressed in this paper.
Also in~\cite{TasTisZab,TisZab},
when the underlying field is $\bR$,
the authors show how to produce quasi-triangularizations
with diagonal blocks of size at most 2$\times$2. %when the field is $ \bR $.

The goal of this paper is similar:
given a regular matrix polynomial $ P(\la) $,
to show how to construct a regular quasi-triangular matrix polynomial
with the same finite and infinite spectral structure as $P(\la)$,
i.e., is spectrally equivalent to $P(\la)$,
and has the same degree as $P(\la)$.
However, this work is an extension %of the results of
of~\cite{TasTisZab}
in that the results presented here %in this paper
are for matrix polynomials over an \emph{arbitrary field}.
In order to achieve this generalization, though, %however,
%we must allow for
the possibility of diagonal blocks
of sizes even larger than 2$\times$2 must be allowed.
We %do, however,
show that a quasi-triangularization can always be constructed %achieved %attained
in which the sizes of the diagonal blocks
% \underline{cannot exceed} (FIX!!!)
do not exceed $k\times k $,
where $ k $ is the highest degree
among all of the irreducible factors of the invariant polynomials
in the Smith form of the given polynomial matrix $ P(\la) $.
% In Section~\ref{sect.diagonal-block-size-bounds},
% we give a sufficient, but not necessary, condition
% for describing when exact triangularization is possible,
% along with examples illustrating the sharpness of $ k $
% as an upper bound on the size of the diagonal blocks.
Note that the term \emph{quasi-triangular} is used throughout this paper
to refer to square matrices
that are block upper (or lower) triangular
with square blocks along the main diagonal,
at least one of which has size 2$\times$2 or larger.
A matrix is \emph{$k$-quasi-triangular}
if the diagonal blocks are no larger than $ k \times k $.

Here is a brief overview of %the strategy for developing
% the results in
the paper.
After some preliminary discussion of concepts, notation, and terminology
in Section~\ref{pre},
we begin in Section~\ref{sect.realization-of-finite-spectral-data}
by solving the quasi-triangular \emph{realization problem}
for finite spectral data over an arbitrary field $\bF$.
That is, we take as starting point
a collection of finite spectral data rather than a matrix polynomial,
and show how to construct a strictly regular
$k$-quasi-triangular matrix polynomial over $\bF$
having exactly the given spectral data;
here $k$ is the largest degree
among the irreducible divisors of the given data.
The solution of this inverse problem
is the main technical result of the paper;
all other results depend on and follow from this.
The central idea of the proof is to take the given spectral data,
form the corresponding Smith form,
and then systematically ``un-diagonalize'' in a way that moves
the matrix polynomial toward the desired degree,
while maintaining quasi-triangularity.
This is a proof technique %first used %pioneered
used in~\cite{matrixpolynomials},
and then developed by~\cite{TasTisZab};
some antecedents of this technique
can also be found in~\cite{MarquesdeSa}.
% \begin{note}
%   \tcb{Was it \emph{really} {\bf first} used in \cite{matrixpolynomials}?
%        That is just the first place that I know about where it was used.}
% \end{note}
We develop it further here,
adapting it to the arbitrary field setting.
Indeed, a number of nontrivial ingredients
go into proving this quasi-triangular realization result by this method ---
majorization plays a role,
as well as a new combinatorial lemma
on the partitioning of integer multisets.
In Section~\ref{sect.eigval-at-infinity}
we extend this realization result
to include elementary divisors at $\infty$
in the given spectral data.
In order to achieve this extension,
we use the well-known tool of M\"obius transformations~\cite{Mobius},
although once again we will need to do a significant amount of work to adapt them
to work smoothly for matrix polynomials over arbitrary fields.
At this point the signature result of the paper ---
the quasi-triangularization of any regular matrix polynomial
over an arbitrary field in Theorem~\ref{thm.QT-regular} ---
now follows easily.
 Finally, in Section~\ref{sect.diagonal-block-size-bounds}
we investigate conditions for describing when exact triangularization is possible
in the arbitrary field setting;
we also display several families of examples,
some that illustrate the sharpness of $k$
as a general upper bound on the size of the diagonal blocks in quasi-triangularizations,
and others that show that this upper bound $k$
can sometimes be a huge overestimate of the diagonal block size
that is actually attainable.

\section{Preliminaries}
\label{pre}
% In this section we are going to state some basic concepts and results we will need.

In this paper, we deal with \emph{polynomial matrices}
(also referred to as \emph{matrix polynomials}) over an arbitrary field $ \bF $,
i.e., matrices whose entries are polynomials with coefficients from $ \bF $.
In particular, we will be working only with \emph{regular} matrix polynomials,
that is, square matrix polynomials with determinant
different from the zero polynomial.
A polynomial matrix is \emph{unimodular} if it is regular
and has a (nonzero) constant determinant. % (different from zero).
Throughout the paper the set of natural numbers are denoted by $\bN$,
and includes zero;
$\bN^{+}$ is then the set of positive integers.

The Smith form is a canonical representation of matrix polynomials
under \emph{unimodular equivalence},
i.e., obtained by left and right multiplication
by unimodular matrix polynomials.
This form was first defined for integer matrices \cite{Smith}.
We will use the extension given in \cite{Fro} for matrix polynomials:
\begin{teor}[Smith form] \label{Smithform} \quad \\
	Let $P(\lambda)$ be an $m\times n$ matrix polynomial over an arbitrary field \,$\mathbb{F}$.
	Then there exists $r\in\mathbb{N}$, and unimodular matrix polynomials $U(\lambda)$ and $V(\lambda)$
	such that
	$$U(\lambda)P(\lambda)V(\lambda)={\rm diag}(s_1(\lambda),\dots, s_{\min\{m,n\}}(\lambda))=:S(\lambda)$$
	where $s_i(\lambda)\in\mathbb{F}[\lambda]$, for $i=1,\dots, \min\{m,n\}$, $s_1(\lambda),\dots, s_r(\lambda)$ are monic, $s_{r+1}(\lambda),\dots, s_{\min\{m,n\}}(\lambda)$ are identically-zero,
	and  $s_i(\lambda)$ is a divisor of $s_{i+1}(\lambda)$ for $i=1,\dots, r-1$.
	Moreover, the number $r$ is equal to the rank of $P$,
	and the diagonal entries of the $m\times n$ matrix polynomial $S(\lambda)$
	are uniquely determined by the multiplicative relations
	\begin{equation} \label{eqn.gcd-characterization}
       s_1(\la) s_2(\la) \dotsb s_j(\la)
           \;=\; \gcd\bigl\{ \,\text{all } j \times j \text{ minors of $P(\la)$\,} \bigr\} \,,
                   \;\text{ for }\; j = 1, \dotsc, r \,.
    \end{equation}
	When $P(\la)$ is regular, then $r = m = n$,
	and $S(\la)$ is a nonsingular diagonal matrix.
\end{teor}

\noindent
The diagonal matrix $ S(\la) $ featuring in this theorem is called the \emph{Smith form} of $ P(\la) $.
The nonzero diagonal entries of $ S(\la) $ are called the \emph{invariant polynomials} of $ P(\la) $,
and their zeros are the \emph{finite eigenvalues} of $ P(\la) $.
An invariant polynomial will be called \emph{trivial}
if it is identically equal to 1 and \emph{nontrivial} otherwise.

A non-constant irreducible polynomial $ \chi(\la) \in \bF[\la] $
that divides some invariant polynomial of $P(\la)$
will be called an \emph{irreducible divisor} of $P(\la)$.
% This non-standard terminology is adopted in this work
% because of the emphasis that will be given to ... key role that will be played by ...
% these irreducible divisors.
This new concept and terminology is adopted in this work
because of the central role that will be played here by these objects.
Given an invariant polynomial $s_i(\la)$ and an irreducible divisor $\chi(\la)$,
or indeed any irreducible polynomial $\chi(\la)$,
there is a unique natural number $ \alpha_i $ (perhaps zero) such that
\[
  s_i(\la) \,=\, \chi(\la)^{\alpha_i} @@ \wh{s}_i(\la) \,,
\]
with $ \wh{s}_i(\la) $ not divisible by $ \chi(\la) $.
Any factor $ \chi(\la)^{\alpha_i} $ with $\alpha_i > 0$ is traditionally called
an \emph{elementary divisor} \cite{Gant59} of $ P(\la) $.
The number $ \alpha_i $, whether it is zero or nonzero,
is called the \emph{partial multiplicity} of the irreducible $ \chi(\la) $
with respect to the invariant polynomial $ s_i(\la) $,
while the sequence $ (\alpha_1,\alpha_2,\dotsc,\alpha_r) $,
with $0 \le \alpha_1 \le \alpha_2 \le \dotsb \le \alpha_r$,
is called %the \emph{partial multiplicity sequence} of $ \chi(\la) $ for $ P(\la) $,
the \emph{partial multiplicity sequence} of $P(\la)$ at $\chi(\la)$,
which we will denote by
\begin{equation} \label{eqn.PM}
   \PM (P, \chi) :=\, (\alpha_1,\alpha_2,\dotsc,\alpha_r) \,.
\end{equation}
Note that $\PM (P, \chi)$ may consist of all zeroes,
and $\chi$ is an irreducible divisor of $P$ exactly when some $\alpha_i$ is nonzero.
The partial multiplicity sequence of a degree one irreducible divisor $ \chi(\la) = \la - \la_0 $
is also sometimes referred to as the partial multiplicity sequence of the eigenvalue $ \la_0 $.
An abbreviated notation, $\PM (\chi)$, will also %sometimes
be used for the partial multiplicity sequence associated to $\chi(\la)$.
This may be used when the underlying matrix polynomial is understood,
but more commonly it will be employed %in the situation
when there is \emph{no} matrix polynomial
in the background at all,
and the sequence being specified (and its association with $\chi$)
is part of a collection of given spectral data
that is yet to be realized.
To emphasize this role %for a partial multiplicity sequence
of being input data for a realization problem,
we will sometimes instead write $\PMg (\chi)$.
% $\PM_{\text{given}}(\chi)}$.

In using the notation \eqref{eqn.PM},
it is very useful to \emph{abandon} a standard convention,
i.e., the convention that only \emph{monic} irreducibles should be allowed or considered.
As far as partial multiplicities are concerned,
there is no important difference between an irreducible $\chi(\la)$
and any nonzero scalar multiple $c \chi(\la)$.
In particular, since for any invariant polynomial $s_i(\la)$ we have
\[
  s_i (\la) \,=\, \chi(\la)^{\alpha_i} \,@ \wh{s}_i (\la)
            \,=\, \bigl( c \chi(\la) \bigr)^{\alpha_i} \,@ \frac{\,\wh{s}_i (\la)\,}{c^{\alpha_i}}
            \,=\, \bigl( c \chi(\la) \bigr)^{\alpha_i} \,@ \wt{s}_i (\la)
\]
where $c \chi$ and $\wt{s_i}$ are coprime,
we see that the definition of the partial multiplicity sequence
$(\alpha_1, \dotsc , \alpha_r)$
is unaffected by allowing non-monic irreducibles.
Thus we may consistently say that
\begin{equation} \label{eqn.scalar-multiples}
  \PM (P, c\chi) \;=\; \PM (P, \chi)
\end{equation}
for any nonzero scalar $c \in \bF$,
and thereby simply ignore whether an irreducible is monic or not.
In line with this,
we will say that two $\bF$-irreducible polynomials are
\emph{distinct irreducibles}
if neither is a nonzero scalar multiple of the other.

\begin{remark} \label{rem.implicit-equiv-relation}
\rm
  In effect we are implicitly defining an equivalence relation on the set
  of $\bF$-irreducible polynomials.
  That is, two irreducibles are equivalent
  (with respect to partial multiplicity sequences)
  if one is a nonzero scalar multiple of the other.
  Distinct irreducibles are then just representatives
  of different equivalence classes under this relation.
\end{remark}

\begin{remark}
\rm
  It is important to keep in mind that there are fields $\bF$
  that support the presence of $\bF$-irreducible polynomials
  of arbitrarily high degree.
  A simple example of this is the field $\bQ$ of rational numbers.
  Using the Eisenstein criterion~\cite{artin91},
  it is easy to see that the polynomial $\,x^n + p\,$
  is $\bQ$-irreducible for any prime number $p \ge 2$ and any $n \ge 1$.
\end{remark}

% \begin{note}
%   Put here a remark about the presence of irreducible polynomials
%   of arbitrarily high degree over arbitrary fields.
%   As a way of emphasizing the phenomena that may occur
%   when working over an arbitrary field.
%   Highlight the example of the rational number field $\bQ$,
%   a non-exotic example supporting irreducible polynomials
%   of \emph{every} possible degree.
% \end{note}

\begin{defi}\label{FiniteStructure}
  The \emph{finite spectral structure} % (or finite spectrum)}
  of $P(\la)$
  refers to any of the following equivalent (and easily inter-convertible) collections of data:
  \begin{enumerate}
    \item[\rm (a)] the set of all distinct irreducible divisors of $P(\la)$,
                   each equipped with their partial multiplicity sequence,
    \item[\rm (b)] the multiset of all elementary divisors of $P(\la)$,
                   together with the number $\rank P$,
    \item[\rm (c)] the multiset of all invariant polynomials of $P(\la)$, including the trivial ones.
  \end{enumerate}
  (Recall that a multiset is like a set, but with repetitions allowed,
   i.e., a ``set with multiplicities'' \cite{Knuth}.)
\end{defi}

\begin{remark}
\rm
  Note that the rank of a matrix polynomial is encoded
  in each partial multiplicity sequence by its length.
\end{remark}

% The structure of some regular matrix polynomials
Some regular matrix polynomials have structure that
is not completely captured by their finite spectral structure alone.
% To ``describe'' what is missing, %
To get the full story for these matrix polynomials,
it is also necessary to include their ``spectral structure at $\infty$''.
% But to do that we need some additional terminology.
% In order to explain what this means,
 For this,
some additional terminology is needed.
The \emph{grade} of a matrix polynomial $P(\la)$ is a natural number $ g $
such that
\[
  P(\la) \,=\, P_0 + P_1\la + \dotsb + P_g\la^g \,,
\]
with each $ P_j \in \bF^{n\times n} $.
Note, however, that unlike in the definition of the degree $d$ of $P(\la)$,
there is no requirement here for the leading coefficient $ P_g $ to be nonzero.
Thus we see that $g \ge d$ always holds,
no matter what the choice of grade might be.
And we emphasize that, in contrast with degree, grade is indeed a choice,
although a very common choice is for $g$ to be taken to be equal to $d$.
The \emph{grade $ g $ reversal} of $ P(\la) $ is the matrix polynomial
\[
  (\textstyle{\rev_g} P)(\la) \,:=\; \la^g P(\sfrac{1}{\la}),
\]
and $ P(\la) $ has an \emph{eigenvalue at infinity}
if $ (\rev_g P)(\la) $ has an eigenvalue at zero.
Moreover, the partial multiplicity sequence
of the eigenvalue at infinity for $ P(\la) $
is, by definition,
identical to the partial multiplicity sequence
of the eigenvalue zero for $ (\rev_g P)(\la) $.
We will use $\PM (P, \infty)$ as temporary notation
for the partial multiplicity sequence of $P$ at $\infty$,
so that this definition can be expressed as
\[
  \PM (P, \infty) \,:=\; \PM (\textstyle{\rev_g} P, \la) \,.
\]
However, later in Section~\ref{sect.eigval-at-infinity}
we will have reason to change this notation
to something that is more consistent with the notation in \eqref{eqn.PM},
and also works more smoothly with M\"obius transformations
and their properties.

\begin{defi}\label{InfiniteStructure}
The \emph{infinite spectral structure} %(or infinite spectrum)}
of $ P(\la) $
refers to the eigenvalue at infinity (if it exists),
together with its partial multiplicity sequence.
\end{defi}

The relationship of grade to degree
(e.g., whether $g=d$ or $g>d$)
is reflected in the infinite spectral structure,
in particular in the first partial multiplicity at $\infty$.
Note that some of the following result
appeared previously in \cite[Lemma 2.17]{spectralequivalence},
but without proof.

\begin{lemma} \label{lem.grade-equal-degree}
  Suppose $P(\la)$ is any $\mbyn$ matrix polynomial over a field $\,\bF$,
  with rank $r$, degree $d$, and grade $g$.
  Let $(\alpha_1, \alpha_2, \dotsc , \alpha_r)$
  with $0 \le \alpha_1 \le \alpha_2 \le \dotsb \le \alpha_r$
  be the partial multiplicity sequence of $P$ at $\infty$.
  Then
  \begin{equation} \label{eqn.alpha1}
    \alpha_1 \,=\, g - d \,,
  \end{equation}
  and hence $\,g = d\,$ if and only if $\,\alpha_1 = 0$.
  \parens{Equivalently, $g > d$\, if and only if \,\,$\alpha_1 > 0$.}
  Furthermore, suppose $\wt{P}(\la)$ is an $\mbyn$ matrix polynomial
  that is entry-wise identical to $P(\la)$,
  but its grade has been chosen to be equal to its degree $d$.
  Then the partial multiplicity sequences at $\infty$ of $P(\la)$ and $\wt{P}(\la)$
  are related by a constant ``shift'' of $g-d$,
  that is,
  \begin{equation} \label{eqn.PM-shift}
    \PM (P, \infty) \;=\; \PM (\wt{P}, \infty) \,+\, (g-d) \cdot (1, 1, \dotsc, 1) \,.
  \end{equation}
\end{lemma}
\begin{proof}
  Begin by expressing $P(\la) = P_g \la^g + \dotsb + P_d \la^d + \dotsb + P_1 \la + P_0$,
  so that $\wt{P}(\la) = P_d \la^d + \dotsb + P_1 \la + P_0$,
  with $P_d \ne 0$ and $P_j = 0$ for the $g-d$ leading coefficients
  with $j = d+1, \dotsc, g$.
  Then the reversals of $P$ and $\wt{P}$ are related by
  \begin{equation} \label{eqn.two-reversals}
    \textstyle{\rev_g} P \,=\, \la^g P \bigl( \sfrac{1}{\la} \bigr)
                         \,=\, \la^{g-d} \la^d P \bigl( \sfrac{1}{\la} \bigr)
                         \,=\, \la^{g-d} \bigl(\textstyle{\rev_d} \wt{P} \bigr) \,.
  \end{equation}
  Note that the constant term of $\rev_d \wt{P}$ is the \emph{nonzero} matrix $P_d$.

  Next observe that for any constant $\mu_0 \in \bF$
  and any matrix polynomial $Q(\la)$ over $\bF$,
  the number of initial \emph{zero} partial multiplicities in the sequence $\PM(Q,\chi)$
  with $\chi(\la) = \la - \mu_0$ is equal to $\rank Q(\mu_0)$.
  This can be easily seen from the Smith form
  \begin{equation} \label{eqn.Smith-form}
    U(\la) Q(\la) V(\la) \;=\;  \wh{S}(\la) \,
  \end{equation}
  by evaluating \eqref{eqn.Smith-form} at $\mu_0$,
%   and then finding the rank of each side to get
  then computing ranks to get
  $\rank \bigl( Q(\mu_0) \bigr) \,=\, \rank \bigl( \wh{S}(\mu_0) \bigr)$\,.
  But any diagonal entry of $\wh{S}(\la)$ corresponding
  to a nonzero partial multiplicity of $\la - \mu_0$
  will be zeroed out in $\wh{S}(\mu_0)$,
  so the remaining nonzero diagonal entries of $\wh{S}(\mu_0)$,
%   (i.e., $\rank S(\mu_0)$ of them),
  there are $\rank \wh{S}(\mu_0) = \rank Q(\mu_0)$ many of them,
%   will correspond to the zero partial multiplicities of $\la - \mu_0$.
  will count the zero partial multiplicities of $\la - \mu_0$ for $Q$.
  Applying this to $Q(\la) = \rev_d \wt{P}$, we see that
  \begin{align*}
    (\text{no. of zero partial multiplicities for $\wt{P}$ at $\infty$})
    &\;=\; (\text{no.\ of zero partial multiplicities for $\textstyle{\rev_d} \wt{P}$ at $(\la - 0)$}) \\
    &\;=\; \rank \bigl( \textstyle{\rev_d} \wt{P} (0) \bigr)
     \;=\; \rank P_d \,>\, 0 \,,
  \end{align*}
  since $P_d$ is nonzero.
  Thus $\PM (\wt{P}, \infty) := \PM \bigl(\textstyle{\rev_d} \wt{P}, (\la - 0) \bigr) = (0, \dotsc )$.
  Hence in the Smith form $\wt{S}(\la)$ for $\rev_d \wt{P}$,
  i.e., $\wt{S}(\la) = \diag(\wt{s}_1(\la), \wt{s}_2(\la), \dotsc )$,
  we see that $\wt{s}_1(\la)$ must be coprime to $\la$.
  Now from \eqref{eqn.two-reversals} it follows that the Smith form $S(\la)$ for $\rev_g P$
  is related to $\wt{S}(\la)$ by
  \[
    S(\la) \;=\; \la^{g-d} @\wt{S}(\la) \,.
  \]
  Hence we can immmediately conclude that $\alpha_1 = g-d$,
  and that the ``shift'' relation in \eqref{eqn.PM-shift} holds.
\end{proof}

% Much of the work
The first major result in this paper will be concerned
with a special class of regular matrix polynomials,
as in the following definition.
\begin{definition} \label{def.strictly-regular}
  Regular matrix polynomials that have \emph{no} infinite spectral structure at all
  will be referred to as \emph{strictly regular}.
\end{definition}
\noindent
For a regular $\nbyn$ matrix polynomial $P(\la)$ of grade $ g $,
being strictly regular is equivalent
to the leading coefficient $P_g$ being \emph{nonsingular}
(and hence necessarily also that $g=d$),
or equivalently to $\deg \bigl( \det P(\la) \bigr)$ being equal to $gn$.

\begin{defi}\label{RegStructure}
  The \emph{complete spectral data} of a matrix polynomial $P(\la)$
  is the combination of the finite and infinite spectral structures.
  Two regular matrix polynomials with the same complete spectral data
  are said to be \emph{spectrally equivalent}.
\end{defi}
\noindent
The notion of spectral equivalence %for general matrix polynomials
was introduced in \cite{spectralequivalence} as a way to compare matrix polynomials,
both regular and singular, even if they have different sizes or degrees.
Although the definition given in \cite{spectralequivalence}
is quite different than Definition~\ref{RegStructure},
it was shown in \cite{spectralequivalence}
that for regular polynomials being spectrally equivalent
is the same as having the same complete spectral data.
Note that in \cite{TasTisZab}, having the same complete spectral data
is termed \emph{strongly equivalent}.

It is important to keep firmly in mind the contrast between
unimodular equivalence and spectral equivalence.
Unimodular equivalence preserves all finite spectral structure,
but carries with it the unfortunate possibility of altering
any infinite spectral structure that might be present.
We will see a concrete illustration of this phenomenon
in Example~\ref{ex.full-reg},
but a more detailed discussion
of the possible effects of unimodular transformations
on infinite spectral structure
can be found in~\cite[Sect.\ 4.3]{spectralequivalence}.
In this paper the ultimate aim is to produce quasi-triangularizations
that are spectrally equivalent to the given matrix polynomial,
and not just unimodularly equivalent.

The following theorem, dubbed the Index Sum Theorem in \cite{spectralequivalence},
appeared for the first time in \cite{VerVDKai} for the complex field,
and in \cite{NevPraag} and \cite{Praag} for the real field.
% and in \cite{VerVDKai} for the complex field.
This theorem describes the fundamental relationship
between the grade, rank, and structural data
(elementary divisors together with minimal indices)
% (includes also the left and right minimal indices when poly is singular and not regular)
of \emph{any} matrix polynomial,
and was recently extended to arbitrary fields in \cite{spectralequivalence}.
Here we state the result only for regular matrix polynomials over an arbitrary field;
the result in its full generality can be found in~\cite{spectralequivalence}.
 For more on the history of this theorem,
as well as its connection to other fundamental results, see~\cite{VanDoorIndexSum}.

% \begin{note}
%   \tcb{The index sum theorem for just the \emph{regular} case was probably known even earlier;
%   perhaps in GLR, or earlier than that?
%   It is only the result for \emph{general} matrix polynomials
%   for which these ``first-time appearance'' references are relevant.
%   Can we find any earlier instance(s) in the literature
%   of just a ``\emph{regular}'' index sum theorem?
%   This might be appropriate, given that we are only stating and using the regular
%   version of the theorem in this paper.}
% \end{note}

\begin{teor}[Index Sum Theorem for Regular Matrix Polynomials]
    \label{thm.index-sum} \quad \\
	Let $P(\lambda)$ be a regular $\nbyn$ matrix polynomial of degree $d$ and grade $g$
	having the %following spectral structure:
	the complete spectral data:
	\begin{itemize}
		\item invariant polynomials $p_j(\lambda)$ of degrees $\delta_j$, for $j=1,\dots, n$,
		\item infinite partial multiplicities $\gamma_1,\dots,\gamma_n$,
	\end{itemize}
where some  %or all,
of the degrees or partial multiplicities can be zero.
Then the index sum $\sigma$ satisfies the relation
\begin{equation} \label{ISMeq}
  \sigma \,:=\; \sum_{j=1}^n \delta_j \,+\, \sum_{j=1}^n \gamma_j \;=\; gn \,.
\end{equation}
If $P(\la)$ is strictly regular,
so that all of the $\gamma_i$'s are zero and $g=d$,
then \eqref{ISMeq} simplifies to just the relation
\begin{equation} \label{ISMeq-StrictlyReg}
  \sigma \;=\; \sum_{j=1}^n \delta_j \;=\; dn \,.
\end{equation}
\end{teor}

The index sum theorem %enables
plays a key role in the following theorem,
which is a very special case of a much more general result in \cite{prescribed}.
The corollary immediately following it will be used at a key moment %later in this paper.
in the proof of our first main result, Theorem~\ref{thm.QTR-strictly-regular}.

\begin{teor}[Fundamental Realization Theorem for Strictly Regular Matrix Polynomials]
	\label{Teorprescribed} \quad \\
Let $\,\bF$ be an arbitrary field,
and $n, d \in \bN^+$ be given positive integers.
Consider a collection of $n$ monic polynomials $p_1(\la),\dots, p_n(\la)$
        with coefficients in $\,\mathbb{F}$
        and respective degrees $\delta_1,\dots,\delta_n$,
        such that $p_j(\lambda)$ divides $p_{j+1}(\lambda)$ for $j=1,\dots,n-1$.
Then there exists a strictly regular $n \times n$ matrix polynomial $P(\lambda)$
over $\,\mathbb{F}$ with degree $d$
and invariant polynomials $p_1(\lambda),\dots, p_n(\lambda)$,
if and only if \eqref{ISMeq-StrictlyReg} holds.
\end{teor}

\begin{coro}\label{cor.prescribed}
  Let $\,\bF$ be an arbitrary field,
  and $Q(\la)$ a regular $\nbyn$ polynomial matrix over $\,\bF$.
  Suppose that $\sigma := \deg \bigl(\det Q(\la) \bigr)$ is divisible by $n$,
  in particular that $\sigma = dn$.
  Then there is a strictly regular $\nbyn$ matrix polynomial $P(\la)$ over $\,\bF$
  of degree $d$
  that is unimodularly equivalent to $Q(\la)$.
\end{coro}
\begin{proof}
  Let $q_1(\la),\dots, q_n(\la)$ be the invariant polynomials of $Q(\la)$,
  with degrees $\delta_1,\dots,\delta_n$, respectively.
  Then from the Smith form of $Q(\la)$
  we know that $\sum_{j=1}^n \delta_j = \sigma = dn$,
  i.e., that \eqref{ISMeq-StrictlyReg} holds.
  Thus by Theorem~\ref{Teorprescribed}
  there exists a strictly regular $n \times n$ matrix polynomial $P(\la)$
  over $\,\mathbb{F}$ with degree $d$ and the same invariant polynomials as $Q(\la)$.
  Now since $P(\la)$ and $Q(\la)$ have the same Smith form,
  they must be unimodularly equivalent.
\end{proof}

An important concept, %for the proof of Theorem \ref{Teorprescribed}
used in~\cite{prescribed} and~\cite{TasTisZab}
for handling regular matrix polynomials that have nontrivial infinite spectral structure,
is that of M\"{o}bius transformations of matrix polynomials.
These transformations and their properties were studied in \cite{Mobius},
and will be important for the results in Section~\ref{sect.eigval-at-infinity}.

\begin{defi}[M\"obius transformations of matrix polynomials] \label{Mobius} \quad \\
  Let $P(\la)$ be a matrix polynomial of grade $ g $ over the field $ \bF $,
  and suppose %$ A = \left[\begin{array}{cc} a & b \\ c & d \end{array}\right] \in \bF^{2\times 2} $
  $A = \bigl[ \begin{smallmatrix} a & b \\ c & d \end{smallmatrix} \bigr] \in \bF^{2\times 2}$
  is nonsingular.
  The matrix polynomial
  \begin{equation} \label{eqn.Mobius-defn}
    \MT_A(P)(\la) \,:=\; (c\la + d)^g P\left(\frac{a\la + b}{c\la + d}\right)
  \end{equation}
  is the \emph{M\"{o}bius transformation} of $ P(\la) $ with respect to $ A $.
\end{defi}

\begin{ejem} \label{ex.reversal-as-Mobius}
  Note that the grade $g$ reversal operation %$\rev_g P$
  is an example of a M\"obius transformation,
  specifically $\,\rev_g P = \MT_R(P)\,$ for
  $\,R = \bigl[ \begin{smallmatrix} 0 & 1 \\ 1 & 0 \end{smallmatrix} \bigr]$.
\end{ejem}

\noindent
 For conceptual clarity,
it is important to keep in mind that for any fixed field $\bF$ and grade $g$,
the formula in~\eqref{eqn.Mobius-defn}
actually defines a whole \emph{family} of transformations,
one for each matrix size $\mbyn$.
Once one has fixed the underlying field $\bF$, grade $g$,
input matrix size $\mbyn$, and matrix $A$,
then it is known~\cite{Mobius}
that the M\"obius transformation $\MT_A$ defines a \emph{bijection}
on the set of all $\mbyn$, grade $g$ matrix polynomials over $\bF$.
However, even by staying within the confines of one of these bijections,
the transformation $\MT_A(P)$ does \emph{not} in general preserve the degree of the input $P$;
degree may increase, decrease, or remain unchanged~\cite{Mobius}.
This lack of degree preservation happens even for the simplest size matrix polynomials,
i.e., for scalar ($1 \times 1$) polynomials.
However, for scalar polynomials the following lemma describes
one simple scenario, important for this paper,
where degree preservation by M\"obius transformations is guaranteed.
The proof of this result will be postponed
until Section~\ref{subsect.Mobius-and-irreducible divisors},
where a number of results about the properties of M\"obius transformations
and their effect on partial multiplicity sequences will be developed.

\begin{lemma} \label{lem.degree-preservation-without-proof}
  Suppose $\bF$ is an arbitrary field,
  and $\chi(\la)$ is any $\bF$-irreducible scalar polynomial with $\deg \chi \ge 2$.
  Let $\MT_A$ be the M\"obius transformation
  associated with any
  $A = \bigl[ \begin{smallmatrix} a & b \\ c & d \end{smallmatrix} \bigr] \in GL(2, \bF)$.
  Then with $\grade(\chi)$ taken to be equal to $\deg(\chi)$,
  the transformation $\MT_A$ preserves both the degree and the $\bF$-irreducibility of $\chi$.
  That is, $\MT_A (\chi)$ is $\bF$-irreducible,
  and $\,\deg \bigl(\MT_A (\chi) \bigr) = \deg(\chi)$.
\end{lemma}

%%%%%%%%%%%%%%%%%%%%%%%%%%%%%%%%%%%%%%%%%%%%%%%%%%%%%%%%%%%%%%%%%%%%%%%%
%%%%%%%%%%%%%%%%%%%%%%%%%%%%%%%%%%%%%%%%%%%%%%%%%%%%%%%%%%%%%%%%%%%%%%%%
\section{Quasi-triangular realization of finite spectral data}
  \label{sect.realization-of-finite-spectral-data}

Our first result is the construction of a quasi-triangular realization
of a given list of finite spectral data,
with a choice of degree that is compatible with the index sum theorem.

 \begin{theorem}[Quasi-Triangular Realization: Strictly Regular Case]
   \label{thm.QTR-strictly-regular} \quad \\
 	Suppose a list of $m$ nonconstant monic polynomials $s_1(\lambda),\dots, s_m(\lambda)$
 	over an arbitrary field $\,\bF$ is given,
 	satisfying the divisibility chain condition
 	$s_1(\la) \,\vert\, s_2(\la) \,\vert \,\dotsb\, \vert\, s_m(\la)$.
 	Let $\sigma := \sum_{i=1}^m \deg \bigl( s_i(\lambda) \bigr)$,
 	and define $k$ to be the maximum degree among all of the $\bF$-irreducible factors
 	of the polynomials $s_i(\lambda)$ for $i = 1,\dotsc, m$.
 	Then for any choice of nonzero $d, n \in \bN$
 	such that $n \ge m$ and $dn = \sigma$,
 	there exists an $n\times n$, degree $d$, strictly regular matrix polynomial $Q(\lambda)$ over $\bF$
 	that is $k$-quasi-triangular,
 	and has exactly the given polynomials $s_1(\lambda),\dots, s_m(\lambda)$
 	as its nontrivial invariant polynomials,
 	together with $n-m$ trivial invariant polynomials. % equal to 1.
 	In addition, $Q(\lambda)$ can always be chosen so that the degree
 	of every entry in any off-diagonal block of $Q(\lambda)$
 	is strictly less than $d$.
 \end{theorem}

\noindent
Note that the two conditions in this theorem,
i.e., that $n \ge m$ and $dn = \sigma$,
are both necessary conditions for \emph{any} strictly regular realization
of the given data,
whether that realization is quasi-triangular or not.
The restriction $n \ge m$ simply says that the realization (and its Smith form)
needs to be big enough to accommodate all of the nontrivial invariant polynomials.
The condition $dn = \sigma$ is simply the Index Sum Theorem
in the form \eqref{ISMeq-StrictlyReg} required of any strictly regular realization
of data with the given $\sigma$.

Now before embarking on %diving into
the extensive technical details
of the proof of Theorem~\ref{thm.QTR-strictly-regular},
which will occupy our attention for the rest %(\tcb{remaining $15$ pages})
of Section~\ref{sect.realization-of-finite-spectral-data},
it will be helpful to give a brief idea of the overall strategy
of the argument.
The first step in our process of quasi-triangular realization
is to construct the Smith form corresponding to the given data
and the choice of $n$, i.e.,
\[
S(\la) =  I_{(n-m)\times(n-m)} \oplus \diag\{s_1(\la),s_2(\la),\dotsc,s_m(\la)\} \,.
\]
We take this as our starting point,
and begin changing $S(\la)$ by unimodular transformations,
``un-diagonalizing'' it and
slowly turning it into the desired quasi-triangularization.
The first phase of this un-diagonalization of the Smith form
aims to shift irreducible factors around on the diagonal,
in such a way as to try to make the degrees
of all of the diagonal entries
as close to  the target degree $d$ as possible. %equal as possible.
This is a proof idea %first
pioneered in~\cite{matrixpolynomials},
and developed further in~\cite{TasTisZab}.
Although the diagonal form is sacrificed, in this first stage
at least upper triangularity is maintained.

Now for some collections of spectral data,
this phase may succeed in making the diagonal entries
all have exactly the target degree $d$;
in this case a triangularization is obtained.
As shown in \cite{TasTisZab}, this can always be achieved
when the underlying field is algebraically closed.
However, for arbitrary fields we have to be satisfied with something less.
The best that can be achieved in general is to rearrange the irreducible factors
along the diagonal so that no two diagonal entries differ in degree by more than $k$
(the maximum degree among all of the irreducible divisors of the given spectral data),
and so that the vector of diagonal entries can be grouped
into contiguous blocks $D_1, \dotsc, D_{\ell}$ for some $\ell \le n$, each of size at most $k$,
where the \emph{average} degree of the entries in each $D_j$
is exactly the target degree $d$.

In the second phase of construction,
the upper triangular matrix attained so far
is partitioned into blocks,
with square (upper triangular) blocks down the diagonal
that correspond to the contiguous blocks $D_j$
of diagonal entries just mentioned.
Each of these diagonal blocks is now ``un-triangularized''
so that it has exactly degree $d$.
We now have a quasi-triangular realization
in which all of the diagonal blocks have degree $d$,
but the off-diagonal blocks have been left uncontrolled,
and may still have degree larger than $d$.

The final step, then, is to visit each of the off-diagonal blocks
in a ``sweep by super-diagonals pattern'',
using matrix polynomial division
to force the degree of all of the off-diagonal blocks to be strictly less than $d$.
This completes the construction of a $k$-quasi-triangular realization of degree $d$
for the given spectral data.

% Next we rearrange the irreducible factors of the diagonal entries in such a way
% that the new diagonal can be partitioned into blocks of sizes $n_j\leq k$
% whose average degree $ d $.
% We achieve this rearrangement via special $2 \times 2$ unimodular transformations,
% but in doing so, we lose diagonality for triangularity.
% Next, the diagonal blocks will be constructed
% so that they all have degree $d$,
% and then they will be divided into the off-diagonal blocks
% to ensure that the overall matrix has degree $d$.

\begin{ejem}\label{ex.intro}
Throughout this paper, we will keep returning to %build from
a single illustrative example,
continually developing it as we go. %along.
To start, consider the finite spectral data over the two-element field $\bF = \bZ_2$,
consisting of the invariant polynomials
\[
  s_1(\la) = \phi\psi \,, \quad s_2(\la) = \chi\phi\psi \,, \quad
  s_3(\la) = \chi^2\phi\psi^2 \,, \quad s_4(\la) = \chi^3\phi\psi^2 \,,
  \;\text{ and }\; s_5(\la) = \chi^3\phi^3\psi^4 \,,
\]
where $ \chi(\la) = \la^4 + \la^3 + 1 $, $ \phi(\la) = \la^2 + \la + 1 $,
and $ \psi(\la) = \la $ are $\bF$-irreducible.
The sum of the degrees of these invariant polynomials is $\sigma = 60$,
so we may choose the target degree to be $ d = 10 $.
As a consequence, the size of the realization must be $6 \times6$,
and so the Smith form for our given data is
\[
  S(\la) = \diag\{1,s_1(\la),s_2(\la),s_3(\la),s_4(\la),s_5(\la)\} \,.
\]
% \tcb{$\bigstar$ Clearly $\phi$ and $\psi$ are $\bF$-irreducible, but what about $\chi$?
%      IS it irreducible, and is it easy to see why?}
The irreducible divisors together with their associated partial multiplicity sequences are
\begin{align*}
    \PMg(\chi) & \,=\, (0,0,1,2,3,3)\\
    \PMg(\phi) & \,=\, (0,1,1,1,1,3) \\
    \PMg(\psi) & \,=\, (0,1,1,2,2,4) \,,
\end{align*}
which provides an alternative way to present the given finite spectral data.
\end{ejem}

%%%%%%%%%%%%%%%%%%%%%%%%%%%%%%%%%%%%%%%%%%%%%%%%%%%%%%%%%%%%%%%%%%%%%%%%
\subsection{Coprime partitions, factor-counting vectors, and the unimodular transfer lemma}

In this section we develop the tools needed to implement the first phase
of the construction of a quasi-triangular realization
from given finite spectral data.
That is, we will see how it is possible
to rearrange irreducible factors along the diagonal
via unimodular equivalence,
while at the same time maintaining upper triangularity.
% A technique we will be using extensively in this section
% is transferring irreducible factors from one diagonal entry to another
% by unimodular equivalence,
However, before we present these tools,
some additional concepts and terminology will be required.

Let $\cM$ be the multiset %collection
of all of the $\mathbb{F}$-irreducible factors of a scalar polynomial $p(\lambda)$.
% A partition $\mathcal{C}=\mathcal{F}\sqcup\mathcal{G}$ of $ \cC $
A partition of $\cM$ into a disjoint union $\cM = \cF \sqcup \cG$
of multisets $\cF$ and $\cG$
% into the disjoint union of $ \cF $ and $ \cG $
is called a \emph{coprime partition} %if $ f(\la) $ and $ g(\la) $ are coprime
% for all $ f\in \cF $ and $ g \in \cG $.
if every $f \in \cF$ is coprime to every $g \in \cG$.
% This reduces to no irreducible factor showing up in both $ \cF $ and $ \cG $.
This is equivalent to saying that there is no irreducible factor of $p$
that appears in both $\cF$ and $\cG$.
Given such a coprime partition, we can uniquely factor $p(\lambda)$ into
\[
  p(\lambda) \;=\; p_{\scF}(\lambda)\cdot p_{\scG}(\lambda) \,,
\]
where $p_{\scF}(\lambda)$ denotes the product
of all of the $\mathbb{F}$-irreducible factors in $p(\la)$ from $\mathcal{F}$,
and $p_{\scG}(\lambda)$ denotes the product
of all of the $\mathbb{F}$-irreducible factors in $p(\la)$ from $\mathcal{G}$.
We also denote by $|p(\la)|_{\scF}$ %\,(resp., $|p(\la)|_\mathcal{G}$)
the total \emph{number} of $\mathbb{F}$-irreducible factors
from $\mathcal{F}$ %\,(resp., from $\mathcal{G}$)\,
in $p(\la)$, with a similar meaning for $|p(\la)|_{\scG}$.
Note that $|p_{\scF}(\la)|_{\scF} = |p(\la)|_{\scF}$.

If $\mathbf{v}(\la)$ is a polynomial $n$-vector
and $\cF \sqcup \cG$ is a coprime partition of the multiset
of all of the $\bF$-irreducible factors of all of the entries of $\mathbf{v} (\la)$,
then the integer vector
\[
   |\mathbf{v}(\la)|_{\scF} \,:=\; \bigl( |v_1 (\la)|_{\scF}, \dotsc, |v_n (\la)|_{\scF} \bigr)
\]
is the \emph{factor-counting vector}
of $\mathbf{v}(\la)$ with respect to $\cF$.
Also useful is the integer \emph{degree vector}
\[
  \deg \mathbf{v} :=\,\bigl( \deg v_1(\la), \dotsc , \deg v_n(\la) \bigr) \,.
\]
 Finally, given an $\nbyn$ matrix polynomial $P(\la)$,
% $P(\lambda)=[a_{ij}(\lambda)]_{\nbyn}$, %1\leq i, j\leq n$
its main diagonal vector $\mathbf{p} (\la) := \diag P(\la)$,
and a coprime partition $\mathcal{F} \sqcup \mathcal{G}$
of the multiset $\mathcal{M}$ of all of the $\mathbb{F}$-irreducible factors
% of \emph{all of the diagonal entries} of $P(\la)$,
of the entries of $\mathbf{p} (\la)$,
we define the \emph{diagonal factor-counting vector} of $P(\lambda)$
with respect to $\mathcal{F}$ to be the integer vector
\begin{equation} \label{eq-diag-count}
  \mathbf{d}_{\scF}(P) \,:=\; %\bigl( |a_{11}(\lambda)|_{\scF},\dots,|a_{nn}(\lambda)|_{\scF} \bigr) \,.
                  |\mathbf{p}(\la)|_{\scF} \,.
\end{equation}

With these concepts, terminology, and notation in hand,
we can now describe and develop the tool for transferring irreducible factors
along the diagonal of an upper triangular matrix polynomial.

\begin{lema}[$2\times 2$ Unimodular Transfer Lemma]
	\label{LemUniTransf}
	Let
	\[
	  T(\lambda) \,=\,
	  \left[
	  \begin{array}{cc}
	    p(\lambda) & q(\lambda)\\
	    0 & r(\lambda)
	  \end{array}
	  \right]
	\]
	be a regular $2\times 2$ upper triangular matrix polynomial
	over an arbitrary field \,$\mathbb{F}$.
	Let $\mathcal{F}\sqcup\mathcal{G}$ be any coprime partition
	of the multiset of all of the $\mathbb{F}$-irreducible factors
	in the product $p(\lambda)r(\lambda)$. %including repetitions.
	Let $m=|p|_{\scF}$ and $ n=|r|_{\scF}$ so that $\mathbf{d}_{\scF}(T)=(m,n)$.
	Then for any $\alpha,\beta\in\mathbb{N}$
	such that $\alpha+\beta=m+n$ and $\min\{m,n\}\leq\alpha,\beta\leq\max\{m,n\}$,
	there exists a regular upper triangular matrix polynomial $\widetilde T(\lambda)$
	of the form
	\begin{equation}
	\label{eqn.wt-T}
	\widetilde T(\lambda)=\left[
	\begin{array}{cc}
	\widetilde p_{\scF}(\lambda)\cdot p_{\scG}(\lambda) & \widetilde q(\lambda)\\
	0 & \widetilde r_{\scF}(\lambda)\cdot r_{\scG}(\lambda)
	\end{array}
	\right]
	\end{equation}
	with $\mathbf{d}_{\scF}(\widetilde T)=(\alpha,\beta)$,
	such that $T(\la) $ is unimodularly equivalent to $\widetilde T(\la)$.
\end{lema}

\begin{proof}
	If $m=n$ there is nothing to do, so assume that $m \ne n$.
	Let $g(\la) := \gcd \bigl\{ p(\la), q(\la), r(\la) \bigr\}$,
	and factor
	\begin{equation} \label{eqn.factor-p-r}
	p(\la) = g(\la) \cdot \wh{p}(\la) \und r(\la) = g(\la) \cdot \wh{r}(\la) \,.
	\end{equation}
	Then the Smith form of $T(\la)$
	is $\diag \{ g(\la), g(\la) \,\wh{p}(\la) \,\wh{r}(\la) \}$,
	so any $\wt{T}$ as in \eqref{eqn.wt-T} that we construct that has this Smith form
	will be unimodularly equivalent to $T$.
	Here is how to construct many such $\wt{T}$'s.
		
	Begin by refining the factorizations of $p$ and $r$ in \eqref{eqn.factor-p-r},
	in a way that is compatible with the given coprime partition $\cF \sqcup \cG$:
	\begin{align}
	p(\la) = g(\la) \cdot \wh{p}(\la)
	&= \bigl( g_{\scF}(\la) g_{\scG}(\la) \bigr)
	\cdot \bigl(\wh{p}_{\scF}(\la) \wh{p}_{\scG}(\la) \bigr) \nonumber \\
	&= \bigl( g_{\scF}(\la) \wh{p}_{\scF}(\la) \bigr)
	\cdot \bigl(g_{\scG}(\la) \wh{p}_{\scG}(\la) \bigr)
	= p_{\scF}(\la) \cdot p_{\scG}(\la)
	\end{align}
	and
	\begin{align}
	r(\la) = g(\la) \cdot \wh{r}(\la)
	&= \bigl( g_{\scF}(\la) g_{\scG}(\la) \bigr)
	\cdot \bigl(\wh{r}_{\scF}(\la) \wh{r}_{\scG}(\la) \bigr) \nonumber \\
	&= \bigl( g_{\scF}(\la) \wh{r}_{\scF}(\la) \bigr)
	\cdot \bigl(g_{\scG}(\la) \wh{r}_{\scG}(\la) \bigr)
	= r_{\scF}(\la) \cdot r_{\scG}(\la).
	\end{align}
		
	Since we wish to leave $p_{\scG}(\la)$ and $r_{\scG}(\la)$
	completely undisturbed in going from $T$ to $\wt{T}$,
	and also to have $g(\la)$ present in both diagonal entries of $\wt{T}$
	in order to preserve the Smith form,
	this means that the only room for maneuvering
	is with the factors in $\wh{p}_{\scF}(\la)$ and $\wh{r}_{\scF}(\la)$.
	So let $a(\la)$ and $b(\la)$ be \emph{any} two polynomials
	(including possibly $a\equiv 1$ or $b\equiv 1$) over $\bF$
	such that
	\begin{equation} \label{eqn.ab}
	a(\la) \cdot b(\la) \,=\, \wh{p}_{\scF}(\la) \cdot \wh{r}_{\scF}(\la) \,,
	\end{equation}
	and consider the polynomial matrix
% 	\begin{align} \label{eqn.wt-T2}
% 	  \wt{T}(\la) &\,=\, \mat{cc} g_{\scF}(\la) a(\la) \cdot p_{\scG}(\la) & g(\la) \\
% 	            0 & g_{\scF}(\la) b(\la) \cdot r_{\scG}(\la) \rix \\
% 	              &\,=\, \mat{cc} g(\la) \cdot a(\la) \wh{p}_{\scG}(\la) & g(\la) \\
% 	            0 & g(\la) \cdot b(\la) \wh{r}_{\scG}(\la) \rix \,, \nonumber
% 	\end{align}
	\begin{equation} \label{eqn.wt-T2}
	  \wt{T}(\la) \,=\, \mat{cc} g_{\scF}(\la) a(\la) \cdot p_{\scG}(\la) & g(\la) \\
	                                     0 & g_{\scF}(\la) b(\la) \cdot r_{\scG}(\la) \rix
	              \,=\, \mat{cc} g(\la) \cdot a(\la) \wh{p}_{\scG}(\la) & g(\la) \\
	                               0 & g(\la) \cdot b(\la) \wh{r}_{\scG}(\la) \rix \,, %\nonumber
	\end{equation}
	of the form in \eqref{eqn.wt-T}
	with $\wt{p}_{\scF}(\la) = g_{\scF}(\la) a(\la)$, $\wt{q}(\la) = g(\la)$,
	and $\wt{r}_{\scF}(\la) = g_{\scF}(\la) b(\la)$.
	Then the $\gcd$ of the entries of $\wt{T}$ is easily seen to be $g(\la)$,
	so the Smith form of $\wt{T}(\la)$ is
	\begin{align*}
	\diag \left\lbrace g(\la),
	\,g(\la) \bigl[ a(\la) b(\la) \bigr]
	\!\cdot \bigl[ \wh{p}_{\scG}(\la) \wh{r}_{\scG}(\la) \bigr] \right\rbrace
	&=\,
	\diag \left\lbrace g(\la),
	\,g(\la) \bigl[ \wh{p}_{\scF}(\la) \wh{r}_{\scF}(\la) \bigr]
	\!\cdot \bigl[ \wh{p}_{\scG}(\la) \wh{r}_{\scG}(\la) \bigr] \right\rbrace \\
	&=\,
	\diag \left\lbrace g(\la),
	\,g(\la) \wh{p}(\la) \wh{r}(\la) \right\rbrace \,,
	\end{align*}
	which is identical to the Smith form of $T(\la)$.
	Thus for any choice of $a(\la)$ and $b(\la)$ in \eqref{eqn.ab},
	we have $\wt{T} \sim T$.
	Letting $\delta = \absF{g(\la)}$, we see that
	the diagonal factor-counting vector $\mathbf{d}_{\scF}(\wt{T})$
	for $\wt{T}(\lambda)$ in \eqref{eqn.wt-T2}
	can be $(\alpha, \beta)$ for any $\alpha + \beta = m + n$
	with $\delta \le \alpha, \beta \le m + n - \delta$.
	Since
	\[
	\delta \,\le\, \min(m, n) \,\le\, \max(m, n) \,\le\, m + n - \delta \,,
	\]
	any $(\alpha, \beta)$ pair given in the statement of the lemma
	is always achievable for $\mathbf{d}_{\scF}(\wt{T})$.
\end{proof}

This lemma for $2 \times 2$ triangular matrices
can now be used on a triangular matrix of any size
to ``transfer" irreducible factors belonging to a family $\mathcal{F}$
between any two \emph{adjacent} diagonal entries,
while maintaining triangularity,
and without disturbing any of the factors
that belong to the complementary family $\mathcal{G}$.
This is done by embedding the $2 \times 2$ unimodular transformations
provided by Lemma~\ref{LemUniTransf} into larger identity matrices.

\begin{coro}\label{cor.unimod-trans}
Let $ P(\la) $ be a regular $ n\times n $ upper triangular polynomial matrix,
and let $\cM = \cF \sqcup \cG $ be a coprime partition
of the multiset of all irreducible factors
of the diagonal entries of $ P(\la) $. %including repetitions.
Consider any $2 \times 2$ principal submatrix of $P(\la)$
with \emph{adjacent} diagonal entries, i.e.,
\[
  T(\la) = \left[\begin{array}{cc} p_{ii}(\la) & p_{ij}(\la) \\
                                       0 & p_{jj}(\la) \end{array}\right]
    \quad \text{ with } \quad j = i+1 \,,
\]
and let $ \mathbf{d}_{\scF}(T)=(m,n) $.
Then for any integers $\alpha,\beta$ such that $\alpha+\beta=m+n$
and $\min\{m,n\}\leq\alpha,\beta\leq\max\{m,n\}$,
there exists a regular upper triangular matrix polynomial $\widetilde P(\lambda)$
such that     %\absF{g(\la)}
\begin{itemize}
  \item $\wt p_{ii}(\la) = [\wt p_{ii} (\la)]_{\scF} \cdot p_{\scG}(\la)$
        \,\,with\,\, $\absF{\wt p_{ii}(\la)} = \alpha$\,,
  \item $\wt p_{jj}(\la) = [\wt p_{jj} (\la)]_{\scF} \cdot p_{\scG}(\la)$
        \,\,with\,\, $\absF{\wt p_{ii}(\la)} = \beta$\,,
  \item $\wt p_{kk}(\la) = p_{kk}(\la)$ \,\,for all\,\, $k \ne i,j$\,,
%
%   $ \wt p_{(i+1)(i+1)}(\la) = [\wt p_{(i+1)(i+1)}]_{\scF}(\la) \cdot p_{\scG}(\la) $
%         with $ |\wt p_{(i+1)(i+1)}| = \beta $,
  \item $\wt P(\la)$ is unimodularly equivalent to $P(\la)$.
\end{itemize}
\end{coro}
\begin{proof}
Apply Lemma \ref{LemUniTransf} to the submatrix $ T $
to get two $2 \times 2$ unimodular matrices $ E(\la) $ and $ F(\la) $
such that
\[
  E(\la)T(\la)F(\la) =
     \left[\begin{array}{cc} [\wt p_{ii} (\la)]_{\scF} \cdot p_{\scG}(\la) & \wt p_{ij}(\la) \\
                                0 & [\wt p_{jj} (\la)]_{\scF} \cdot p_{\scG}(\la) \end{array}\right]
\]
with diagonal $ \cF $-factor-counting vector $ (\alpha,\beta) $.
Then construct the $ n\times n $ unimodular matrices
% \[
%   \wh E(\la) =
%      \left[\begin{array}{c|cc|c} I_{i-1} &&& \\
%               \hline \\[-12pt]
%               & e_{11}(\la) & e_{12}(\la) & \\
%               & e_{21}(\la) & e_{22}(\la) & \\
%               \hline \\[-12pt] &&& I_{n-i-1}
%            \end{array}\right]
%    \quad\text{ and }\quad
%   \wh F(\la) =
%      \left[\begin{array}{c|cc|c} I_{i-1} &&& \\
%               \hline \\[-12pt]
%               & f_{11}(\la) & f_{12}(\la) & \\
%               & f_{21}(\la) & f_{22}(\la) & \\
%               \hline \\[-12pt]
%               &&& I_{n-i-1}
%            \end{array}\right].
% \]
\[
  \wh E(\la) \,=\, \mat{c|c|c}
                       I_{i-1} && \\
                    \hline \\[-12pt]
                          & \mystrut{4.1mm} E(\la) & \\
                    \hline \\[-12pt]
                          && \mystrut{4.1mm} I_{n-i-1}
                   \rix
   \quad\text{ and }\quad
  \wh F(\la) \,=\, \mat{c|c|c}
                       I_{i-1} && \\
                    \hline \\[-12pt]
                          & \mystrut{4.1mm} F(\la) & \\
                    \hline \\[-12pt]
                          && \mystrut{4.1mm} I_{n-i-1}
                   \rix \,.
\]
Transforming $P(\la)$ using these two unimodular matrices,
we obtain the desired matrix $\wt P(\la) = \wh E(\la)P(\la)\wh F(\la)$.
\end{proof}

\begin{ejem}\label{ex.factor-counting}
Consider the Smith form $ S(\la) $ from Example \ref{ex.intro}.
The multiset $\cM$ of all of the irreducible factors in the entries of $S(\la)$
contains many copies of $\chi$, $\phi$, and $\psi$,
but we can partition it into %sublists of irreducible factors
$\cM = \cF_1 \sqcup \cF_2 \sqcup \cF_3 \sqcup \cF_4$,
where $\cF_j$ contains all of the irreducible factors of degree $j$.
Thus we see that $ \cF_1 $ contains all of the copies of $ \psi $,
$ \cF_2 $ contains all of the copies of $ \phi $, $ \cF_3 $ is empty,
and $ \cF_4 $ contains all of the copies of $ \chi $.
The diagonal factor-counting vectors of $ S(\la) $ then are
% \begin{align*}
% \mathbf{d}_{\scFo}(S) & \;=\; (0,1,1,2,2,4) \,, \\
% \mathbf{d}_{\scFt}(S) & \;=\; (0,1,1,1,1,3) \,, \\
% \mathbf{d}_{\scFth}(S) & \;=\; (0,0,0,0,0,0) \,, \\
% % \text{and }\;
% \mathbf{d}_{\scFf}(S) & \;=\; (0,0,1,2,3,3) \,.
% \end{align*}
\begin{align*}
  \mathbf{d}_{\scFo}(S) & \;=\; (0,1,1,2,2,4) \,,
              \hspace*{12mm} \mathbf{d}_{\scFth}(S) \;=\; (0,0,0,0,0,0) \,, \\
  \mathbf{d}_{\scFt}(S) & \;=\; (0,1,1,1,1,3) \,,
              \hspace*{12mm} \mathbf{d}_{\scFf}(S) \;=\; (0,0,1,2,3,3) \,.
\end{align*}
Note that we include $ \mathbf{d}_{\cF_3}(S) $ here not only for the sake of completeness,
but also to show that %it is okay for one of the partition multisets to be empty.
having any of the partition multisets be empty is allowed.
Later on, we will use Corollary \ref{cor.unimod-trans}
to rearrange the irreducible factors along the diagonal
so that these diagonal factor-counting vectors
will have the property of being \emph{1-homogeneous},
a concept to be defined in the next section.

Note that these diagonal factor-counting vectors just happen
to match up with the partial multiplicity sequences in this example.
But this is not typical,
and follows from there being at most one irreducible divisor
in each piece of the given coprime partition of $\cM$.
 For general coprime partitions, diagonal factor-counting vectors of Smith forms
will each be a \emph{sum} of partial multiplicity sequences.
\end{ejem}

%%%%%%%%%%%%%%%%%%%%%%%%%%%%%%%%%%%%%%%%%%%%%%%%%%%%%%%%%%%%%%%%%%%%%%%%%%%%%%%%%%%%%%
%%%%%%%%%%%%%%%%%%%%%%%%%%%%%%%%%%%%%%%%%%%%%%%%%%%%%%%%%%%%%%%%%%%%%%%%%%%%%%%%%%%%%%%
\subsection{Homogenization of natural vectors and un-diagonalizing the Smith form}
   \label{subsect.homogenization}

With the ability to transfer irreducible factors along the diagonal
(Corollary~\ref{cor.unimod-trans}) now in our tool box,
it is time to see how to employ that tool to rearrange the diagonal irreducible factors
so as to make the new diagonal entries as close in degree to each other as possible.
This is phase $1$ of our quasi-triangular realization construction.
It will be shown  % We will see
that the best that can be done in general
is to make these diagonal entry degrees differ by at most $k$,
where $k$ is the highest degree among all of the irreducible factors along the diagonal.
To facilitate the discussion of this process,
we introduce the following two concepts.
% Since we will be working extensively with
Note that vectors whose entries are natural numbers,
in particular diagonal factor-counting vectors,
appear frequently in this discussion;
% we will refer to
such vectors will be referred in brief as \emph{natural vectors}.
Keep in mind that in this paper the natural numbers $\bN$ \emph{includes zero}.

\begin{defi}
	A natural vector $\mathbf{v}=(v_1,\dots,v_r)\in\mathbb{N}^r$ is \emph{$k$-homogeneous}
	if $|v_i-v_j|\leq k$ for any $1\leq i,j\leq r$.
\end{defi}

\begin{defi}
\label{Def-hom-ver}
	Let $\mathbf{v}=(v_1,\dots, v_r)\in\mathbb{N}^r$
	with component sum $s = v_1 + v_2 + \dotsb + v_r$,
	and divide $s$ by $r$ to get $s = qr + t$, with $0 \le t < r$.
	Then any permutation of the $1$-homogeneous vector
	$$(q, q, \dotsc, q, \underbrace{q+1, q+1, \dotsc, q+1}_{t \,\text{ copies}}) \,\in\, \bN^r$$
	is called a \emph{homogenization} of $\mathbf{v}$.
\end{defi}

Before addressing our primary objective, that is,
the phase $1$ rearrangement of diagonal irreducible factors,
it will be useful to do a preliminary examination
of the process of homogenizing natural vectors
via two very simple operations
that we will refer to as ``interchange'' and ``compression''.
We will see that these two operations on natural vectors
are closely related to transformations of the diagonal of upper triangular polynomial matrices
achievable by the Unimodular Transfer Corollary~\ref{cor.unimod-trans}.
The operations of interchange and compression also bring us
into contact with the classical notion of majorization of vectors,
which we recall next.

\begin{defi}[Majorization~\cite{Marshall}] \label{def.majorization} \quad \\
  For vectors $\mathbf{x}, \mathbf{y} \in \bR^n$,
  let $ \mathbf{x}^{\prime} $ and $ \mathbf{y}^{\prime} $
  denote the permutations of those vectors
  in which the entries have been arranged in decreasing order.
  We say that \emph{$\mathbf{x}$ majorizes $\mathbf{y}$},
  or \emph{$ \mathbf{y} $ is majorized by $ \mathbf{x} $},
  and write $\, \mathbf{x} \succeq \mathbf{y} \,$, if
\begin{align}\label{eqn.majorization}
  \sum_{i=1}^{\ell} x_i^{\prime} \;\geq\; \sum_{i=1}^{\ell} y_i^{\prime}
    \quad \text{ for } \quad \ell = 1,2,\dotsc, n \;,
\end{align}
with equality when $ \ell = n $.
Clearly this definition also applies without change when restricted
to integer vectors $\mathbf{x}, \mathbf{y} \in \bZ^n$,
or even to natural vectors $\mathbf{x}, \mathbf{y} \in \bN^n$.
\end{defi}

\begin{remark} \label{rem.spread-out}
\rm
  It is useful to keep in mind the basic intuition about majorization,
  i.e., that if $\mathbf{x}$ majorizes $\mathbf{y}$,
  then the entries of $\mathbf{x}$ are more ``spread out''
  than those of $\mathbf{y}$.
  This can be seen, at least in part,
  by examining the two extreme components of the ordered vectors
  $\mathbf{x}'$ and $\mathbf{y}'$.
  From the $\ell = 1$ inequality in \eqref{eqn.majorization}
  we have $x^{\prime}_1 \ge y^{\prime}_1$,
  and by combining the $\ell = n$ and $\ell = n-1$ parts of \eqref{eqn.majorization}
  we see that $y^{\prime}_n \ge x^{\prime}_n$.
  Thus we have $x^{\prime}_1 \ge y^{\prime}_1 \ge \dotsb \ge y^{\prime}_n \ge x^{\prime}_n$,
  displaying some of the greater dispersion of the entries of $\mathbf{x}'$.
\end{remark}
It is a classical result of the study of majorization
(attributed in~\cite{Marshall} to a 1903 article of Muirhead \cite{Muirhead})
that for natural vectors $\mathbf{x}$ and $\mathbf{y}$,
$\mathbf{x}$ majorizes $\mathbf{y}$ $(\mathbf{x} \succeq \mathbf{y})$
if and only if
$\,\mathbf{x}$ can be transformed into $\mathbf{y}$
by a finite sequence of operations known
variously as \emph{transfers}, \emph{Dalton transfers},
or more recently (and colorfully) as \emph{Robin Hood transfers}
-- think ``rob from the rich and give to the poor''.
Such an operation takes any two components of a natural vector $\mathbf{v}$,
say $v_i$ and $v_j$, and replaces them by natural numbers $\alpha$ and $\beta$
that are closer together in size;
more specifically, where $\alpha + \beta = v_i + v_j$
and $|\alpha - \beta| \le |v_i - v_j|$.
(Note that in his proof of this result,
 Muirhead uses only transfers in which $v_i$ and $v_j$
 change by exactly $1$, i.e., $\alpha = v_i \pm 1$ and $\beta = v_j \mp 1$.)

 For our purposes, we need to have a Muirhead-like result
for converting a natural vector into a homogenization of itself.
However, since we will ultimately need to implement these conversions
by unimodular transformations that preserve upper triangularity
(see Corollary~\ref{cor.unimod-trans}),
it is essential that we limit our operations on natural vectors
to ones that act only on \emph{adjacent entries} of a vector,
which is more restrictive than the Robin Hood transfers used by Muirhead.
Thus we introduce the following two operations
acting on natural vectors:

\begin{itemize}
\item \textbf{Interchange} of \emph{adjacent} components:
  \[
    (v_1,\dots,v_i, v_{i+1}, \dots,v_r)
      \leadsto
    (v_1,\dots, v_{i+1}, v_i,\dots,v_r) \in \mathbb{N}^r
  \]
\item \textbf{Compression} of \emph{adjacent} components:
  \[
    (v_1,\dots, v_{i-1}, v_i, v_{i+1}, v_{i+2}, \dots,v_r)
      \leadsto
    (v_1,\dots, v_{i-1},\alpha, \beta, v_{i+2},\dots,v_r) \in \mathbb{N}^r \,,
  \]
where $\alpha + \beta = v_i + v_{i+1}$ and $|\alpha - \beta| < |v_i - v_{i+1}|$. \\
Equivalently, we can instead require that $\alpha + \beta = v_i + v_{i+1}$
and $\min\{v_i,v_{i+1}\} < \alpha,\beta < \max\{v_i,v_{i+1}\}$.
\end{itemize}

\noindent
Note that we use the new terminology ``interchange'' and ``compression'' here
for these more limited operations, to try to avoid any confusion
with the more flexible Robin Hood transfers.
We now state and prove the modified version
of the classic Muirhead result,
specialized to the transformation of a natural vector
into a homogenization of itself.
As an immediate consequence of Lemma~\ref{lem.conversion-to-homogen-version}
and the Muirhead theorem,
we see that any natural vector $\mathbf{v}$
majorizes any homogenization of itself.

\begin{lemma}[Homogenization Lemma]
	\label{lem.conversion-to-homogen-version} \quad \\
	Consider any $\mathbf{v}=(v_1,\dots, v_r)\in\mathbb{N}^r$,
	with component sum $s = v_1+\dots+v_r$, and average component value $a = s/r$.
	Let $q \in \bN$ be an integer such that the average $a$
	is contained in the closed interval $[@ q, q+1 @]$.
% 	Divide $s$ by $r$ to get $s=qr+t$, with $0 \le t < r$.
	Then by a finite sequence of interchanges and at most $r-1$ compressions,
	$\mathbf{v}$ can be transformed into a homogenization of itself,
	comprised of only $q$ and $q+1$ components, ordered arbitrarily.
	Writing $s=qr+t$ with $0 \le t \le r$,
	then in the homogenization of $\mathbf{v}$
	there are exactly $t$ copies of $q+1$
	and $r-t$ copies of $q$.
\end{lemma}

\begin{proof}
  The proof will proceed by a pair of inductions,
  each on the length $r$ of the vector $\mathbf{v}$.
  The first induction considers only vectors $\mathbf{v}$
  whose average component value $a$ is an integer, either $q$ or $q+1$.
  In this case we will see that the homogenization of $\mathbf{v}$
  is the vector $(a, a, \dotsc , a)$.
  The second induction, which builds on the result of the first,
  considers the vectors $\mathbf{v}$ where $a$ is \emph{not} an integer,
  i.e., $q < a < q+1$.
  For this second case the homogenization will have both $q$ and $q+1$ entries.

  \smallskip
  \noindent
  \underline{Part 1} ($a$ is an integer) \\
  \hspace*{4mm} The case $r=1$ is trivial, and the case $r=2$ is also very easy --
  just do a single ($r-1 = 1$) compression to produce the homogenization $(a, a)$.
  So now assume that every $k$-vector $\mathbf{v} \in \bN^k$ with $k < r$
  and integer average $a$
  can be finitely transformed into a homogenization $(a, a, \dotsc , a)$
  using at most $k-1$ compressions.
  Let $\mathbf{v} \in \bN^r$ be any $r$-vector with average value $a = q$.
  If $\mathbf{v}$ already has any component equal to $q$,
  then finitely many interchanges move this component to the end,
  producing $\mathbf{v} \leadsto (\wt{\mathbf{v}}, q)$.
  If not, then one can find at least one component larger than $q$
  and one component smaller than $q$.
  By interchanges make these components adjacent,
  and then a compression on these two components can produce at least one $q$ component.
  Finitely many more interchanges then moves this newly-produced $q$ to the end,
  again producing $\mathbf{v} \leadsto (\wt{\mathbf{v}}, q)$.
  In either case, the $(r-1)$-vector $\wt{\mathbf{v}}$
  has average component value $\wt{a} = (s - q) / (r-1) = (ar - a) / (r-1) = a = q$.
  By the induction hypothesis the vector $\wt{\mathbf{v}}$
  can now be transformed by finitely many interchanges and at most $r - 2$ compressions
  into the vector $(a, a, \dotsc , a)$,
  and thus we have the desired conclusion for the $r$-vector $\mathbf{v}$.
  The same argument works, mutatis mutandis, for vectors $\mathbf{v} \in \bN^r$
  with average value $a = q+1$.

  \smallskip
  \noindent
  \underline{Part 2} ($a$ is \emph{not} an integer, i.e., $q < a < q+1$ and $0 < t < r$) \\
  \hspace*{4mm} Here the base case for the induction cannot be $r=1$;
  it must be $r=2$, with $a = q + \frac{1}{@2@}$.
  In this $r=2$ case, then,
  a single compression can produce the vector $(q, q+1)$, as desired.
  Again as in Part~1, now assume that every $k$-vector $\mathbf{v} \in \bN^k$ with $k < r$
  and average $q < a < q+1$  %$a \in (q, q+1)$
  can be finitely transformed into a homogenized version containing only $q$ and $q+1$ components,
  using at most $k-1$ compressions.
  Let $\mathbf{v} \in \bN^r$ be any $r$-vector with non-integer average value $q < a < q+1$.
  If $\mathbf{v}$ already has any component equal to $q$ or $q+1$,
  then finitely many interchanges moves this component to the end,
  producing $\mathbf{v} \leadsto (\wt{\mathbf{v}}, q)$ or $(\wt{\mathbf{v}}, q+1)$.
  If no component of $\mathbf{v}$ is $q$ or $q+1$,
  then there must be at least one component larger than $q+1$
  and one component smaller than $q$.
  Make these components adjacent by interchanges,
  and compress these two components
  to produce at least one $q$ or $q+1$ entry
  (possibly both, or perhaps even two $q$'s or two $q+1$'s).
  Move this $q$ or $q+1$ entry to the end by interchanges,
  once again transforming $\mathbf{v}$
  into the form $(\wt{\mathbf{v}}, q)$ or $(\wt{\mathbf{v}}, q+1)$.
  The $(r-1)$-vector $\wt{\mathbf{v}}$ now has a component average
  \[ \wt{a} \,=\,
        \begin{cases}
          \frac{\,s - q\,}{r-1}
            \,=\, \frac{\,qr + t - q\,}{r-1}
            \,=\, q + \frac{t}{\,r-1\,} \,\le\, q+1
         & \text{\,if\; $\mathbf{v} \leadsto (\wt{\mathbf{v}}, q)$,} \\[5pt]
          \frac{\,s - (q+1)\,}{r-1}
            \,=\, \frac{\,qr + t - q - 1\,}{r-1}
            \,=\, q + \frac{t-1}{\,r-1\,} \,\ge\, q
         & \text{\,if\; $\mathbf{v} \leadsto (\wt{\mathbf{v}}, q+1)$.}
        \end{cases}
  \]
  If $\wt{a}$ is $q$ or $q+1$, then we are done by Part~1.
  Otherwise $q < \wt{a} < q+1$,
  and we are done by the induction hypothesis of Part~2.

\smallskip
  This completes the inductive argument
  that a homogenization of $\mathbf{v}$ can always be achieved.
  The expression of the component sum $s$ in the form $s=qr+t$ with $0 \le t \le r$
  then uniquely determines the number of $q$ and $q+1$ entries
  in the homogenization.
  A final sequence of interchanges
  can put the entries of the homogenization into any desired order,
  thus completing the proof of the lemma.
\end{proof}

% \begin{remark}
% \rm
% It should be noted that the operation on natural vectors
% that we call a ``compression" is a special case of an operation
% classically known as a \emph{Robin Hood transfer} \cite{Marshall}
% -- think ''rob from the rich and give to the poor''.
% One reason we do not use this classical terminology in this paper
% is because Robin Hood transfers are typically not restricted
% to adjacent components of the vector.
% Our reason for restricting to adjacent components
% is because we will be carrying out these compressions
% using the Corollary \ref{cor.unimod-trans},
% and adjacency is crucial to maintaining triangularity.
% \end{remark}

With the Homogenization Lemma in hand,
we return to the first phase of our quasi-triangular realization process,
the ``un-diagonalizing'' of a Smith form.
% redistributing irreducible factors along the diagonal
% so as to make the degrees of the diagonal entries differ by as little as possible.
The next result addresses the combinatorial essence of this problem,
showing how it is possible to take a multiset of irreducible polynomials
and distribute them among the entries of a vector
in a way that minimizes the degree differences,
and at the same time produces a viable configuration
for the diagonal vector of an upper triangular un-diagonalized Smith form.
The corollary immediately following
shows that the target diagonal produced by Lemma~\ref{Stacking}
is in fact reachable from the Smith form
via the type of triangularity-preserving unimodular transformations
developed in Corollary~\ref{cor.unimod-trans}.

% Recall the collection $ \cC $ of the irreducible factors
% that show up in the invariant polynomials (including repetitions),
% and let $ \cC = \cF_1 \sqcup \cF_2 \sqcup \dotsm \sqcup \cF_k $
% be a coprime partition such that
% all the irreducible factors of degree $ i $ are contained in $ \cF_i $.
% The next result shows that we can transfer irreducible factors
% from one invariant polynomial to another until
% \begin{enumerate}
% \item the polynomials differ in degree by at most $ k $,
% \item each vector $ \mathbf{d}_{\cF_i}(s_1), \mathbf{d}_{\cF_i}(s_2), \dotsc, \mathbf{d}_{\cF_i}(s_n) $
%       for $ i = 1,2,\dotsc,k $ is 1-homogeneous. \tcr{\underline{FIX ERROR. Just delete}?!}
% \end{enumerate}

% \begin{lema}
% \label{Stacking}
% 	Over an arbitrary field $\mathbb{F}$,
% 	let $\mathcal{C}$ be a finite list of irreducible polynomials
% 	(including potential repetitions) of degree less than or equal to $k$
% 	and let $n$ be a positive integer.
% 	Then, there exists an $n\times n$ diagonal matrix polynomial $D(\lambda)$
% 	whose diagonal entries are either a product of elements of $\mathcal{C}$ or 1,
% 	each factor of $\mathcal{C}$ is a factor of a diagonal entry of $D$
% 	and such that the difference in the degrees of any two diagonal entries
% 	is less than or equal to $k$.
% 	In addition, if we consider coprime partitions
% 	of $\mathcal{C}=\mathcal{F}_1\cup\dots\cup\mathcal{F}_k$,
% 	with $\mathcal{F}_i$ containing all the $\mathbb{F}$-irreducible degree $i$ factors
% 	of $\mathcal{C}$ for each $i=1,\dots,k$, $\mathbf{d}_{\mathcal{F}_i}(D)$
% 	is 1-homogeneous.
% \end{lema}

\begin{lema}
  \label{Stacking}
	Let $\,\bF$ be an arbitrary field,
	and consider a finite multiset $\cM$ of $\,\bF$-irreducible polynomials
	of degree less than or equal to $k$.
	Also let $n$ be any positive integer.
	Let $\cM = \cF_1 \sqcup \dotsb \sqcup \cF_k$
	be the coprime partition of $\cM$
	in which $\cF_j$ contains all of the $\,\bF$-irreducible factors in $\cM$ of degree $j$.	
	Then there exists a polynomial $n$-vector $\,\mathbf{p} (\la)$
	with all nonzero entries
	such that
	\begin{itemize}
	  \item the multiset of all of the $\,\bF$-irreducible factors
	        of all the entries of $\,\mathbf{p} (\la)$ is exactly $\cM$,
	  \item the degree vector
	        $\,\deg \mathbf{p} :=\,\bigl( \deg p_1(\la), \dotsc , \deg p_n(\la) \bigr)$
	        is $k$-homogeneous,
	  \item and each of the factor-counting vectors $|\mathbf{p}(\la)|_{\scFj}$
	        for $j = 1,\dotsc,k$ is $1$-homogeneous.
	\end{itemize}
\end{lema}
\begin{proof}
 For notational convenience, we first append
enough copies of the constant polynomial $1$ to the multiset $\cM$
so that the total number $m$ of elements in $\cM$
is a multiple of $n$, i.e., $|\cM| =: m = (q+1)n$ with $q \ge 0$.
The presence of these copies of $1$ will have no impact
on the vector $\mathbf{p}$ that is constructed,
but the exposition will be simplified by including them.

Next order the elements of $\cM$ into a list $\cL$
so that the sequence of degrees is decreasing.
That is, let $\cL = \bigl[ a_1(\la), a_2(\la), \dotsc, a_m(\la) \bigr]$,
where $\deg a_\ell \geq \deg a_{\ell+1}$ for every $\ell=1,2,\dotsc,m-1$.
Now partition the list $\cL$ into $q+1$ contiguous sublists of $n$ polynomials each,
% as in Figure~\ref{fig.stacking1} below.
and stack them up as in the following diagram. \,\,(Keep in mind that $qn+n = m$.)
	
\begin{figure}[h]\label{fig.stacking1}\centering
\begin{tikzpicture}
\node at (0,0) {$\begin{array}{cc@{\hspace{5mm}}l@{\hspace{5mm}}lll}
	\cL_0 & = & \;[\, a_1, & \;a_2, & \dotsc & \;\;,a_n ] \\
	\cL_1 & = & [\, a_{n+1}, & a_{n+2}, & \dotsc & \;\,,a_{2n}\,] \\
	& \vdots & & & & \\
	\cL_j & = & \![\, a_{jn + 1}, & a_{jn + 2}, & \dotsc & \,,a_{jn + n} \,] \\
	& \vdots & & & & \\
	\cL_q & = & \![\, a_{qn+1}, & a_{qn+2}, & \dotsc & ,a_{qn+n}\,]
	\end{array}$};
\draw[blue,thick] (-1.2,-0.2) ellipse (0.65cm and 2.6cm);
\node at (-1.2,-2) {$\color{red}v_1(\la)$};
\draw[blue,thick] (0.3,-0.2) ellipse (0.65cm and 2.6cm);
\node at (0.3,-2) {$\color{red}v_2(\la)$};
\node at (1.325,-2) {$\dotsc$};
\draw[blue,thick] (2.55,-0.2) ellipse (0.8cm and 2.6cm);
\node at (2.6,-2) {$\color{red}v_n(\la)$};
\end{tikzpicture}
% \caption{The list of irreducible factors is broken up
%      into $ q+1 $ groups of $ n $ factors.
%      Taking the product of the factors in each red circle
%      produces the diagonal entries $ d_i $ for $ i=1,2,\dotsc,n $.}
\end{figure}
% 	and define $ t_\ell = \deg a_{\ell n+1} - \deg a_{\ell n+n} $
% 	for each $ \ell=0,1,\dotsc, q $. \tcb{Meaning of $\ell$ is overloaded here! FIX!}

\noindent
The $n$ component polynomials of the desired vector $\mathbf{p} (\la)$
are formed by taking the column-wise products indicated by the blue ovals.
 For example,
$p_1(\la) := a_1(\la) a_{n+1}(\la) \dotsb a_{qn+1}(\la)$ \,,\,
$p_2(\la) := a_2(\la) a_{n+2}(\la) \dotsb a_{qn+2}(\la)$ \,, etc.
It is clear by construction that the multiset
of all of the $\,\bF$-irreducible factors
of all the entries of $\,\mathbf{p} (\la)$ is exactly $\cM$.
All that remains is to see why this $\mathbf{p} (\la)$
has the other two desired properties.

To see why the degree vector $\deg \mathbf{p}$ is $k$-homogeneous,
first observe that the component degrees are decreasing,
i.e., $\deg p_1 \ge \deg p_2 \ge \dotsb \ge \deg p_n$.
This follows immediately from the elements of the list $\cL$
being in decreasing order.
Consequently the largest degree difference in $\mathbf{p}$
will be between the first and last components $p_1(\la)$ and $p_n(\la)$.
But we have
\begin{align*}
    \deg p_1 - \deg p_n \; & = \; (\deg a_1 + \deg a_{n+1} + \dotsb + \deg a_{(qn+1)})
                                - (\deg a_n + \deg a_{2n} + \dotsb + \deg a_{(qn+n)}) \\
% 		                   & = \; (\deg a_1 - \deg a_n) + (\deg a_{n+1} - \deg a_{2n})
% 		                        + \dotsb + (\deg a_{(qn+1)} - \deg a_m) \\
		                   & = \; \deg a_1 - \bigl[(\deg a_n - \deg a_{n+1}) + (\deg a_{2n} - \deg a_{2n+1})
		                        + \dotsb
		                        + (\deg a_{qn} - \deg a_{qn+1}) + \deg a_m \bigr] \\
                           & \le \; k \,,
\end{align*}
since $\deg a_1(\la) \le k$,
and the expression inside the brackets is a sum of non-negative numbers
(due to the list $\cL$ being in decreasing-degree order).
Hence the degree vector $\deg \mathbf{p}(\la)$ is $k$-homogeneous.

 Finally, consider the factor-counting vectors $|\mathbf{p}(\la)|_{\scFj}$.
To see why these vectors are all $1$-homogeneous,
observe that because the irreducible polynomials are listed in $\cL$
in order of decreasing degree,
the degree $j$ irreducibles in $\cL$ form a contiguous segment of $\cL$,
and hence are distributed among consecutive sublists $\cL_{\delta}$
in a manner analogous to the green entries in the following diagram.
On the other hand, if $\alpha < \beta$,
then

\begin{figure}[h]\label{fig.stacking2}\centering
\begin{tikzpicture}
\node at (0,0) {$\begin{array}{cclllllll}
	& \vdots & & & & & & & \\
	\cL_{\alpha} & = & \{a_{\alpha n+1} \,, & \dotsc & \,, \boxed{\color{green}a_{\alpha n+i}} \,,
                     & \color{green}\dotsc & ,\,\color{green}a_{\alpha n+ \ell} \,,
                     & \color{green}\dotsc & ,\;\color{green}a_{\alpha n+n}\} \\
	\cL_{\alpha + 1} & = & \{\color{green}a_{(\alpha+1) n+1}, & \color{green}\dotsc
	                 & \,,\,\color{green}a_{(\alpha+1) n+i} , & \color{green}\dotsc
	                 & ,\color{green}a_{(\alpha+1) n+ \ell} , & \color{green}\dotsc
	                 & ,\;\color{green}a_{(\alpha+1) n+n}\} \\
	& \vdots & & & & & & & \\
	\cL_{\beta} & = & \{\color{green}a_{\beta n+1} \,,  & \color{green}\dotsc
	                & \,,\,\color{green}a_{\beta n+i} \,, & \color{green}\dotsc
	                & ,\boxed{\color{green}a_{\beta n+ \ell}} \,, & \dotsc & ,\;a_{\beta n+n}\} \\
	& \vdots & & & & & & & \\
	\end{array}$};
\draw[blue,thick] (-3.6,0) ellipse (0.95cm and 2.5cm);
\node at (-3.6,-2) {$\color{red}v_1$};
\node at (-2.05,-2) {$\dotsc$};
\draw[blue,thick] (-0.65,0) ellipse (0.95cm and 2.5cm);
\node at (-0.65,-2) {$\color{red}v_i$};
\node at (0.7,-2) {$\dotsc$};
\draw[blue,thick] (2.2,0) ellipse (0.95cm and 2.5cm);
\node at (2.2,-2) {$\color{red}v_{\ell}$};
\node at (3.4,-2) {$\dotsc$};
\draw[blue,thick] (5.1,0) ellipse (0.95cm and 2.5cm);
\node at (5.1,-2) {$\color{red}v_n$};
\end{tikzpicture}
% \caption{The irreducible factors of degree $ i $ are highlighted in green.
%          Taking the product of the factors in each red circle
%          produces the diagonal entries as before.}
\end{figure}

\noindent
The first instance of a degree $j$ irreducible
is the boxed entry $a_{\alpha n+i}$ with $1 \le i \le n$ in some sublist $\cL_{\alpha}$,
and the last degree $j$ irreducible is the boxed entry $a_{\beta n + \ell}$,
where $\alpha \le \beta$.

If $\alpha = \beta$, then $i \le \ell$,
and all of the degree $j$ irreducibles in $\cL$ are in the sublist $\cL_{\alpha}$,
so the factor-counting vector $|\mathbf{p}(\la)|_{\scFj}$
has only $0$ and $1$ entries;
this certainly constitutes a $1$-homogeneous vector.
On the other hand,
if $\alpha < \beta$ then we may have either $\ell < i-1$, $\ell = i-1$, or $\ell > i-1$.
Consider each of these possibilities in turn:
\begin{itemize}
  \item \underline{$\ell < i-1$}:
        In this case the combined contribution of the degree $j$ entries
        in $\cL_{\alpha}$ and $\cL_{\beta}$ to $|\mathbf{p}(\la)|_{\scFj}$ is
        the $n$-vector
        \[
          (1, 1, \dotsc , \underbrace{1}_{\ell}, 0, \dotsc , 0, \underbrace{1}_{i}, \dotsc , 1) \,,
        \]
        while the contribution of the sublists $\cL_{\delta}$
        with $\alpha < \delta < \beta$
        is a constant vector with $n$ entries all equal to $\beta - \alpha - 1$.
        The sum of these two vectors is $|\mathbf{p}(\la)|_{\scFj}$,
        and is clearly $1$-homogeneous.
  \item \underline{$\ell = i-1$}:
        Now the contribution to $|\mathbf{p}(\la)|_{\scFj}$
        from $\cL_{\alpha}$ and $\cL_{\beta}$ combined is just
        the constant $n$-vector $(1, 1, \dotsc, 1)$,
        which together with the contribution from the sublists
        between $\cL_{\alpha}$ and $\cL_{\beta}$
        gives a constant vector for $|\mathbf{p}(\la)|_{\scFj}$,
        with entries all equal to $\beta - \alpha$.
        This is certainly $1$-homogeneous, indeed even $0$-homogeneous.
  \item \underline{$\ell > i-1$}:
        In this final case the combined contribution of the degree $j$ entries
        in $\cL_{\alpha}$ and $\cL_{\beta}$ to $|\mathbf{p}(\la)|_{\scFj}$ is
        the $n$-vector
        \[
          (1, \dotsc , 1, \underbrace{2}_i, 2, \dotsc , \underbrace{2}_{\ell}, 1, \dotsc , 1) \,.
        \]
        Together with the constant vector contribution
        from the sublists between $\cL_{\alpha}$ and $\cL_{\beta}$,
        we again see that $|\mathbf{p}(\la)|_{\scFj}$ is $1$-homogeneous.
\end{itemize}
\end{proof}

\begin{remark} \label{rem.k-homogeneity-is-best-possible}
\rm
  Note that $k$-homogeneity for $\mathbf{p} (\la)$ in Lemma~\ref{Stacking}
  is the best possible general result here,
  as illustrated by the following simple example.
  Suppose $\cM$ contains \emph{only} irreducible polynomials of degree $k$,
  say $\ell$ of them, and $n > \ell$.
  Then clearly the $n$-vector $\mathbf{p} (\la)$ that minimizes the degree differences
  has $|\mathbf{p}(\la)|_{\scFk} = (k, k, \dotsc , k, 0, \dotsc, 0)$,
  and $k$-homogeneity cannot be improved upon in this situation.
\end{remark}

We return to our running illustration of the results of this paper,
as begun earlier in Examples~\ref{ex.intro} and \ref{ex.factor-counting}.
The next example demonstrates the application of Lemma~\ref{Stacking}
to this data.

\begin{ejem}\label{ex.stacking}
Recall the $\,6 \times 6\,$ Smith form from Example \ref{ex.intro},
i.e.,
\[
  S(\la) = \diag\{1,\phi\psi,\chi\phi\psi,\chi^2\phi\psi^2,\chi^3\phi\psi^2,\chi^3\phi^3\psi^4\} \,,
\]
and consider the coprime partition of the list of irreducible divisors
as in Example \ref{ex.factor-counting},
with
% \begin{align*}
%   \cF_1 & = \{\underbrace{\tcg{\psi,\psi,\psi,\psi,\psi,\psi,\psi,\psi,\psi,\psi}}_{10}\} \\
%   \cF_2 & = \{\underbrace{\tcb{\phi,\phi,\phi,\phi,\phi,\phi,\phi}}_{7}\} \\
%   \cF_4 & = \{\underbrace{\tcr{\chi,\chi,\chi,\chi,\chi,\chi,\chi,\chi,\chi}}_{9}\}.
% \end{align*}
\[
  \cF_1 \,=\, \{\underbrace{\tcg{\psi,\psi,\psi,\psi,\psi,\psi,\psi,\psi,\psi,\psi}}_{10}\} \,,
  \quad \;
  \cF_2 \,=\, \{\underbrace{\tcb{\phi,\phi,\phi,\phi,\phi,\phi,\phi}}_{7}\} \,,
  \quad \;
  \cF_4 \,=\, \{\underbrace{\tcr{\chi,\chi,\chi,\chi,\chi,\chi,\chi,\chi,\chi}}_{9}\} \,,
\]
and an $\cF_3$ that is empty.
 Following the proof of Lemma \ref{Stacking} gives us five sublists,
each of length $6$:
\[\begin{array}{lclccccr}
  \cL_0 & = & \{\color{red}{\chi},&\color{red}{\chi},&\color{red}{\chi},
            &\color{red}{\chi},&\color{red}{\chi},&\color{red}{\chi}\} \\
  \cL_1 & = &\{\color{red}{\chi},&\color{red}{\chi},&\color{red}{\chi},
            &\color{blue}{\phi},&\color{blue}{\phi},&\color{blue}{\phi}\} \\
  \cL_2 & = &\{\color{blue}{\phi},&\color{blue}{\phi},&\color{blue}{\phi},
            &\color{blue}{\phi},&\color{green}{\psi},&\color{green}{\psi}\} \\
  \cL_3 & = &\{\color{green}{\psi},&\color{green}{\psi},&\color{green}{\psi},
            &\color{green}{\psi},&\color{green}{\psi},&\color{green}{\psi}\} \\
  \cL_4 & = &\{\color{green}{\psi},&\color{green}{\psi},&1,&1,&1,&1\} \,.
\end{array}\]
Then taking the products going down the columns,
we get $\mathbf{p} (\la) = \bigl( p_1(\la), \dotsc, p_6(\la) \bigr)$ with
% \[
% \begin{array}{lll}
%   v_1(\la) = \chi^2\phi\psi^2 & v_2(\la) = \chi^2\phi\psi^2 & v_3(\la) = \chi^2\phi\psi \\
%   v_4(\la) = \chi\phi^2\psi   & v_5(\la) = \chi\phi\psi^2   & v_6(\la) = \chi\phi\psi^2
% \end{array}.
% \]
\[
  \begin{array}{llllll}
    p_1(\la) = \chi^2\phi\psi^2 \,,  &\; p_2(\la) = \chi^2\phi\psi^2 \,,  &\; p_3(\la) = \chi^2\phi\psi \,,
    &\; p_4(\la) = \chi\phi^2\psi \,,  &\; p_5(\la) = \chi\phi\psi^2 \,,  &\; p_6(\la) = \chi\phi\psi^2 \,.
  \end{array}
\]
The vector $\mathbf{p} (\la)$ has factor-counting vectors
\begin{align*}
  |\mathbf{p}(\la)|_{\scFo} & = (2,2,1,1,2,2) \\
  |\mathbf{p}(\la)|_{\scFt} & = (1,1,1,2,1,1) \\
  |\mathbf{p}(\la)|_{\scFth} & = (0,0,0,0,0,0) \\
  |\mathbf{p}(\la)|_{\scFf} & = (2,2,2,1,1,1) \,,
\end{align*}
which are all $1$-homogeneous,
and a $4$-homogeneous degree vector $\deg \mathbf{p} = (12, 12, 11, 9, 8, 8)$,
just as guaranteed by Lemma~\ref{Stacking}.
\end{ejem}

We now have the tools needed to reach the next milestone
in our construction of a quasi-triangular realization
of given finite spectral data.
The following corollary is the main result to carry forward
into the next stages of this construction.

\begin{coro}[Un-diagonalizing the Smith form] \label{cor.un-diagonalization} \quad \\
  Let $\,\bF$ be an arbitrary field,
  and consider any regular $\nbyn$ diagonal polynomial matrix $S(\la)$ over $\,\bF$
  that is in Smith form.
  Let $\cM$ be the multiset of all of the $\,\bF$-irreducible factors
  of all of the invariant polynomials in $S(\la)$,
  and let $k$ be the maximum degree among all elements of $\cM$.
  Consider also the coprime partition $\cM = \cF_1 \sqcup \dotsb \sqcup \cF_k$,
  in which each $\cF_j$ contains all of the $\,\bF$-irreducible factors in $\cM$ of degree $j$.
  Then there is an upper triangular polynomial matrix $T(\la)$
  that is unimodularly equivalent to $S(\la)$,
  with diagonal degree vector $\,\deg \bigl( \diag T(\la) \bigr)$
  that is $k$-homogeneous,
  and such that each diagonal factor-counting vector $\mathbf{d}_{\scFj}(T)$
  is $1$-homogeneous.
\end{coro}

\begin{proof}
  Use the given multiset $\cM$ of $\bF$-irreducible polynomials
  as input to Lemma~\ref{Stacking}.
  The output vector $\mathbf{p} (\la)$ from that Lemma
  is now the target diagonal for the desired upper triangular $T(\la)$.
  Since each factor-counting vector $|\mathbf{p}(\la)|_{\scFj}$
  is $1$-homogeneous,
  the Homogenization Lemma~\ref{lem.conversion-to-homogen-version}
  guarantees that the transition from the vector $\diag S(\la)$
  to the vector $\mathbf{p} (\la)$
  can be achieved using a finite number of compressions and interchanges.
  But Corollary \ref{cor.unimod-trans} gives us the means
  to implement all of these compressions and interchanges
  as unimodular transformations applied to $S(\la)$.
  Doing this then converts $S(\la)$
  into the desired upper triangular $T(\la)$.
%
% \medskip
% \tcb{Flesh out this argument!} \quad
% This lemma gives us a target diagonal,
% and Lemma \ref{lem.conversion-to-homogen-version} guarantees
% that we only need a finite number of compressions and interchanges
% to move the degree $ i $ irreducible factors around
% until their factor-counting  vector is 1-homogeneous
% for each $ i = 1,2,\dotsc,k $.
% Corollary \ref{cor.unimod-trans} gives us the means
% to perform all the compressions as unimodular transformations,
% and it is here that we go from having a diagonal matrix $ S(\la) $
% to having an upper triangular matrix $ T(\la) $
% with diagonal entries equal to those
% of the diagonal matrix $ D(\la) $ in Lemma \ref{Stacking}.
\end{proof}

%%%%%%%%%%%%%%%%%%%%%%%%%%%%%%%%%%%%%%%%%%%%%%%%%%%%%%%%%%%%%%%%%%%%%%%%%%%%
\subsection{A combinatorial lemma} \label{subsect.combinatorial-lemma}

In this short section we focus on establishing a new combinatorial property %result
of ``tightly packed'' \emph{integer} multisets,
that is, multisets that contain more, perhaps even many more elements
than the width of the interval into which they are packed.
This property will enable us to permute the diagonal entries
of the $T(\la)$ from Corollary~\ref{cor.un-diagonalization}
in preparation for the final phase
of our quasi-triangular realization construction.
Note that for ease of expression,
in this section we use the word ``list''
as a synonym for multiset;
however, nothing about any ordering of these lists
is of any relevance for the development here.

\begin{lemma}[Homogeneous Partitioning Property] \label{lem.homog-partition} \quad \\
  Let $\cI = [\,j, j+k\,]$ be a closed interval with \emph{integer} endpoints and length $k \ge 1$,
  and consider a list $\mathcal{L} = \{n_1,n_2,\dotsc,n_m\}$ of $m$ integers,
  all in $\cI$.
  If the average value of all of the entries in $\cL$
  is an \emph{integer} $\mu \in \cI$,
  then $\cL$ can be partitioned into sublists $\cS_1,\cS_2,\dotsc,\cS_{\ell}$
  such that the number of entries in each $\cS_i$ does not exceed $k$,
  and the average value of each sublist is exactly $\mu$.
%   \quad \tcr{This use of $M$ violates the Householder convention.
%              Pick another letter for this average!}
\end{lemma}

To help prove this result,
we need another lemma characterizing the solution set
of a certain diophantine equation in two variables.

\begin{lemma}
\label{lem-dio}
	Let $ a $ and $ b $ be positive integers, with $d := \gcd\{a, b\}$.
	Then the set of integer solutions $ (x,y) $ of the equation $ ax = by $
	consists of all the integer multiples of the pair $ (\widetilde b,\widetilde a) $,
	where $ \widetilde a = a/d $ and $ \widetilde b = b/d$.
\end{lemma}
\begin{proof}
	We can immediately reduce to the equation $ \widetilde ax = \widetilde by $,
	which has exactly the same solution set as $ax = by$.
	Since $ \widetilde a $ and $ \widetilde b $ are relatively prime,
	for any solution pair $(x,y)$
	it must be true that $ \widetilde a | y $ and $ \widetilde b | x $.
	Letting $ x = \widetilde bk $ and $ y = \widetilde a\ell $,
	then $ \widetilde a\widetilde b k = \widetilde b\widetilde a\ell $
	implies that $ k = \ell $.
	Thus the solutions of $ \widetilde ax = \widetilde by $
	(and hence also of $ax = by$)
	are exactly the integer multiples of $ (\widetilde b,\widetilde a) $.
	%   and since there is a 1-1 correspondance between solutions
	%   of the original equation and solutions of the reduced equation,
	%   the proof is complete.
\end{proof}

We now prove the Homogeneous Partitioning Property.

\begin{proof}
	(of Lemma~\ref{lem.homog-partition}) \,\,
	One strategy to achieve this homogeneous partitioning
	is to first translate all the given data
	(i.e., the interval $\cI$ and the list $\cL$)
	by any fixed constant $c \in \bZ$, solve the translated problem,
	and then ``un-translate'' the solution back to the original location.
	Thus it suffices to solve the problem for the case when $\mu = 0$
	and the interval $\cI$ is $[\,e_{\ell}, e_r\,]$,
	with endpoints $e_{\ell} \le 0$ and $e_r \ge 0$ such that $e_r - e_{\ell} = k$.
	The goal in this more specialized scenario, then, is to partition $\cL$
	into sublists of no more than $k$ entries each,
	such that the \emph{sum} (equals the average)
	of the entries of each sublist $\cS_i$ is zero.
	
% 	Let us consider now a translate of the list $ \cL $, i.e.,
% 	\[
% 	\wh\cL = \{\widetilde n_1,\widetilde n_2,\dotsc,\widetilde n_m\} \,
% 	\]
% 	with entries $ \widetilde n_i = n_i - \delta $ ranging from $ -\delta $ up to $ k-\delta $,
% 	and an average value of $0$.
% 	This modifies the goal slightly; we are now looking to partition $ \wh\cL $
% 	into sublists of no more than $ k $ entries each,
% 	such that the sum of the entries of each sublist $\wh\cS_i$ is zero.
%
% \begin{algorithm}[h!]
% $i=1$\\
% \While{$|\cL| > k$}{
% 	\eIf{$\exists n_{i_1}\in\cL \text{ s.t. } -\delta < n_{i_1} < k-M$}{
% 		\eIf{$n_{i_1} = 0$}{
% 			$\cS_i = \{n_{i_1}\}$ \;
% 			$\cL = \cL \setminus\cS_i$ \;
% 			$i = i+1$\;}
% 		{
% 			$\cS_i = \{n_{i_1}\} $\;
% 			$j=1$\;
% 			$\sigma_{i_1} = n_{i_1}$\;
% 			\While{$\sigma_{i_j} \neq 0$}{
% 				\If{$\sigma_{i_j} = \sigma_{i_p} \text{ for some } p<j$}{
% 					$\cS_i = \cS_i \setminus \{n_{i_1},\dotsc,n_{i_p}\}$\;
% 					{\bf exit loop}\;}
% 				choose $n_{i_{j+1}} \in \cL \setminus\cS_i$ s.t. $ \sign(\sigma_i) = -\sign(n_{i_{j+1}}) $\;
% 				$\cS_i = \cS_i \cup \{n_{i_{j+1}}\}$\;
% 				$\sigma_{i_{j+1}} = \sigma_{i_j} + n_{i_{j+1}}$\;
% 				$j = j+1$\;}
% 			$\cL = \cL \setminus \cS_i $ \;
% 			$i = i+1$\;}}
% 	{
% %		$\alpha = $ \# of entries in $\cL$ with value $-M$ \;
% %		$\beta = $ \# of entries in $ \cL $ with value $k-M$ \;
% 		$d = \gcd(M,k-M) $\;
% 		$\cS_i = \{\underbrace{-M,\dotsc,-M}_{\sfrac{(k-M)}{d}},\underbrace{k-M,\dotsc,k-M}_{\sfrac{M}{d}}\} $\;
% 		$\cL = \cL\setminus\cS_i$\;
% 		$i=i+1$\;}}
% \caption{Homogeneous Partitioning Algorithm}
% \end{algorithm}

The proof for this special scenario proceeds by an induction on $m$,
the number of elements in the list $\cL$.
(Note that this induction can be easily converted into an algorithm
for computing the desired partition.)
The base case(s) for this induction are all $m$ such that $1 \le m \le k$,
for which the result is trivially true.
So now suppose that the homogeneous partitioning property
holds for all number lists satisfying the hypotheses of Lemma~\ref{lem.homog-partition}
with $m \le h$,
and consider a list $\cL$ of $m = h+1$ integers
with values in $\cI = [\,e_{\ell}, e_r\,]$ and average value $\mu = 0$.
There are now three cases to consider:
\begin{enumerate}
  \item[\rm(a)] \emph{Some $n_i \in \cL$ is equal to zero.} \\
                In this case we can split off the singleton sublist $\cS_1 = \{ n_i \}$
                from $\cL$, leaving a smaller list $\wh{\cL}$ with $m=h$ entries,
                and average value $\wh{\mu} = 0$.
                Applying the inductive hypothesis to $\wh{\cL}$ completes
                the homogeneous partitioning of $\cL$.

  \item[\rm(b)] \emph{No element of $\cL$ is zero,
                      and all elements of $\cL$ lie at the endpoints $e_{\ell}$ and $e_r$ of $\,\cI$.} \\
                Suppose there are $\alpha$ copies of $e_{\ell}$
                and $\beta$ copies of $e_r$, so that $\,\alpha + \beta = h+1 > k$,
                and $\,\alpha e_{\ell} + \beta e_r = 0$,
                or equivalently $(-e_{\ell}) \alpha = e_r \beta$.
                Let $d = \gcd \{ -e_{\ell}, e_r \}$.
                Applying Lemma~\ref{lem-dio} to the equation $ax = by$
                with $a= -e_{\ell}$ and $b = e_r$,
                we see that the solution $(x,y) = (\alpha, \beta)$ to $ax = by$
                is an integer multiple of $(\frac{\,e_r\,}{d}, \frac{\,-e_{\ell}\,}{d})$.
                Thus we can completely partition $\cL$ into sublists $\cS_i$,
                each consisting of $\frac{\,e_r\,}{d}$ copies of $e_{\ell}$
                and $\frac{\,-e_{\ell}\,}{d}$ copies of $e_r$.
                Since $\,\frac{\,e_r\,}{d} + \frac{\,-e_{\ell}\,}{d} = \frac{\,k\,}{d}$,
                each of these sublists has $\frac{\,k\,}{d} \le k$ elements,
                and sum zero, as desired.
  \item[\rm(c)] \emph{No element of $\cL$ is zero,
                      but there is some $n_{i_1}$ from $\cL$
                      in the \emph{open} interval $\,\wt{\cI} = (e_{\ell}, e_r)$.} \\
                Begin building a sublist $\cS$ with the given (nonzero) element $\,n_{i_1}$ in $\wt{\cI}$.
                Pick from among the remaining elements of $\cL$
                to update $\cS = \{n_{i_1}, n_{i_2}, \dotsc, n_{i_j}, \dotsc \}$,
                and keep track of the ``partial sums'' $\sigma_j := \sum_{\ell=1}^j n_{i_{\ell}}$
                as you go to see if a zero sum has been achieved.

                At each stage, the element $n_{i_j}$ to be appended to $\cS$ is chosen
                to be any one of the remaining elements of $\cL$ that have
                a sign \emph{opposite} to that of $\sigma_{j-1}$,
                in order to try to drive the partial sum value to zero.
                Observe that there must always exist such an ``opposite-sign'' element
                remaining in $\cL$,
                since otherwise the sum (hence also the average)
                of all of the elements in $\cL$ would not be zero.
                Another consequence of this opposite-sign strategy
                is that the $\sigma_j$ values can never be equal
                to either endpoint $e_{\ell}$ or $e_r$ of the interval $\cI$;
                each $\sigma_j$ must be one of the $k-1$ integers
                in the interior of $\cI$, i.e. in $\,\wt{\cI}$.
                To see why this is so, first observe that
                $\sigma_1 = n_{i_1}$ is in $\,\wt{\cI}$ by construction.
                For the passage from $\sigma_j$ to $\sigma_{j+1}$ with $j \ge 1$,
                there are three scenarios:
                \[
                  \text{(i)} \;\;e_{\ell} < \sigma_j < 0 \,,
                     \qquad \text{(ii)} \;\;\sigma_j = 0 \,,
                     \quad\text{ or }\quad \text{(iii)} \;\;0 < \sigma_j < e_r \,.
                \]
                In case (i), $\sigma_{j+1}$ can be at most $e_r$ larger than $\sigma_j$,
                so $\sigma_{j+1} \in \,\wt{\cI}$.
                In case (ii), the construction will cease, and there will be no $\sigma_{j+1}$.
                And in case (iii), $\sigma_{j+1}$ can be at most
                $\lvert e_{\ell} \rvert$ smaller than $\sigma_j$,
                so once again $\sigma_{j+1} \in \,\wt{\cI}$.

                \smallskip
                Now carry on the building up of the sublist $\cS$
                using the ``opposite-sign'' strategy,
                until either \,(a) a partial sum $\sigma_j = 0$ with $j \le k-1$ is attained,
                or \,(b) the sublist contains $k-1$ elements with \emph{every} partial sum
                $\sigma_1, \dotsc, \sigma_{k-1}$ being nonzero.
                If (a) occurs, then split off the sublist
                $\cS = \{n_{i_1}, n_{i_2}, \dotsc, n_{i_j}\}$ from $\cL$,
                and the remaining sublist $\wh{\cL}$ can be homogeneously partitioned
                by the inductive hypothesis.
                On the other hand, if (b) occurs,
                then the $k-1$ nonzero partial sums must have some repetitions,
                since there are only $k-2$ nonzero integers in the open interval $\,\wt{\cI}$.
                So suppose that $\sigma_{i_p} = \sigma_{i_j}$ for some $p < j \leq k-1$,
                and let $\wh{\cS} := \{ n_{i_{p+1}}, \dotsc, n_{i_j} \}$,
                with at most $k-2$ elements.
                Observe that the sum of the elements in $\wh{\cS}$ is
                \[
                  \sum_{\ell=p+1}^j n_{i_\ell}
                      \;=\; \biggl( \,\sum_{\ell = 1}^j n_{i_\ell} - \sum_{\ell=1}^p n_{i_\ell} \biggr)
                      \;=\; \sigma_{i_j} - \sigma_{i_p}
                      \;=\; 0 \,,
                \]
                so splitting off the sublist $\wh{\cS}$ from $\cL$
                starts the homogeneous partitioning,
                leaving a remaining sublist $\wh{\cL}$
                that can be homogeneously partitioned by the inductive hypothesis.
\end{enumerate}
The result for the general interval $\cI = [\,j, j+k\,]$ now follows by translation.
\end{proof}

\begin{remark} \label{rem.HP-proof-analysis}
\rm
  A closer examination of the proof %of Lemma~\ref{lem.homog-partition}
  of the Homogeneous Partitioning property
  indicates that the presence of sublists of ``full length'' $k$
  in a homogeneous partitioning may be somewhat rare.
  In most of the scenarios for splitting off a sublist $\cS$ from the main list $\cL$,
  the length of the split-off sublist is strictly less than $k$.
  In fact, the \emph{only} scenario that can force a sublist to have length $k$
  is very special;
  with all data translated so that $\mu = 0$,
  \emph{all} of the elements of $\cL$ must be at the endpoints $e_{\ell}$ and $e_r$,
  \emph{and} these endpoints must be relatively prime.

  Another feature of the homogeneous partitioning problem
  that is hinted at in the proof
  is that often these partitions are \emph{not unique}.
  Indeed, in the partitioning procedure described in the proof,
  in particular for case (c),
  there may be many arbitrary choices that can be made,
  all of which lead to an admissible partitioning.
  In addition, there may be other homogeneous partitionings
  that cannot be generated from the procedure in the proof at all,
  no matter what choices are made.
\end{remark}

\begin{remark}
 \rm
   It is worth noting that Lemma~\ref{lem.homog-partition}
   can be extended to apply to the situation
   in which the average $\mu \in \cI$ is \emph{not} necessarily an integer.
   For this general case the statement %of the Homogeneous Partitioning property
   looks like the following:
   \begin{quote}
     \underline{General Homogeneous Partitioning Property} \\[4pt]
     Let $\cI = [\,j, j+k\,]$ be a closed interval with \emph{integer} endpoints and length $k \ge 1$,
     and consider a list $\mathcal{L} = \{n_1,n_2,\dotsc,n_m\}$ of $m$ integers,
     all in $\cI$.
     Let $\mu \in \cI$ denote the average value of all of the entries in $\cL$.
     Then $\cL$ can be partitioned into sublists $\cS_1,\cS_2,\dotsc,\cS_{\ell}$
     such that the number of entries in each $\cS_i$ does not exceed $k$,
     and the average value of each sublist
     is in the closed interval $\bigl[\, \lfloor \mu \rfloor, \lceil \mu \rceil \,\bigr]$.
   \end{quote}
   Note that this more general partitioning result
   reduces to exactly Lemma~\ref{lem.homog-partition}
   when the average $\mu$ is an integer.
   However, since we will not need the general property for this paper,
   we omit the proof.
\end{remark}

\begin{ejem} \label{ex.homogeneous-partition}
  Recall that in Example~\ref{ex.stacking} we took the diagonal vector of a Smith form $S(\la)$,
  and rearranged the $\bF$-irreducible factors via Lemma~\ref{Stacking}
  to obtain a vector $\mathbf{p}(\la)$ with average degree $10$
  and $4$-homogeneous degree vector $\deg \mathbf{p} = (12, 12, 11, 9, 8, 8)$.
  Lemma~\ref{lem.homog-partition} now guarantees that there is a partitioning
  of $\,\deg \mathbf{p}\,$ (and a corresponding partitioning of the entries of $\mathbf{p}$ itself)
  into sublists of size at most $4$, where \emph{each} sublist
  also has average degree $10$.
  Applying the procedure described in the proof of the Lemma
  produces the partition $(12,8\mid12,8\mid11,9)$ for $\,\deg \mathbf{p}\,$,
  and the corresponding rearrangement and partition
  \begin{equation} \label{eqn.partition-example}
    \bigl( p_1(\la), p_6(\la) \mid p_2(\la), p_5(\la) \mid p_3(\la), p_4(\la) \bigr)
      \,=\,
    \bigl( \chi^2\phi\psi^2, \chi\phi\psi^2 \mid \chi^2\phi\psi^2,
                   \chi\phi\psi^2 \mid \chi^2\phi\psi, \chi\phi^2\psi \bigr)
  \end{equation}
  of the entries of $\,\mathbf{p}(\la)$.
  It is interesting to note
  that there are several pathways through the partitioning procedure for this example,
  but all of them lead to the partitioning in \eqref{eqn.partition-example}.
  However, there are two other homogeneous partitionings of $\,\deg \mathbf{p}\,$
  with sublist size at most $k=4$,
  neither of which can be generated by the procedure of the Lemma.
  They are $(12,8,12,8 \mid 11,9)$ and $(12,8 \mid 12,8,11,9)$.
\end{ejem}

\subsection{Un-triangularizing $T(\la)$}
  \label{subsect.untriangularizing}

One final tool is needed in order to complete the construction
of our $k$-quasi-triangular realization.
After our realization is initially brought into quasi-triangular form,
it may happen that the off-diagonal blocks have increased
in degree beyond the target for the final matrix polynomial,
since no control of these blocks has even been attempted
in the early stages of the construction.
Thus we need some method to bring these degrees
back within the target range.
 Furthermore, it is important to do this by unimodular transformations,
so that the desired finite spectral structure will not be spoiled.
Lemma~\ref{LemOffDiagClean}, a %(\tcb{minor? should we say this?})
generalization of Lemma 2.4 from \cite{TisZab},
shows how to achieve this goal.

Two preliminaries are needed for Lemma~\ref{LemOffDiagClean}.
 First it will be helpful to recall the notion of %(and notation for)
the ``diagonals'' of a matrix,
and more generally the block-diagonals of a block-partitioned matrix.

\begin{defi}[Diagonals of a matrix] \label{def.diagonals} \quad \\
  The \emph{$\ell^{th}$-diagonal} of a matrix $A$
  is the set of entries $a_{ij}$ of $A$ such that $j-i = \ell$.
  (Note that the $0^{th}$-diagonal is conventionally known as the main diagonal of $A$.)
  If $A$ is block-partitioned with blocks $A_{ij}$,
  then the \emph{$\ell^{th}$-block-diagonal} of $A$
  is the collection of blocks $A_{ij}$ with $j-i = \ell$.
\end{defi}

\noindent
A second important background fact concerns the division of matrix polynomials.
We recall now this fundamental result for the convenience of the reader.

\begin{lema}[Division of matrix polynomials] \label{lem.matrix-poly-division} \quad \\
  Suppose $A(\la)$ is any $\mbyn$ matrix polynomial over an arbitrary field $\,\bF$.
  Furthermore, let $B(\la)$ and $C(\la)$ be any two
  \emph{strictly regular} matrix polynomials over $\,\bF$,
  with size $\mbym$ and $\nbyn$, respectively.
  Then we have:
  \begin{enumerate}
    \item[\rm (a)] \underline{\rm Left division by $B$}:
                   There exist unique $\mbyn$ matrix polynomials $Q_{\ell}(\la)$ and $R_{\ell}(\la)$
                   such that
                   \[
                     A(\la) \;=\; B(\la) Q_{\ell}(\la) \,+\, R_{\ell}(\la) \,,
                   \]
                   and either $R_{\ell}(\la) = 0$,
                   or $R_{\ell}(\la)$ is nonzero with $\deg R_{\ell} < \deg B$.
                   The matrices $Q_{\ell}(\la)$ and $R_{\ell}(\la)$
                   are called the left quotient and left remainder of $A(\la)$, respectively,
                   upon \parens{left} division by $B(\la)$.
    \item[\rm (b)] \underline{\rm Right division by $C$}:
                   There exist unique $\mbyn$ matrix polynomials $Q_r(\la)$ and $R_r(\la)$
                   such that
                   \[
                     A(\la) \;=\; Q_r(\la) C(\la) \,+\, R_r(\la) \,,
                   \]
                   and either $R_r(\la) = 0$,
                   or $R_r(\la)$ is nonzero with $\deg R_r < \deg C$.
                   The matrices $Q_r(\la)$ and $R_r(\la)$
                   are called the right quotient and right remainder of $A(\la)$, respectively,
                   upon \parens{right} division by $C(\la)$.
  \end{enumerate}
\end{lema}
\begin{proof}
  See \cite[Ch.\ 4]{Gant59}, \cite[Sect.\ 6.3]{Kailath}, or \cite[Sect.\ 7.2]{LanTis}.
\end{proof}

\begin{lema}[Degree reduction of off-diagonal blocks]
	\label{LemOffDiagClean} \quad \\
	Let $T(\lambda)\in\mathbb{F}[\lambda]^{n\times n}$
	be a block upper triangular matrix polynomial,
	partitioned into blocks such that all of the $s$ diagonal blocks
	$T_{ii}(\la) \in \mathbb{F}[\lambda]^{n_i\times n_i}$
	are \emph{strictly regular}. % for $i=1,\dots, s$.
	Then $T(\lambda)$ is unimodularly equivalent to a block upper triangular matrix polynomial
	$\widetilde T(\lambda)$ with exactly the same diagonal blocks as $T(\la)$,
	and with off-diagonal blocks satisfying
	\begin{equation} \label{eqn.degree-condition}
	  \deg \wt T_{ij}(\lambda)
	     \;<\;
      \min \bigl\{\,\deg T_{ii}(\lambda), \,\deg T_{jj}(\lambda) \,\bigr\}
    \end{equation}
	for $\,1 \leq i < j \leq s$.
\end{lema}

% \begin{proof}
\medskip
\noindent
\emph{Proof.}
  Let us begin by focusing
  on a single fixed but arbitrary off-diagonal block $T_{ij}(\la)$ with $i < j$,
  and showing how to reduce its degree by a unimodular transformation
  so as to satisfy the condition \eqref{eqn.degree-condition}.
  It may be that \eqref{eqn.degree-condition} is already satisfied;
  in this case do nothing.
  Otherwise, suppose first that $\deg T_{ii} \le \deg T_{jj}$.
  Then by Lemma~\ref{lem.matrix-poly-division} we can
  divide $T_{ij}(\la)$ by $T_{ii}(\la)$ \emph{on the left}
  to obtain
  \[ T_{ij}(\la) \,=\, T_{ii}(\la)Q_{ij}(\la) + R_{ij}(\la)
      \;\text{ with }\; R_{ij}(\la) \,=\,0  \;\text{ or }\; \deg R_{ij}(\la) < \deg T_{ii}(\la) \,,
  \]
  and hence $T_{ij}(\la) - T_{ii}(\la)Q_{ij}(\la) =  R_{ij}(\la)$.
  Now define the unimodular matrix
  \[
    V_{ij}(\la) \,:=\;
    \mat{ccccc} I_a & & & & \\
                  & I_{n_i}&   & -Q_{ij}(\la) & \\
                  &        & I_b &  & \\
                  &        &   & I_{n_j} & \\
                  &        &   &         & I_c \rix \,,\;
    \text{ where } \quad
    a = \sum_{m=1}^{i-1} n_m \,, \;
    b = \sum_{m=i+1}^{j-1} n_m \,, \;
    c = \sum_{m = j+1}^{s} n_m \,.
  \]
  Multiplying $T(\la)$ on the right by $V_{ij}(\la)$
  has the effect of an elementary block-column operation,
  which replaces the block $T_{ij}(\la)$ by $\wt{T}_{ij}(\la) = R_{ij}(\la)$,
  thus satisfying condition \eqref{eqn.degree-condition}
  at this one location.
  In addition we see that the diagonal blocks of $T(\la)$
  remain unchanged by this block-column operation.
  Indeed, the \emph{only} other blocks of $T(\la)$
  that may even possibly be affected
  are the ones directly above the $(i,j)$-block,
  i.e., blocks $T_{\ell j}(\la)$ with $\ell \le i$.

  On the other hand, if $\deg T_{ii} > \deg T_{jj}$
  then something analogous can be done.
  In this case use Lemma~\ref{lem.matrix-poly-division}
  to divide $T_{ij}(\la)$ by $T_{jj}(\la)$ \emph{on the right}
  to obtain
  \[ T_{ij}(\la) \,=\, \wh{Q}_{ij}(\la) T_{jj}(\la) + \wh{R}_{ij}(\la)
      \;\text{ with }\; \wh{R}_{ij}(\la) \,=\,0  \;\text{ or }\; \deg \wh{R}_{ij}(\la) < \deg T_{jj}(\la) \,,
  \]
  and hence $T_{ij}(\la) - \wh{Q}_{ij}(\la) T_{jj}(\la) =  \wh{R}_{ij}(\la)$.
  Define the unimodular matrix
  \[
    U_{ij}(\la) \,:=\;
    \mat{ccccc} I_a & & & & \\
                  & I_{n_i}&   & -\wh{Q}_{ij}(\la) & \\
                  &        & I_b &  & \\
                  &        &   & I_{n_j} & \\
                  &        &   &         & I_c \rix \,,\;
    \text{ where } \quad
    a = \sum_{m=1}^{i-1} n_m \,, \;
    b = \sum_{m=i+1}^{j-1} n_m \,, \;
    c = \sum_{m = j+1}^{s} n_m \,,
  \]
  and multiply $T(\la)$ on the left by $U_{ij}(\la)$.
  This has the effect of an elementary block-row operation,
  which replaces the block $T_{ij}(\la)$ by $\wt{T}_{ij}(\la) = \wh{R}_{ij}(\la)$,
  once again satisfying condition \eqref{eqn.degree-condition}
  at this one location.
  In the product $U_{ij}(\la) T(\la)$,
  the only other blocks of $T(\la)$ that can possibly affected
  are those to the right of the $(i,j)$-block,
  i.e., blocks $T_{i \ell}(\la)$ with $\ell \ge j$.
  Altogether, then, the only possible collateral damage that can be inflicted
  on $T(\la)$ by this degree reduction of the $(i,j)$-block
  is to the blocks of $T(\la)$ in the L-shaped region marked in \eqref{eqn.block-T}.
\begin{equation} \label{eqn.block-T}
  T(\la) \;=\;
  \mat{cccccccc} \diamond &        &  &    & \tcr{\bigl\Vert}  & & & \\[-5pt]
                          & \ddots &  &    & \tcr{\Bigl\lVert} & & & \\[4pt]
            &    & T_{ii} &  & T_{ij} & \tcr{***} & \hspace*{-2.5mm}\tcr{***\,*} & \hspace*{-2.7mm}\tcr{***} \\
                     &        &         & \ddots &         &  &  &  \\
                     &        &         &        & T_{jj}  &  &  &  \\
                     &        &         &        &         & \ddots &  &\\
                     &        &         &        &         &        & \ddots & \\
                     &        &         &        &         &        &    & \hspace*{-3mm} \diamond \rix
\end{equation}
  It is important to keep this L-shape firmly in mind
  as we see how to order the block degree reductions
  so that no individual block's degree reduction
  spoils a block that has already had its degree reduced.

  The key observation here is that whenever a block $T_{ij}(\la)$
  has its degree reduced by the procedure described above,
  then all of the other blocks in $T(\la)$ that are affected by that reduction
  lie on ``higher'' diagonals of $T(\la)$ than the diagonal of $T_{ij}(\la)$.
  More precisely, if $T_{ij}(\la)$ lies on the $\ell^{th}$-block-diagonal of $T(\la)$,
  then all other blocks affected by that degree reduction
  lie on an $m^{th}$-block-diagonal with $m > \ell$.
  Thus a single \emph{sweep-by-diagonals} through the off-diagonal blocks
  will have the desired effect
  of achieving condition \eqref{eqn.degree-condition}
  on \emph{all} off-diagonal blocks simultaneously.
  To do this kind of sweep,
  first target each of the blocks in the $1^{th}$-block-diagonal for degree reduction,
  in any order.
  After visiting each block in the $1^{th}$-block-diagonal,
  go next to the $2^{th}$-block-diagonal and do a degree reduction
  on each of these blocks, again in any order.
  Because of the key observation,
  none of the degree reductions done in the $1^{th}$-block-diagonal
  will be spoiled by the degree reductions done in the $2^{th}$-block-diagonal,
  and none of the $2^{th}$-block-diagonal degree reductions will spoil each other.
  Continue in this manner,
  moving up one block-diagonal at a time,
  until each off-diagonal block has been visited exactly once.
  At this point, the desired matrix polynomial $\wt{T}(\la)$
  will have been attained. \hfill $\square$

\medskip
We finally have all the tools we will need to ``un-triangularize" $ T(\la) $
into a degree-$d$, $k$-quasi-triangular realization
of the original list of finite spectral data.
 For convenience, we recall the statement of Theorem~\ref{thm.QTR-strictly-regular} here.
\begin{teor*}[Quasi-Triangular Realization: Strictly Regular Case]
 	\label{thm.QTR-strictly-regular-reprise} \quad \\
 	Suppose a list of $m$ nonconstant monic polynomials $s_1(\lambda),\dots, s_m(\lambda)$
 	over an arbitrary field $\,\bF$ is given,
 	satisfying the divisibility chain condition
 	$s_1(\la) \,\vert\, s_2(\la) \,\vert \,\dotsb\, \vert\, s_m(\la)$.
 	Let $\sigma := \sum_{i=1}^m \deg \bigl( s_i(\lambda) \bigr)$,
 	and define $k$ to be the maximum degree among all of the $\bF$-irreducible factors
 	of the polynomials $s_i(\lambda)$ for $i = 1,\dotsc, m$.
 	Then for any choice of nonzero $d, n \in \bN$
 	such that $n \ge m$ and $dn = \sigma$,
 	there exists an $n\times n$, degree $d$, strictly regular matrix polynomial $Q(\lambda)$ over $\bF$
 	that is $k$-quasi-triangular,
 	and has exactly the given polynomials $s_1(\lambda),\dots, s_m(\lambda)$
 	as its nontrivial invariant polynomials,
 	together with $n-m$ trivial invariant polynomials. % equal to 1.
 	In addition, $Q(\lambda)$ can always be chosen so that the degree
 	of every entry in any off-diagonal block of $Q(\lambda)$
 	is strictly less than $d$.
 \end{teor*}
\begin{proof}
  From the given spectral data, begin by constructing the $ n\times n $ Smith form
  \[
    S(\la) = \diag\{\, \,\underbrace{1,\dotsc,1}_{n-m}, s_1(\la), \dotsc, s_m(\la) \,\} \,.
  \]
  Now we can use the tools developed in Corollary~\ref{cor.un-diagonalization},
  Lemma~\ref{lem.homog-partition}, Corollary~\ref{cor.unimod-trans}, Corollary~\ref{cor.prescribed},
  and Lemma~\ref{LemOffDiagClean} to build the desired $k$-quasi-triangular realization
  of the spectral data contained in $S(\la)$,
  in the following five steps:
  \begin{itemize}
    \item Use $S(\la)$ as input to Corollary~\ref{cor.un-diagonalization}
          to generate an upper triangular $T(\la)$ that is unimodularly equivalent to $S(\la)$,
          and has a diagonal degree vector $\,\deg \bigl( \diag T(\la) \bigr)$
          that is $k$-homogeneous.
    \item Use the natural vector $\,\deg \bigl( \diag T(\la) \bigr)$,
          with average value $\mu = d$,
          as input to Lemma~\ref{lem.homog-partition},
          and find a homogeneous partitioning of the degrees of the diagonal entries of $T(\la)$
          into $\ell$ sublists.
          The corresponding partitioning of the diagonal entries themselves,
          with the elements of the entry sublists arranged into $\ell$ contiguous groups,
          provides a target for the rearrangement of the entries on the diagonal of $T(\la)$.
    \item Implement this rearrangement of the diagonal entries of $T(\la)$
          via the triangularity-preserving unimodular transformations of $T(\la)$
          provided by Corollary~\ref{cor.unimod-trans},
          each of which has the effect of simply performing an interchange of adjacent diagonal entries.
          (Of course some of the off-diagonal entries are being changed in the process,
           but we do not try to keep any control of them at this stage of the construction.)
          Denote the resulting upper triangular matrix by $\wh{T}(\la)$,
          and partition this $\wh{T}(\la)$ into blocks
          \begin{equation}\label{eqn.tri-partition1}
            \wh{T}(\la) \;=\;
               \left[\begin{array}{cccc} \wh{T}_{11}(\la) & \wh{T}_{12}(\la) & \dotsm & \wh{T}_{1\ell}(\la) \\
                                       & \wh{T}_{22}(\la) & \dotsm & \wh{T}_{2\ell}(\la) \\[-2pt]
                                       & & \ddots & \vdots \\
                                       & & & \wh{T}_{\ell\ell}(\la) \end{array}\right] \,,
          \end{equation}
          so that each (upper triangular and square) diagonal block $\wh{T}_{jj}(\la)$
          has diagonal entries that correspond to the sublist $\cS_j$
          of the homogeneous partitioning of $\,\deg \bigl( \diag T(\la) \bigr)$.
          For each $j$, if we denote the size of the block $\wh{T}_{jj}(\la)$ by $n_j \times n_j$,
          then from the Homogeneous Partitioning property
          we know that $n_j \le k$,
          and the average degree of the diagonal entries of $\wh{T}_{jj}(\la)$ is $d$.
          So the sum of the degrees of the diagonal entries of $\wh{T}_{jj}(\la)$ is $d n_j$.
    \item By Corollary~\ref{cor.prescribed} we know that each \emph{diagonal} block $\wh{T}_{jj} (\la)$
          is unimodularly equivalent to a strictly regular polynomial matrix $P_{jj}(\la)$
          of degree $d$, which is (probably) no longer upper triangular.
          Let these equivalences be denoted by
          \[
            P_{jj}(\la) :=\, \wh{U}_{jj}(\la) \wh{T}_{jj}(\la) \wh{V}_{jj}(\la) \,,
          \]
          where each $\wh{U}_{jj}(\la)$ and $\wh{V}_{jj}(\la)$ is $n_j \times n_j$ and unimodular.
          Now define the $\nbyn$ block-diagonal unimodular matrices
          \[
            \wh{U}(\la) :=\, \diag \bigl\{\, \wh{U}_{11}(\la), \wh{U}_{22}(\la),
                                   \dotsc, \wh{U}_{\ell \ell}(\la) \,\bigr\}
                \quad \text{ and } \quad
            \wh{V}(\la) :=\, \diag \bigl\{\, \wh{V}_{11}(\la), \wh{V}_{22}(\la),
                                   \dotsc, \wh{V}_{\ell \ell}(\la) \,\bigr\} \,,
          \]
          and apply them to the full $\wh{T}(\la)$,
          to get
          \begin{equation}\label{eqn.tri-partition2}
              P(\la) :=\, \wh{U}(\la) \wh{T}(\la) \wh{V}(\la) \,=\,
                     \left[\begin{array}{cccc} P_{11}(\la) & P_{12}(\la) & \dotsm & P_{1\ell}(\la) \\
                                   & P_{22}(\la) & \dotsm & P_{2\ell}(\la) \\
                                   & & \ddots & \vdots \\
                                   & & & P_{\ell\ell}(\la) \end{array}\right].
          \end{equation}
          This matrix $P(\la)$ is now $k$-quasi-triangular,
          with strictly regular diagonal blocks each of degree $d$.
          But $P(\la)$ as a whole may not yet be of degree $d$,
          because the off-diagonal blocks have not been kept under any control at all.

    \item The final step brings the degrees of the off-diagonal blocks back under control,
          while at the same time not disturbing the diagonal blocks in the process.
          Lemma~\ref{LemOffDiagClean} applied to $P(\la)$ achieves this,
          reducing the degree of each off-diagonal block to be strictly less than $d$,
          and leaving the diagonal blocks unchanged,
          to obtain the final desired realization %$\wt{P}(\la)$.
          \begin{equation} \label{eqn.final-str-reg-realization}
             Q(\la) \;=\;
                     \left[\begin{array}{cccc} Q_{11}(\la) & Q_{12}(\la) & \dotsm & Q_{1\ell}(\la) \\
                                   & Q_{22}(\la) & \dotsm & Q_{2\ell}(\la) \\
                                   & & \ddots & \vdots \\
                                   & & & Q_{\ell\ell}(\la) \end{array}\right] \,.
          \end{equation}
          with $Q_{jj}(\la) = P_{jj}(\la)$ for $j = 1, \dotsc, \ell$.
  \end{itemize}
  This polynomial matrix $Q(\la)$ is $k$-quasi-triangular,
  has degree $d$,
  and is unimodularly equivalent to the original $S(\la)$,
  and hence has exactly the given finite spectral data.
  To see that $Q(\la)$ is strictly regular,
  we regard $\grade (Q)$ as being equal to the degree $d$;
  by Lemma~\ref{lem.grade-equal-degree}, any other choice
  will force $Q(\la)$ to have nontrivial infinite spectral structure.
  Since by construction the sum of the degrees
  of all of the invariant polynomials of $Q(\la)$ is $dn$,
  the Index Sum Theorem~\ref{thm.index-sum} immediately shows
  that the sum of the partial multiplicities of $Q(\la)$ at infinity must be zero,
  and hence that $Q(\la)$ is strictly regular.
\end{proof}

\begin{ejem}\label{ex.partition}
  We now bring the extended illustrative example
  (started in Example~\ref{ex.intro} and continuing
  through Examples~\ref{ex.factor-counting}, \ref{ex.stacking}, and \ref{ex.homogeneous-partition})
  to a culmination,
  using the proof of Theorem~\ref{thm.QTR-strictly-regular}
  to complete the construction of a strictly regular, $k$-quasi-triangular realization
  (over the field $\bF = \bZ_2$)
  of the given finite spectral data from back in Example~\ref{ex.intro}.
  Recall that $k=4$, the target size is $6 \times 6$ with degree $10$,
  and the three irreducible divisors in the original spectral data
  are $\,\chi(\la) = \la^4 + \la^3 + 1$, $\,\phi(\la) = \la^2 + \la + 1$, and $\,\psi(\la) = \la$.

  In Example~\ref{ex.homogeneous-partition},
  we found a homogeneous partitioning $(12,8\mid12,8\mid11,9)$ %for $\,\deg \mathbf{p}\,$
  of diagonal degrees,
  and corresponding permutation of diagonal entries to give us
  the target diagonal
  \begin{equation} \label{eqn.target-diagonal}
   \bigl(\, \chi^2\phi\psi^2, \,\chi\phi\psi^2, \,\chi^2\phi\psi^2,
                   \,\chi\phi\psi^2, \,\chi^2\phi\psi \,, \,\chi\phi^2\psi \,\bigr)
  \end{equation}
  for the upper triangular matrix \eqref{eqn.tri-partition1}
%   (called $\wh{T}(\la)$ in the proof of Theorem~\ref{thm.QTR-strictly-regular})
  in our construction.
  Now the theory we have developed guarantees
  that the $S(\la)$ in Example~\ref{ex.intro}
  can be unimodularly transformed
  into an upper triangular $\wh{T}(\la)$ such that $\diag \wh{T}(\la)$
  is exactly the vector in \eqref{eqn.target-diagonal}.
  And furthermore, that this transformation can be implemented
  as a finite sequence of embedded $2 \times 2$ unimodular transformations
  acting only on adjacent diagonal entries.
  It would be very tedious to display all of these transformations,
  and the resulting upper triangular $\wh{T}(\la)$
  is very likely to have a densely populated upper triangular part.
  So instead, for ease of exposition
  we exhibit an alternative $\wh{T}(\la)$ with the desired diagonal vector
  that is not only sparse,
  but is also easily checked to be unimodularly equivalent to $S(\la)$.
  \[
    \wh{T}(\la) \,=\, \left[
                    \begin{array}{cc|cc|cc}
                      \chi^2\phi\psi^2 & \chi\phi\psi & 0 & 0 & 0 & 0 \\
                                     0 & \chi\phi\psi^2 & 0 & 0 & 0 & 1 \\
                         \hline \\[-12pt]
                                       & & \mystrut{3.7mm} \chi^2\phi\psi^2 & 0 & 0 & 0 \\
                                       & & 0 & \chi\phi\psi^2 & \phi\psi & 0 \\
                         \hline \\[-12pt]
                                       & & & & \mystrut{3.7mm} \chi^2\phi\psi & 0 \\
                                       & & & & 0 & \chi\phi^2\psi
                    \end{array} \right]
  \]
  Observe that this $\wh{T}(\la)$ has been partitioned into blocks
  as in \eqref{eqn.tri-partition1},
  in a manner that conforms to the partitioning of the diagonal vector \eqref{eqn.partition-example}
  arising from the homogeneous partitioning of the diagonal degrees.

% \begin{note}
%   \tcb{Did this $\wh{T}(\la)$ \emph{really} come from applying our procedures to $S(\la)$,
%        or was it constructed by RH using techniques from his thesis??
%        Or perhaps just made up ad hoc, on the fly??}
% \end{note}

Each of the diagonal blocks of $\wh{T}(\la)$ can now be transformed
into 2$\times$2 blocks having degree 10,
via simple unimodular transformations.
Define the unimodular matrices
\[
  U_1 := \mat{cc} 1 & \phi + 1 \\ 0 & 1 \rix\,,
    \quad
  U_2 := U_1\,,
    \quad
  U_3 := \mat{cc} 1 & \psi \\ 0 & 1 \rix\,,
    \quad
  V_1 := \mat{cc} 1 & 0 \\ \psi^2 & 1 \rix\,,
    \quad
  V_2 := V_1\,, \;\;\text{ and }
    \;\;
  V_3 := U_3^T\,.
\]
Since the underlying field here is $\bF = \bZ_2$, we then have
\begin{align*}
  U_1 \wh{T}_{11} V_1 \,=\, \chi \phi \psi \cdot
  \left[ \begin{array}{cc} 1 & \phi + 1 \\ 0 & 1 \end{array} \right]
  \left[ \begin{array}{cc} \chi \psi & 1 \\ 0 & \psi \end{array} \right]
  \left[ \begin{array}{cc} 1 & 0 \\ \psi^2 & 1 \end{array} \right]
  & \,=\; \chi \phi \psi \cdot \left[ \begin{array}{cc} \phi + 1 & \phi \psi + \psi + 1 \\
                                                        \psi^3 & \psi  \end{array} \right] \,,
     \\[3pt]
  U_2 \wh{T}_{22} V_2 \,=\, \chi \phi \psi^2 \cdot
  \left[ \begin{array}{cc}1 & \phi + 1 \\ 0 & 1 \end{array} \right]
  \left[ \begin{array}{cc} \chi & 0 \\ 0 & 1 \end{array} \right]
  \left[ \begin{array}{cc} 1 & 0 \\ \psi^2 & 1 \end{array} \right]
  & \,=\; \chi \phi \psi^2 \cdot \left[ \begin{array}{cc} 1 & \phi + 1 \\
                                                     \psi^2 & 1 \end{array} \right] \,,
     \\[3pt]
  U_3 \wh{T}_{33} V_3 \,=\, \chi \phi \psi \cdot
  \left[ \begin{array}{cc} 1 & \psi \\ 0 & 1 \end{array} \right]
  \left[ \begin{array}{cc} \chi & 0 \\ 0 & \phi \end{array} \right]
  \left[ \begin{array}{cc} 1 & 0 \\ \psi & 1 \end{array} \right]
  & \,=\; \chi \phi \psi \cdot \left[ \begin{array}{cc} \psi^2 + 1 & \phi \psi \\
                                                         \phi \psi & \phi \end{array} \right] \,,
\end{align*}
which are each readily seen to have degree $10$.

 Applying these transformations collectively to all of $\wh{T}(\la)$
via $U := \diag \{ U_1, U_2, U_3 \}$ and $V := \diag \{ V_1, V_2, V_3 \}$
produces the 2-quasi-triangular realization
\[
  P(\la) :=\, U \wh{T}(\la) V
          = \left[\begin{array}{cc|cc|cc}
                \chi\phi\psi (\phi + 1) & \chi\phi\psi(\phi \psi + \psi +1) & 0 & 0 & \psi(\phi + 1) & \phi + 1 \\
                \chi\phi\psi^4 & \chi\phi\psi^2 & 0 & 0 & \psi & 1 \\
            \hline \\[-12pt]
                     &  & \mystrut{3.7mm} \chi\phi\psi^2 & \chi\phi\psi^2(\phi + 1) & \phi\psi(\phi + 1) & 0 \\
                     &  & \chi\phi\psi^4 & \chi\phi\psi^2 & \phi\psi & 0 \\
            \hline \\[-12pt]
                     &  &  &  & \mystrut{3.7mm} \chi\phi\psi (\psi^2 + 1) & \chi\phi^2\psi^2 \\
                     &  &  &  &  \chi\phi^2\psi^2 & \chi\phi^2\psi
\end{array}\right].
\]
with diagonal blocks that are all of degree $10$.
Observe that all of the off-diagonal blocks of $P(\la)$
have degree strictly less than $10$.
Hence we can skip the final step
(i.e., using Lemma~\ref{LemOffDiagClean} to reduce the degrees of the off-diagonal blocks)
of the general procedure,
and declare that our final degree $10$, strictly regular,
$2$-quasi-triangular realization $Q(\la)$
is identical to $P(\la)$ in this example.

 Finally, note that there are at least two ways to see that our final realization $Q(\la)$
is indeed strictly regular.
One is by an index sum argument -- taking $\grade Q$ to be equal to $\deg Q = 10$,
we see that there is no room in the index sum constraint \eqref{ISMeq}
for any infinite partial multiplicities to be nonzero,
hence $Q(\la)$ must be strictly regular.
The second way is simply to examine the leading coefficient of $Q(\la)$
as a matrix polynomial, i.e., the matrix coefficient of $\la^{10}$.
That matrix is easily seen to be just $\diag (R, R, R)$,
where $R$ is the $2 \times 2$ matrix
$\bigl[ \begin{smallmatrix} 0 & 1 \\ 1 & 0 \end{smallmatrix} \bigr]$;
clearly this coefficient is nonsingular.
Hence $Q(\la)$ is strictly regular (see comments just after Definition~\ref{def.strictly-regular}).
\end{ejem}

%%%%%%%%%%%%%%%%%%%%%%%%%%%%%%%%%%%%%%%%%%%%%%%%%%%%%%%%%%%%%%%%%%
%%%%%%%%%%%%%%%%%%%%%%%%%%%%%%%%%%%%%%%%%%%%%%%%%%%%%%%%%%%%%%%%%%
% \section{Including an eigenvalue at infinity}
\section{Including infinite spectral data in a realization}
\label{sect.eigval-at-infinity}

In this section we extend the range of our quasi-triangular realization construction
to handle spectral data arising from a general regular matrix polynomial
--- i.e., data that may now include nontrivial infinite spectral structure.
The main tool that allows us to achieve this extension
are M\"obius transformations,
employing a now well-used technique;
see \cite{prescribed}, \cite{Mobius}, and \cite{TasTisZab}.
In a nutshell, the strategy of this technique is to
use an appropriately chosen M\"obius transformation
to translate a realization problem involving a mixture
of finite and infinite spectral data
into one that has \emph{only} finite spectral data,
solve this strictly regular realization problem,
and then translate the solution back to the original spectral data
using the inverse M\"obius transformation.

However, before we are able to implement this strategy,
it will be necessary to reexamine one of the fundamental properties
of M\"obius transformations from \cite{Mobius}, and recast %(revise, reconceive)
it in a way that enables us to work with these transformations
more smoothly in the context of arbitrary fields $\bF$,
in particular in the presence of higher degree $\bF$-irreducible divisors
in spectral data.
This will be the task of Section~\ref{subsect.Mobius-and-irreducible divisors}.
With this reimagined property of M\"obius transformations in hand,
in Section~\ref{subsect.quasi-tri-with-infinite-structure}
we then solve the quasi-triangular realization problem
with nontrivial infinite spectral structure.
This gives us the final result needed to now easily prove
in Section~\ref{subsect.quasi-triangularization}
the featured result of the paper,
the quasi-triangularization of arbitrary regular matrix polynomials.

Before embarking on this final stage, though,
it is worth pointing out several subtle issues
that arise in the course of carrying out this strategy.
The first concerns an extra hypothesis that is only needed
if the field $\bF$ is finite,
and even then only to exclude very special types of spectral data set.
When $\bF$ is finite, then it is possible that \emph{every} element of $\bF$,
as well as $\infty$, may appear in a given spectral data set as an eigenvalue.
If that happens,
then any possible M\"obius transformation
that one might attempt to use
will simply map the set of eigenvalues bijectively to itself,
and thus will not be able to transform the given spectral data
into one that has only finite spectral data.
To exclude this one problematic scenario,
it has been necessary to include an additional hypothesis,
i.e., that there must be some element of $\bF$
that is \emph{not} an eigenvalue in the given spectral data.
Of course for any infinite field $\bF$,
this hypothesis always holds,
and so has no impact on the range of spectral data sets
to which the argument applies.
But for finite fields, having this one additional condition satisfied
is sufficient to make the rest of the argument work smoothly.
(When this condition is violated,
 it is not known whether quasi-triangular realizations exist or not.)

A second issue concerns the reduction of off-diagonal degrees
to be less than the target degree, as was done in the strictly regular case.
When there is nontrivial spectral structure at $\infty$,
such a reduction may not be possible.
The difficulty is that this reduction is achieved via unimodular transformations,
which may alter the spectral structure at $\infty$.
% Also note that
Indeed, Example~\ref{ex.full-reg} in Section~\ref{subsect.quasi-tri-with-infinite-structure}
gives a concrete illustration
of how not only the infinite spectral structure,
but even the overall degree of the realization itself,
can be spoiled by doing the kind of degree reduction
of off-diagonal blocks described in Lemma~\ref{LemOffDiagClean}.

%%%%%%%%%%%%%%%%%%%%%%%%%%%%%%%%%%%%%%%%%%%%%%%%%%%%%%%%%%%%%%%%%
\subsection{M\"obius transformation of spectral structure over an arbitrary field}
   \label{subsect.Mobius-and-irreducible divisors}

We begin this section by reviewing the effect of M\"obius transformations
on the spectral structure of a matrix polynomial,
in the manner developed in~\cite{Mobius}.
In that paper the emphasis was on scalars in the ambient field $\bF$
(plus $\infty$) as potential eigenvalues,
and so the development was best adapted to algebraically closed fields.
However, when working with matrix polynomials over arbitrary fields,
with irreducible divisors of higher degree,
the formulation in~\cite{Mobius} can be rather inconvenient,
even a bit clumsy to use.
Thus our aim in this section is to reformulate the relationship
between partial multiplicity sequences and M\"obius transformations
in such a way that it works smoothly and efficiently over all fields,
and is well-adapted to our extended notion of partial multiplicity sequences
for irreducible divisors of any degree, as defined in \eqref{eqn.PM}.
We begin, though, with a brief review of the relevant
concepts, notation, and results from~\cite{Mobius},
in particular Theorem~5.3 from that paper.

If $P(\la)$ is a matrix polynomial over $\bF$,
 and $\mu_0 \in \bF_{\infty}$ is any scalar in the ``extended'' field
$\bF_{\infty} := \bF \cup \{ \infty \}$,
then in~\cite{Mobius} the partial multiplicity sequence
associated with $\mu_0$ was denoted by $\cJ (P, \mu_0)$.
Note that the letter $\cJ$ was chosen there because
the partial multiplicity sequences for $P$ for \emph{all} scalars $\mu_0 \in \bF_{\infty}$
was collectively termed the ``Jordan characteristic'' of $P$.
In the current paper, the partial multiplicity sequence $\cJ (P, \mu_0)$
is associated with the degree one irreducible polynomial $\la - \mu_0$,
rather than with the scalar $\mu_0$.
Thus we have the following equivalence between notations:
\begin{equation} \label{eqn.notation -conversion}
  \cJ (P, \mu_0) \;\equiv\; \PM (P, \la - \mu_0) \,.
\end{equation}
Continuing the review of \cite{Mobius} ---
to any nonsingular
$A = \bigl[ \begin{smallmatrix} a & b \\ c & d \end{smallmatrix} \bigr]$ over $\bF$
there is an associated \emph{M\"obius function}
$\mf{A} (\la) : \bF_{\infty} \rightarrow \bF_{\infty}$,
which is a bijection of scalars given by the formula
\[
  \mf{A} (\la) \,:=\; \frac{a \la + b}{\, c \la + d \,} \,,
\]
and a M\"obius transformation on matrix polynomials of grade $g$
given by
\[
  \MT_A \bigl( P(\la) \bigr) \,:=\; (c \la + d)^g \,P\bigl( \mf{A} (\la) \bigr) \,.
\]
Note that this is exactly the same as in Definition~\ref{Mobius}.
Now a key result from \cite{Mobius} establishes a simple relationship
between the Jordan characteristics of $P$ and $\MT_A(P)$.

\begin{theorem}[Theorem 5.3 from \cite{Mobius}] \label{thm.Thm-from-mobius-paper} \quad \\
  Suppose $P(\la)$ is an $\mbyn$ matrix polynomial of grade $g$ over an arbitrary field $\,\bF$,
  and $A \in GL(2, \bF)$ is a nonsingular matrix
  with associated M\"obius function $\mf{A} (\la)$
  and M\"obius transformation $\MT_A$.
  Then for any $\mu_0 \in \bF_{\infty}$,
  \begin{equation} \label{eqn.Thm5.3.1}
    \cJ \bigl( \MT_A (P), \mu_0 \bigr) \;=\; \cJ \bigl( P, \mf{A}(\mu_0) \bigr) \,,
  \end{equation}
  or equivalently,
  \begin{equation} \label{eqn.Thm5.3.2}
    \cJ \bigl( \MT_A (P), \mf{A^{-1}} (\mu_0) \bigr) \;=\; \cJ \bigl( P, \mu_0 \bigr) \,.
  \end{equation}
\end{theorem}
The formulas in Theorem~\ref{thm.Thm-from-mobius-paper} have one glaring drawback
when trying to capture the effect of a M\"obius transformation
on the complete spectral structure of a matrix polynomial $P$
over a field $\bF$ that is \emph{not} algebraically closed.
The problem is that the irreducible divisors of $P$
may all have degree $2$ or larger,
and thus have \emph{no} eigenvalues in $\bF_{\infty}$ at all.
In this scenario the formulas in Theorem~\ref{thm.Thm-from-mobius-paper}
tell us nothing at all about the spectral effects of the M\"obius transformation.
At least not directly.
The only way to recover any information about the spectral effects of $\MT_A$
would be to pass to the algebraic closure $\ov{\bF}$,
which though feasible,
% but the prospect of being repeatedly forced into this maneuver is certainly
may become an unwelcome annoyance.
It would be very useful to instead
have some analog of Theorem~\ref{thm.Thm-from-mobius-paper}
that applies directly to partial multiplicity sequences
associated with irreducible divisors of any degree,
i.e., to the $\PM (P, \chi)$ defined earlier in \eqref{eqn.PM}.
There is indeed such an analog,
and the goal of the remainder of this section is to establish
this extension of Theorem~\ref{thm.Thm-from-mobius-paper}.
We claim that the following relationship holds:

\begin{equation} \label{eqn.NewFormula}
  \PM \bigl( \MT_A(P), \MT_A(\chi) \bigr) \;=\; \PM (P, \chi) \,.
\end{equation}
To properly interpret this formula, however,
two conventions must be observed:
\begin{itemize}
  \item On the left-hand side of \eqref{eqn.NewFormula},
        $\MT_A(P)$ is taken with respect to the specified grade for $P$,
        but $\MT_A (\chi)$ should always be taken with \emph{$\grade \chi$ equal to $\deg \chi$},
        with just the one exception described next.
  \item If $\chi(\la) = \beta$ is any \emph{nonzero} constant,
        then $\chi$ is to be regarded as a \emph{grade one} polynomial,
        i.e., as $\alpha \la + \beta$ with $\alpha = 0$.
\end{itemize}
Note that the naive intuition underlying this second convention
is that the ``root'' of $0 \la + \beta$ is $\la = -\beta/0 = \infty$,
and so $0 \la + \beta$ can play the role of a ``polynomial stand-in''
for an eigenvalue at $\infty$.
Thus we will now use the notation $\PM(P, 0\la + \beta)$
to replace the earlier temporary notation $\PM(P, \infty)$
for the partial multiplicity sequence associated with an eigenvalue at $\infty$.
We will see in Theorem~\ref{thm.improved-Thm5.3} and its proof
that this now makes the formula \eqref{eqn.NewFormula}
internally consistent and universally applicable
to all partial multiplicity sequences,
both for finite and for infinite spectral structure.
Some additional motivation and justification for this
% (perhaps somewhat surprising?)
(perhaps unexpected?)
second convention is given in Remark~\ref{rem.commutative-diagram}.

\begin{remark} \label{rem.commutative-diagram}
\rm
Let $\cG_1$ denote the set of all \emph{nonzero} grade one scalar (i.e., $1 \times 1$) polynomials
$L(\la) = \alpha \la + \beta$ over $\bF$.
It is well known \cite{Mobius} that any M\"obius transformation $\MT_A$
defines a bijection $\cG_1 \rightarrow \cG_1$,
and the associated M\"obius function $\mf{A}$
defines a bijection $\bF_{\infty} \rightarrow \bF_{\infty}$
on the extended field of scalars $\bF_{\infty}$.
The purpose of this remark is to briefly explore the parallelism
between the bijection $\MT_A$ and the bijection $\mf{A^{-1}}$
(\emph{not} $\mf{A}$).

What underlies this parallelism is the simple calculation of $\MT_A (\la - \mu_0)$,
which also plays a central role in the proof of Theorem~\ref{thm.improved-Thm5.3}.
Using Definition~\ref{Mobius}, it is not hard to see that $\MT_A (\la - \mu_0)$
is (usually) a nonzero scalar multiple of $\la - \mf{A^{-1}}(\mu_0)$,
with the one possible exception of being just a nonzero scalar.
Now recall that in the context of discussing spectral data,
we are associating the degree one $\bF$-irreducible $(\la - \mu_0)$
with the eigenvalue $\mu_0$.
So we see that the mapping of degree one irreducibles by $\MT_A$
is mirrored by the mapping of eigenvalues in $\bF$ by $\mf{A^{-1}}$,
at least for most $\mu_0$.

To make this parallelism more precise,
let us define a map $\rho: \cG_1 \rightarrow \bF_{\infty}$
that sends $\alpha \la + \beta \in \cG_1$ with $\alpha \ne 0$
to its root $\frac{\, -\beta\,}{\alpha} \in \bF$,
i.e., define $\rho(\alpha \la + \beta) := \frac{\, -\beta\,}{\alpha}$.
This specifies $\rho$ on most of the desired domain $\cG_1$,
and enables us to express the parallelism in the diagram
\begin{equation} \label{eqn.commutative-diagram}
  \begin{CD}
    \cG_1         @>\MT_A>>   \cG_1 \\
    @V{\rho}VV                 @VV{\rho}V \\
    \bF_{\infty}  @>\mf{A^{-1}}>> \bF_{\infty}
  \end{CD} \qquad .
\end{equation}
 For all elements of $\cG_1$ with $\alpha \ne 0$ (with at most one exception),
we see that the diagram \eqref{eqn.commutative-diagram} commutes.
But what about elements of $\cG_1$ with $\alpha = 0$
(and hence $\beta \ne 0$), and the one possible exceptional case
when $\MT_A (\la - \mu_0)$ is just a nonzero scalar?
Is there a ``natural'' way to complete the definition of the map $\rho$
to the whole domain $\cG_1$,
and to do so in such a way
that \eqref{eqn.commutative-diagram} commutes?
There is indeed a unique way to achieve this,
and that is to define $\rho(0 \la + \beta) := \infty \in \bF_{\infty}$,
for any $\beta \ne 0$.
Thus we see from another direction
why it is natural to associate the grade $1$ polynomial $0 \la + \beta$
with an eigenvalue at $\infty$.
Note that the commuting of the diagram \eqref{eqn.commutative-diagram}
holds for any $A \in GL(2, \bF)$
and its associated M\"obius transformation and function.

The simplest (and perhaps most revealing) example of \eqref{eqn.commutative-diagram}
uses $\MT_R$ with $R = \bigl[ \begin{smallmatrix} 0 & 1 \\ 1 & 0 \end{smallmatrix} \bigr]$
% acting on $\cG_1$,
i.e., $\MT_R = \rev_1$.
In this case with $\alpha \ne 0$ and $\beta \ne 0$
we have $\MT_R (\alpha \la + \beta) = \beta \la + \alpha$
with reciprocal roots $\frac{\, -\beta\,}{\alpha}$ and $\frac{\, -\alpha\,}{\beta}$, respectively,
and also $\mf{R^{-1}} (\mu_0) = \mf{R} (\mu_0) = \frac{1}{\,\mu_0 \,}$
is the reciprocal map.
So clearly everything in the diagram commutes when $\alpha$ and $\beta$
are \emph{both} nonzero.
But if $\alpha=0$,
then with our convention we have $\MT_R(0 \la + \beta) = \beta \la$,
with roots $\infty$ and $0$,
which are commonly regarded in the literature as being ``reciprocal''.
Thus we see directly when $A=R$ that \eqref{eqn.commutative-diagram} commutes
for all elements of $\cG_1$.

It is also useful to note that \eqref{eqn.commutative-diagram} still commutes
if we pass to the equivalence classes discussed
earlier in Remark~\ref{rem.implicit-equiv-relation}.
That is, let us declare that two $\bF$-irreducible polynomials
are equivalent if one is a nonzero scalar multiple of the other,
i.e., $\chi \sim c \chi$ for any nonzero scalar $c \in \bF$.
Letting $\wt{\cG}_1$ denote the resulting set of equivalence classes
of grade one polynomials over $\bF$,
then both of the maps $\MT_A$ and $\rho$
respect these equivalence classes,
and so induce well-defined quotient maps $\wt{\MT}_A$ and $\wt{\rho}$,
which are both bijections.
This, then, gives us the following commutative diagram
in which all of the mappings are bijective.
\begin{equation} \label{eqn.commutative-diagram -with-bijections}
  \begin{CD}
    \wt{\cG}_1     @>\wt{\MT}_A>>   \wt{\cG}_1 \\
    @V{\wt{\rho}}VV                 @VV{\wt{\rho}}V \\
    \bF_{\infty}  @>\mf{A^{-1}}>> \bF_{\infty}
  \end{CD}
\end{equation}

 Finally, note that with our conventions, the scenario $\MT_A(\chi) = \beta$
\emph{cannot} arise in \eqref{eqn.NewFormula}
if the $\bF$-irreducible $\chi$ has $\deg (\chi) \ge 2$;
it can only occur if $\chi \in \cG_1$.
This is not obvious, but follows from a property to be proved
in Lemma~\ref{lem.degree-preservation},
namely that for any $\bF$-irreducible $\chi$ with $\deg (\chi) \ge 2$,
and $\MT_A(\chi)$ taken with respect to degree,
then $\deg \MT_A (\chi) = \deg (\chi)$.
\end{remark}

Before proceeding to prove \eqref{eqn.NewFormula} in Theorem~\ref{thm.improved-Thm5.3},
we first establish some preliminary lemmas.
The first of these lemmas appeared in \cite{Mobius} as Corollary 3.24a,
but we recall it here together with its simple proof
for the convenience of the reader.

\begin{lemma}[Basic product property of M\"obius transformations]
              \label{lem.product-property} \quad \\
  Let $p, q \in \bF[\la]$ be nonzero scalar polynomials,
  and let $A = \bigl[ \begin{smallmatrix} a & b \\ c & d \end{smallmatrix} \bigr] \in GL(2, \bF)$.
  Letting the grades of $p$, $q$, and $pq$ all be chosen to be equal to their degrees,
  then we have $\MT_A (pq) = \MT_A(p) \MT_A(q)$.
\end{lemma}
\begin{proof}
  Let $d_1 = \deg p$ and $d_2 = \deg q$.
  Then from Definition~\ref{Mobius}  we have
  \begin{align*}
    \MT_A (pq) &\;=\; (c \la + d)^{d_1 + d_2} \,p\bigl( \mf{A} (\la) \bigr) \,q\bigl( \mf{A} (\la) \bigr) \\
               &\;=\; (c \la + d)^{d_1} \,p\bigl( \mf{A} (\la) \bigr)
                     \cdot (c \la + d)^{d_2} \,q\bigl( \mf{A} (\la) \bigr)
                \;=\; \MT_A(p) \, \MT_A(q) \,.
  \end{align*}
\end{proof}

The next result was mentioned in Section~\ref{pre}
under the name Lemma~\ref{lem.degree-preservation-without-proof}.
We recall it here and provide the proof postponed from earlier.

\begin{lemma} \label{lem.degree-preservation}
  Suppose $\bF$ is an arbitrary field,
  and $\chi(\la)$ is any $\bF$-irreducible scalar polynomial with $\deg \chi \ge 2$.
  Let $\MT_A$ be the M\"obius transformation
  associated with any
  $A = \bigl[ \begin{smallmatrix} a & b \\ c & d \end{smallmatrix} \bigr] \in GL(2, \bF)$.
  Then with $\grade(\chi)$ taken to be equal to $\deg(\chi)$,
  the transformation $\MT_A$ preserves both the degree and the $\bF$-irreducibility of $\chi$.
  That is, $\MT_A (\chi)$ is $\bF$-irreducible,
  and $\,\deg \bigl(\MT_A (\chi) \bigr) = \deg(\chi)$.
\end{lemma}

\begin{proof}
  That $\MT_A$ preserves $\bF$-irreducibility was shown in~\cite[Corollary~3.24]{Mobius}.
  To see why degree is preserved, we look at $\chi(\la)$
  over the algebraic closure $\ov{\bF}$ of $\bF$.
  Over $\ov{\bF}$ we may factor $\chi$ completely into linear factors,
  i.e., $\,\chi(\la) \,=\, k (\la - r_1)(\la - r_2) \dotsb (\la - r_d)$,
  where $d = \deg \chi \ge 2$ and $k \in \bF$.
  Now observe that for $\chi$ to be $\bF$-irreducible,
  \emph{all} of these roots $r_i$ must be in $\ov{\bF} \setminus \bF$,
  otherwise $\chi$ would be $\bF$-reducible.
  Computing $\MT_A (\chi)$ with grade equal to degree $d$,
  and using the multiplicative property of M\"obius transformations
  in Lemma~\ref{lem.product-property},
  we obtain
  \[
    \MT_A (\chi) \;=\; k \,\MT_A (\la - r_1) \cdot \MT_A (\la - r_2) \dotsb \MT_A (\la - r_d) \,,
  \]
  where each of these transformations $\MT_A (\la - r_j)$ is computed with respect to grade $1$.
  Using the definition, we easily see that
  \begin{equation} \label{eqn.Mobius-of-linear}
    \MT_A(\la - r_j) \,=\, (a - c r_j) \la \,+\, (b - d r_j) \,.
  \end{equation}
  Now consider two cases: $c=0$ and $c \ne 0$.
  If $c=0$, then since $A$ is nonsingular we must have $a \ne 0$,
  so each of the factors $\MT_A(\la - r_j)$ will have degree one,
  and $\MT_A (\chi)$ will have degree $d = \deg \chi$.
  On the other hand, if $c \ne 0$,
  then once again each of these $\MT_A(\la - r_j)$ will have degree one,
  and hence $\deg \bigl( \MT_A (\chi) \bigr) = d = \deg \chi$,
  unless one of the roots $r_j$ satisfies $a - c r_j = 0$,
  i.e., $r_j = \tfrac{a}{c}$.
  This, however, is impossible,
  since $\tfrac{a}{c} \in \bF$, but $r_j \in \ov{\bF} \setminus \bF$.
\end{proof}

% \begin{lemma} %\label{lem.no-common-roots}
%   Suppose $p,q \in \bF[\la]$ are distinct $\bF$-irreducible polynomials.
%   Then as polynomials over the algebraic closure $\,\ov{\bF}$,
%   $p$ and $q$ have no roots in common.
% \end{lemma}
% \begin{note}
%   \tcb{Should this lemma be rewritten to say that $p,q \in \bF[\la]$
%        are coprime in $\bF[\la]$ if and only if they are coprime in $\ov{\bF}[\la]$??
%        This is really what the argument actually proves.}
% \end{note}
% \begin{proof}
%   As polynomials over $\bF$, $p$ and $q$ have no nontrivial factors of any degree in common,
%   since they each individually have no nontrivial factors at all.
%   Thus $p$ and $q$ are coprime, and satisfy the Bezout identity $pr + qt \equiv 1$,
%   for some polynomials $r, t \in \bF[\la]$.
%   This identity still holds for $p, q, r, t$ viewed as polynomials over $\ov{\bF}$,
%   so $p$ and $q$ are coprime over $\ov{\bF}$,
%   and thus have no common roots in $\ov{\bF}$.
% \end{proof}

\begin{lemma} \label{lem.no-common-roots}
  Suppose $\,\bF$ is an arbitrary field,
  and $\,\wh{\bF} \supseteq \bF$ is any field extension of $\,\bF$.
  Then polynomials $p,q \in \bF[\la]$ are coprime in $\,\bF[\la]$
  if and only if they are coprime in $\,\wh{\bF}[\la]$.
  More specifically, if $p,q \in \bF[\la]$ are distinct $\bF$-irreducible polynomials,
  then as polynomials over the algebraic closure $\,\ov{\bF}$,
  $p$ and $q$ have no roots in common.
\end{lemma}
\begin{proof}
  Suppose $p$ and $q$ are coprime in $\,\bF[\la]$.
  Then $p$ and $q$ satisfy the Bezout identity $pr + qt \equiv 1$,
  for some polynomials $r, t \in \bF[\la]$.
  This identity still holds for $p, q, r, t$ viewed as polynomials over $\wh{\bF}$,
  so $p$ and $q$ are coprime in $\wh{\bF}[\la]$.
  Conversely, if $p$ and $q$ are coprime in $\wh{\bF}[\la]$,
  then they have no nontrivial common factor in $\wh{\bF}[\la]$,
  and hence no nontrivial common factor in the smaller ring $\bF[\la]$,
  and thus are coprime in $\bF[\la]$.
  The statement about $\bF$-irreducible polynomials $p$ and $q$ now follows immediately,
  since \emph{distinct} $\bF$-irreducibles are necessarily coprime in $\bF[\la]$.
\end{proof}

We are now in a position to prove the main result of this section,
i.e., the formula that describes the effect of M\"obius transformations
on spectral structure.
It is important to emphasize that this formula holds in complete generality
for \emph{all} matrix polynomials
(regular or singular, of any size, over arbitrary fields),
for \emph{all} irreducible divisors (of any degree),
and for \emph{all} M\"obius transformations $\MT_A$,
whatever the underlying (nonsingular) matrix $A$.
Note also that the proof makes extensive use
of the previously-known formulas
in Theorem~\ref{thm.Thm-from-mobius-paper}.

\begin{theorem}[Effect of M\"obius on spectral structure over arbitrary fields]
  \label{thm.improved-Thm5.3} \quad \\
  Let $P(\la)$ be any grade $g$ matrix polynomial over $\,\bF$,
  where $\bF$ is an arbitrary field.
  Also let $A = \bigl[ \begin{smallmatrix} a & b \\ c & d \end{smallmatrix} \bigr] \in GL(2, \bF)$
  be any nonsingular matrix over $\bF$,
  with associated M\"obius transformation $\MT_A$.
  Then for any $\bF$-irreducible scalar polynomial $\chi(\la)$,
  including the grade one polynomial $(0@\la + \beta)$ with $\beta \ne 0$,
  we have
  \begin{equation} \label{eqn.improved-Mobius-relation}
    \PM \bigl( \MT_A(P), \MT_A(\chi) \bigr) \;=\; \PM (P, \chi) \,.
  \end{equation}
  \parens{Here $\MT_A(P)$ is taken with respect to grade $g$,
   while each $\MT_A(\chi)$ is taken with grade equal to $\deg \chi$,
   with the sole exception of the grade one
   $\chi(\la) = @0@\la + \beta$, with $\beta \ne 0$.}
\end{theorem}
\begin{proof}
  We consider three cases,
  depending on the nature of the $\bF$-irreducible $\chi(\la)$ involved.
  In all cases we have
  $A = \bigl[ \begin{smallmatrix} a & b \\ c & d \end{smallmatrix} \bigr] \in GL(2, \bF)$,
  and $A^{-1} = \frac{1}{\,\det A \,} \bigl[ \begin{smallmatrix} d & -b \\ -c & a \end{smallmatrix} \bigr]$.
  In the following calculations, keep in mind that we are free
  (because of \eqref{eqn.scalar-multiples})
  to alter $\bF$-irreducibles by (nonzero) scalar multiples, whenever convenient.
  We will also make repeated use of the formula \eqref{eqn.Thm5.3.2}
  from Theorem~\ref{thm.Thm-from-mobius-paper},
  as well as the notation conversion formula \eqref{eqn.notation -conversion},
  without ever explicitly mentioning that they are being used.
\begin{enumerate}
  \item[(a)] [\,$\chi = 0 \la + \beta$, with $\beta \ne 0$ and $\grade(\chi) = 1$\,] \quad
             Given the above conventions, we have
    \begin{align*}
      \PM(P, \chi) \,=\, \cJ (P, \infty)
                     &\,=\, \cJ \bigl( \MT_A(P), \mf{A^{-1}}(\infty) \bigr) \\
                     & \,=\, \cJ \bigl(\MT_A(P), \tfrac{d}{\,-c\,} \bigr)
                       \,=\, \PM \bigl(\MT_A(P), \,c \la + d \bigr)
                       \,=\, \PM \bigl(\MT_A(P), \,\MT_A(\chi) \bigr) \,,
    \end{align*}
             as desired.
  \item[(b)] [\,$\chi = \la - \mu_0$, with $\mu_0 \in \bF$ and $\grade(\chi) = \deg(\chi) = 1$\,] \\[2pt]
             Recall from \eqref{eqn.Mobius-of-linear}
             that $\MT_A(\la - \mu_0) = (a - c \mu_0) \la \,+\, (b - d \mu_0)$.
             Splitting this into two cases, depending on whether $a - c \mu_0$ is zero or not,
             we have
        \[
          \MT_A(\chi) \,=\,
          \begin{cases}
            (a - c \mu_0) \bigl[\, \la \,-\, \mf{A^{-1}} (\mu_0)  \,\bigr]
                     & \text{ if $\,a - c \mu_0 \ne 0$\,}, \\
            (b - d \mu_0) & \text{ if $\,a - c \mu_0 = 0$\,}.
          \end{cases}
        \]
        So if $a - c \mu_0 \ne 0$, then
    \begin{align*}
      \PM (P, \chi) \,=\, \cJ (P, \mu_0)
                     &\,=\, \cJ \bigl( \MT_A(P), \mf{A^{-1}}(\mu_0) \bigr) \\
                     & \,=\, \PM \bigl(\MT_A(P), \la - \mf{A^{-1}}(\mu_0) \bigr) %\\
%                      & \,=\, \PM \bigl(\MT_A(P), \,\MT_A (\la - \mu_0) \bigr)
                       \,=\, \PM \bigl(\MT_A(P), \,\MT_A(\chi) \bigr) \,.
    \end{align*}
             On the other hand, if $a - c \mu_0 = 0$ (i.e., $\mu_0 = \frac{a}{\,c\,} = \mf{A}(\infty)$),
             then $b - d \mu_0 \ne 0$ because of the nonsingularity of $A$,
             and we have
    \begin{align*}
      \PM (P, \chi) \,=\, \cJ (P, \mu_0)
                     &\,=\, \cJ \bigl( \MT_A(P), \mf{A^{-1}}(\mu_0) \bigr) \\
                     & \,=\, \cJ \bigl(\MT_A(P), \infty \bigr)
                       \,=\, \PM \bigl(\MT_A(P), \,0\la + (b - d\mu_0) \bigr)
                       \,=\, \PM \bigl(\MT_A(P), \,\MT_A(\chi) \bigr) \,.
    \end{align*}
  \item[(c)] [\,$\grade(\chi) = \deg(\chi) \ge 2$\,] \\[2pt]
             In this argument we will go back and forth between $\bF$
             and its algebraic closure $\ov{\bF}$,
             exploiting the fact that $P$, $\MT_A$, and $\chi$ can be viewed
             either as objects over $\bF$ or as objects over $\ov{\bF}$.

             Viewing $\chi(\la)$ as a polynomial in $\ov{\bF}[\la]$,
             we may factor it completely into linear factors
             \begin{equation} \label{eqn.factored-chi}
               \chi(\la) \,=\, \beta (\la - r_1)^{\delta_1}(\la - r_2)^{\delta_2}
                                    \dotsb (\la - r_k)^{\delta_k} \,, \quad 0 \ne \beta \in \bF \,,
             \end{equation}
             where the roots $r_i$ are \emph{distinct}, and all in $\ov{\bF} \setminus \bF$.
             (See proof of Lemma~\ref{lem.degree-preservation}.)
%              Note that we must have $k \ge 2$ here,
%              otherwise $\chi$ would not be a polynomial over $\bF$.
%              \begin{note}
%                \tcb{Is it really correct that $k \ge 2$?
%                     Or does it depend on $\bF$ having characteristic zero?
%                     And does it really matter for the proof anyway?
%                     I am starting to think that it is irrelevant for the proof.
%                     Is that correct??}
%              \end{note}
             Denoting the $j^{\text{th}}$ partial multiplicity of $\chi$ in $P$
             by $\alpha_j := \bigl[@@\PM (P, \chi)@ \bigr]_j$,
             so that the the $j^{\text{th}}$ invariant polynomial $s_j$ of $P$
             can be expressed as
             \begin{equation} \label{eqn.jth-invar-poly}
               s_j(\la) \,=\, \chi^{\alpha_j} \wt{s}_j(\la)
             \end{equation}
             with $\chi$ coprime to $\wt{s}_j$,
             we can now easily see that over $\ov{\bF}$ we have
             $\bigl[@@\PM (P, \la - r_i)@ \bigr]_j = \delta_i \alpha_j$,
             for each $1 \le i \le k$.
             This follows from \eqref{eqn.factored-chi} and \eqref{eqn.jth-invar-poly},
             together with the fact that \emph{no additional copies} of $(\la - r_i)$
             can arise from $\wt{s}_j$,
%              any of the other $\bF$-irreducible divisors
%              of the invariant polynomial $s_j(\la)$,
             by Lemma~\ref{lem.no-common-roots}.
             Applying the result of part (b) of this proof (over the field $\ov{\bF}$),
             we now have that
             \begin{equation} \label{eqn.jth-partial-multiplicities}
               \bigl[@@\PM (\MT_A(P), \MT_A \bigl(\la - r_i \bigr)@ \bigr]_j
                          = \bigl[@@\PM (P, \la - r_i)@ \bigr]_j = \delta_i \alpha_j \,,
             \end{equation}
            for all $1 \le i \le k$ and $1 \le j \le \rank P$.
            Note also from part(b) that
            $\MT_A \bigl(\la - r_i \bigr) = (a - c r_i) \bigl[\, \la \,-\, \mf{A^{-1}} (r_i)  \,\bigr]$
            is always degree one, since $a - cr_i \ne 0$
            ($a, c \in \bF$, but $r_i \in \ov{\bF} \setminus \bF$).
            (Although it is not strictly needed for this argument,
             it is also helpful to keep in mind that the transformation $\MT_A$ is \emph{bijective}
             on the set of %$\grade = 1$
             grade one polynomials over $\ov{\bF}$,
             so that $\MT_A(\la - r_i)$ is distinct from any other $\MT_A(\la - \mu_0)$.)

             Gathering together all of the partial multiplicities in \eqref{eqn.jth-partial-multiplicities},
             we now see that the $j^{\text{th}}$ invariant polynomial of $\MT_A(P)$
             can be written, over $\ov{\bF}$, as
             \begin{align*}
               m_j(\la) &\,=\, \bigl[@@ \MT_A(\la - r_1)@ \bigr]^{\delta_1 \alpha_j}
                              \,\bigl[@@ \MT_A(\la - r_2)@ \bigr]^{\delta_2 \alpha_j}
                          \dotsb \bigl[@@ \MT_A(\la - r_k)@ \bigr]^{\delta_k \alpha_j}
                          \cdot \wh{m}_j (\la) \\
                        &\,=\, \Bigl[@@ \bigl[@@\MT_A(\la - r_1)@ \bigr]^{\delta_1}
                                        \bigl[@@\MT_A(\la - r_2)@ \bigr]^{\delta_2}
                                 \dotsb \bigl[@@\MT_A(\la - r_k)@ \bigr]^{\delta_k} \Bigr]^{\alpha_j}
                          \cdot \wh{m}_j (\la) \,,
             \end{align*}
             where $\wh{m}_j (\la)$ is coprime in $\ov{\bF}[\la]$
             to each of the degree one factors $\MT_A(\la - r_i)$.
             But by the product property of M\"obius transformations
             in Lemma~\ref{lem.product-property}
             we also have
             \begin{multline*}
                  \bigl[@@ \MT_A(\la - r_1)@ \bigr]^{\delta_1}
                  \bigl[@@ \MT_A(\la - r_2)@ \bigr]^{\delta_2}
           \dotsb \bigl[@@ \MT_A(\la - r_k)@ \bigr]^{\delta_k} \\
                 \,=\,
                  \MT_A\bigl[(\la - r_1)^{\delta_1} \bigr]
                \,\MT_A\bigl[(\la - r_2)^{\delta_2} \bigr]
           \dotsb \MT_A\bigl[(\la - r_k)^{\delta_k} \bigr] \\
                 \,=\,
                  \MT_A\bigl[(\la - r_1)^{\delta_1} (\la - r_2)^{\delta_2}
                      \dotsb (\la - r_k)^{\delta_k} \bigr]
                 \,=\, \MT_A \bigl[ \chi(\la) \bigr] \; / \;\beta \,.
             \end{multline*}
             So the invariant polynomial $m_j(\la)$ for $\MT_A(P)$,
             which is the same over $\bF$ or over $\ov{\bF}$
             (since Smith forms are invariant under field extensions),
             can be expressed as
             \[
               m_j(\la) \,=\, \Bigl[@ \MT_A \bigl[ \chi(\la) \bigr] @\Bigr]^{\alpha_j} \wt{m}_j(\la) \,,
             \]
             where $\MT_A \bigl[ \chi(\la) \bigr]$ and $\wt{m}_j (\la)$
             are coprime in $\ov{\bF}[\la]$.
             Since $\MT_A \bigl[ \chi(\la) \bigr]$, and hence also $\wt{m}_j (\la)$,
             is a polynomial over $\bF$,
             by Lemma~\ref{lem.no-common-roots} they are also coprime in $\bF[\la]$.
             Also we know from Lemma~\ref{lem.degree-preservation}
             and the discussion surrounding it that $\MT_A \bigl[ \chi(\la) \bigr]$
             is an $\bF$-irreducible polynomial of the same degree as $\chi(\la)$.
             Thus it is well-defined to speak of the partial multiplicities
             of $\MT_A(P)$ at $\MT_A(\chi)$,
             and we see that over $\bF$ we have
             \[
               \bigl[@@\PM \bigl( \MT_A(P), \MT_A ( \chi ) \bigr) @ \bigr]_j
               \,=\, \alpha_j
               \,=\, \bigl[@@\PM (P, \chi)@ \bigr]_j
             \]
             for each $1 \le j \le \rank P$, i.e., for each individual partial multiplicity.
             Hence we also have the equality
             \[
               \PM \bigl( \MT_A(P), \MT_A (\chi) \bigr)
               \,=\, \PM (P, \chi)
             \]
             for the whole partial multiplicity sequence,
             as desired.
\end{enumerate}
\end{proof}

%%%%%%%%%%%%%%%%%%%%%%%%%%%%%%%%%%%%%%%%%%%%%%%%%%%%%%%%%%%%%%%%%
%%%%%%%%%%%%%%%%%%%%%%%%%%%%%%%%%%%%%%%%%%%%%%%%%%%%%%%%%%%%%%%%%

\subsection{Quasi-triangular realization with infinite spectral structure}
   \label{subsect.quasi-tri-with-infinite-structure}

With the new tools for working with M\"obius transformations
developed in Section~\ref{subsect.Mobius-and-irreducible divisors},
we are now able to show how to construct a quasi-triangular realization
for spectral data that may now include nontrivial structure at $\infty$.
An example to illustrate this construction
follows immediately after the theorem.

\begin{theorem}[Quasi-Triangular Realization with Eigenvalue at $\infty$]
   \label{thm.QTR-regular} \quad \\
    Let $\bF$ be an arbitrary field,
 	and suppose a list of $m$ invariant polynomials $s_1(\la),\dots s_m(\la)$ over $\,\bF$
 	forming a divisibility chain $s_1(\la) \vert s_2(\la) \vert \dotsb \vert s_m(\la)$ is given,
 	together with a nonempty list of $\ell$ nonzero partial multiplicities at infinity
 	$\,\alpha_1 \le \alpha_2 \le \dotsb \le \alpha_\ell$.
 	Let
 	\[
 	  \sigma \,:=\; \sum_{i=1}^m \deg \bigl( s_i(\lambda) \bigr)
 	                    \,+\, \sum_{j=1}^{\ell} \alpha_j
    \]
 	be the index sum for this data,
 	and define $k$ to be the maximum degree
 	among all of the $\,\bF$-irreducible divisors of $s_m(\la)$.
%  	of the invariant polynomials $s_i(\lambda)$ for $i = 1,\dotsc, m$.
 	Suppose also that the field $\,\bF$ contains some scalar $\,\omega \in \bF$
 	such that $s_m(\omega) \ne 0$. %for $i = 1,\dotsc, m$.
 	Then for any choice of nonzero $g, n \in \bN$
 	such that $n \ge \max\{m, \ell\}$ and $gn = \sigma$,
 	there exists an $n\times n$, grade $g$ matrix polynomial $Q(\lambda)$ over $\,\bF$
 	that is $k$-quasi-triangular,
 	has exactly the given invariant polynomials
 	$s_1(\lambda),\dots, s_m(\lambda)$
 	with all other $n-m$  invariant polynomials equal to $1$,
 	and partial multiplicity sequence at infinity
 	$(0, \dotsc, 0, \alpha_1, \dots, \alpha_\ell)$.
 	In addition, %if $n>\ell$, then
 	$\deg Q(\la) = g$ if and only if $n > \ell$.
 \end{theorem}
 \begin{proof}
   Begin by taking the given spectral data and converting it
   into the ``first form'' described in Definition~\ref{FiniteStructure},
   i.e., into a list of $\bF$-irreducible divisors $\chi_j$ (for $j = 1, t$),
   each equipped with a partial multiplicity sequence $\PM_{\text{given}}(\chi_j)$ of length $n$.
   Recall that the partial multiplicities at $\infty$
   are recorded as $\PM_{\text{given}}(\beta)$ for some (any) nonzero $\beta \in \bF$.
   Of course, doing this conversion
   may require adjoining some additional initial trivial invariant polynomials,
   and also perhaps some additional initial zero partial multiplicities at $\infty$,
   as needed to fill out the length $n$.

   Next we design some M\"obius transformations $\MT_A$ that will interchange
   the spectral roles of $\infty$ and the ``special'' scalar $\omega \in \bF$
   that has been assumed to exist.
   (The ``special'' property of $\omega$ is that $\PM (\la - \omega) = (0, 0, \dotsc, 0)$.)
   We claim that the $\MT_A$ defined by the nonsingular matrix
   \[
     A \,=\,
       \begin{cases}
         \bigl[ \begin{smallmatrix} \omega & 0 \\ 1 & -\omega \end{smallmatrix} \bigr]
             & \text{\, if \;$\omega \ne 0$} \,, \\[3pt]
         \;\bigl[ \begin{smallmatrix} 0 & \;1 \\ 1 & \;0 \end{smallmatrix} \bigr]
             & \text{\, if \;$\omega = 0$} \,,
       \end{cases}
   \]
   achieves this goal.
   As a mapping on grade one scalar polynomials, this $\MT_A$ is a bijection,
   and it is straightforward to check that
   \[
     \MT_A(\la - \omega) \,=\, \begin{cases}
                                 \omega^2 & \text{\, if \;$\omega \ne 0$} \\
                                    1  &  \text{\, if \;$\omega = 0$}
                               \end{cases}
        \qquad\text{ and }\qquad
     \MT_A(\beta) \,=\, \begin{cases}
                          \beta (\la - \omega) & \text{\, if \;$\omega \ne 0$} \\
                              \beta \la = \beta (\la - \omega) &  \text{\, if \;$\omega = 0$}
                        \end{cases} \quad,
   \]
   for any nonzero $\beta \in \bF$.
   Thus we see that for any $\omega$, the transformation $\MT_A$ effects an interchange of the roles
   of $\omega$ and $\infty$ as eigenvalues,
   while all other finite scalars $\mu \in \bF$
   (represented by degree one polynomials $\la - \mu$)
   are simply permuted in some fashion,
   the details of which are irrelevant here. % in this context.

   Using this $\MT_A$,
   we now define a new collection of spectral data that has no eigenvalue at $\infty$.
   For each irreducible divisor $\chi_j$ in the original data,
   we replace it by $\MT_A(\chi_j)$,
   but assign to it the same partial multiplicity sequence
   that $\chi_j$ has in the given spectral data.
   (Here each $\MT_A(\chi_j)$ is taken with grade equal to degree.)
   In other words, we declare that
   \[
      \PM \bigl( \MT_A (\chi_j) \bigr) \,:=\; \PM_{\text{given}} (\chi_j) \,.
   \]
   By Lemma~\ref{lem.degree-preservation} we know that each $\MT_A (\chi_j)$
   is $\bF$-irreducible with $\deg \MT_A (\chi_j) = \deg \chi_j$.

   Similarly, for the eigenvalue at $\infty$ in the original given data,
   we replace it by $\MT_A (0 \la + 1) = \la - \omega$,
   and assign to it the partial multiplicity sequence that $\infty$
   has in the given spectral data.
   (Here $0 \la + 1$ is viewed as a grade one polynomial.)
   In other words, we declare that
   \[
      \PM (\la - \omega) \;=\; \PM \bigl( \MT_A (0 \la + 1) \bigr)
                         \,:=\; \PM_{\text{given}} (0 \la + 1) \,.
   \]
   Thus we have specified a new collection $\cC$ of purely finite spectral data,
   that is, a list of $\bF$-irreducible divisors
   $\bigl\{ \MT_A(\chi_1), \MT_A(\chi_2), \dotsc , \MT_A(\chi_t),
            \MT_A(0 \la + 1)= (\la-\omega) \bigr\}$,
   together with assigned partial multiplicity sequences for each.
   Since the irreducible divisors in $\cC$ have the same degrees
   as their partners from the original spectral data
   (the same grade in the case of partners $0 \la + 1$ and $\la - \omega$),
   as well as the same partial multiplicities,
   then the index sum $\sigma$ is the same for $\cC$
   as it was for the original data.
   Thus we may use the same values of $n$, $g$, and $k$ for $\cC$
   as was used for the original data.

   Now by Theorem~\ref{thm.QTR-strictly-regular} there exists an $n\times n$,
   degree $g$ matrix polynomial $P(\la)$ over $\bF$
   that is $k$-quasi-triangular,
   and has exactly the spectral data in $\cC$.
   In other words,
   \begin{align}
     \PM \bigl( P(\la), \MT_A (\chi_j) \bigr) &\,=\; \PM_{\text{given}} (\chi_j) \,,
          \label{eqn.PM1} \\
     \text{and} \qquad
     \PM \bigl( P(\la), \MT_A (0 \la + 1) \bigr) &\,=\; \PM_{\text{given}} (0 \la + 1) \,.
          \label{eqn.PM2}
   \end{align}
   Applying Theorem~\ref{thm.improved-Thm5.3}
   using the M\"obius transformation $\MT_{A^{-1}}$ (with respect to grade $g$)
   to \eqref{eqn.PM1} and \eqref{eqn.PM2},
   and defining $Q(\la) := \MT_{A^{-1}}(P)$, we have
   \begin{align*}
     \hspace*{12mm}
     \PM \bigl( Q(\la) , \chi_j \bigr) \,=\,
     \PM \bigl( \MT_{A^{-1}}(P) , \chi_j \bigr) &\,=\,
     \PM \Bigl( \MT_{A^{-1}} \bigl( P \bigr), \MT_{A^{-1}} \bigl(\MT_A (\chi_j) \bigr) \Bigr) \\
         &\,=\, \PM \bigl( P(\la), \MT_A (\chi_j) \bigr)
                     \,=\; \PM_{\text{given}} (\chi_j) \,, \hspace*{15mm}
   \end{align*}
   and
   \begin{multline*}
     \PM \bigl( Q(\la) , (0 \la + 1) \bigr) \,=\,
     \PM \bigl( \MT_{A^{-1}}(P) , (0 \la + 1) \bigr) \,=\,
     \PM \Bigl( \MT_{A^{-1}} \bigl( P \bigr), \MT_{A^{-1}} \bigl(\MT_A (0 \la + 1) \bigr) \Bigr) \\
         \,=\, \PM \bigl( P(\la), \MT_A (0 \la + 1) \bigr)
                     \,=\; \PM_{\text{given}} (0 \la + 1) \,, \hspace*{20mm}
   \end{multline*}
%    \begin{equation*}
%      \PM \bigl( P(\la), \MT_A (0 \la + 1) \bigr) \,=\; \PM_{\text{given}} (0 \la + 1) \,.
%    \end{equation*}
   Thus the matrix polynomial $Q(\la)$
   is the desired quasi-triangular realization of the given spectral data.
   The relation between the degree and the grade of $Q(\la)$ follows
   immediately from Lemma~\ref{lem.grade-equal-degree}.
 \end{proof}

\begin{ejem}\label{ex.full-reg}
Consider the following irreducible divisors
along with their given partial multiplicity sequences:
\begin{center}
\begin{tabular}{l|c}
  Irreducible divisor & $\PMg$ \\[1pt]
     \hline
  $\eta(\la) := \la^4 + \la + 1$ \mystrut{4.1mm} &  $\;(0,0,1,2,3,3)$ \\[1pt]
  $\phi(\la) = \la^2+ \la +1$ &  $\;(0,1,1,1,1,3)$ \\[1pt]
  $\infty$, i.e.,\,``$0\la + 1$''  &  $\;(0,1,1,2,2,4)$
\end{tabular}
\end{center}
with index sum $60$.
(Recall that the grade one matrix polynomial $0\la + 1$ is our stand-in for $\infty$,
 so the third partial multiplicity sequence records the desired infinite spectral structure.)
Thus we may legitimately choose to seek a quasi-triangular realization of this data
with grade $g=10$ and size $n=6$.

The first step in constructing such a realization
is to translate the data using
an appropriate M\"{o}bius transformation.
Since $0$ is \emph{not} an eigenvalue in the given spectral data,
we may take $\omega = 0$ as our ``special'' value in the field $\bF = \bZ_2$,
which will be the recipient of the spectral data at $\infty$.
The M\"obius transformation(s) $\MT_R$ given by the matrix
$R =\bigl[ \begin{smallmatrix} 0 & \;1 \\ 1 & \;0 \end{smallmatrix} \bigr]$,
i.e., reversal,
% $ R = \left[\begin{array}{cc}0 & 1 \\ 1 & 0 \end{array}\right] $,
achieves this goal.
Note that this reversal is taken with respect to grade equalling degree
when applied to the irreducible divisors (with the exception of the grade \emph{one} $0\la + 1$),
and $\rev_g$ with $g=10$ when applied to the matrix polynomial as a whole.
The transformed spectral data is
\begin{center}
\begin{tabular}{l|c}
  Transformed irreducible divisor & $\PMg$ \\[1pt]
     \hline
  $\MT_R(\eta) = \la^4 + \la^3 + 1 = \chi(\la)$ \mystrut{4.1mm} &  $\;(0,0,1,2,3,3)$ \\[1pt]
  $\MT_R(\phi) \,=\, \phi(\la)$ &  $\;(0,1,1,1,1,3)$ \\[1pt]
  $\MT_R(0\la + 1) \,=\, \la = \psi(\la)$  &  $\;(0,1,1,2,2,4)$
\end{tabular}
\end{center}
Observe that this is the same data as was used in Examples \ref{ex.intro}-\ref{ex.partition},
so an appropriate realization for this transformed data
is the 2-quasi-triangular (grade $10$, strictly regular, and degree $10$) realization
from Example \ref{ex.partition}:
\[
  Q(\la) = \left[\begin{array}{cc|cc|cc}
                \chi\phi\psi (\phi + 1) & \chi\phi(\chi + \psi +1) & 0 & 0 & \psi(\phi + 1) & \phi + 1 \\
                \chi\phi\psi^4 & \chi\phi\psi^2 & 0 & 0 & \psi & 1 \\
            \hline \\[-12pt]
                     &  & \mystrut{3.7mm} \chi\phi\psi^2 & \chi\phi\psi^2(\phi + 1) & \phi\psi(\phi + 1) & 0 \\
                     &  & \chi\phi\psi^4 & \chi\phi\psi^2 & \phi\psi & 0 \\
            \hline \\[-12pt]
                     &  &  &  & \mystrut{3.7mm} \chi\phi\psi (\psi^2 + 1) & \chi\phi^2\psi^2 \\
                     &  &  &  &  \chi\phi^2\psi^2 & \chi\phi^2\psi
\end{array}\right].
\]
To complete the construction of the realization for the original spectral data,
we need to apply the inverse for the M\"{o}bius transformation $\MT_R$
used at the beginning, i.e., $\MT_{R^{-1}} = \MT_R$.
The result is the desired realization
\[
  \wt{Q}(\la) \,=\, \MT_R (Q) \,=\, \textstyle{\rev_{10}} (Q) \,=\,
     \left[\begin{array}{cc|cc|cc}
            \eta\phi (\phi + 1) & \eta\phi(\la^3 + \la + 1) & 0 & 0 & \la^8 + \la^7 & \la^9 + \la^8 \\
                   \eta\phi & \eta  \phi \la^2 & 0 & 0 & \la^9 & \la^{10} \\
            \hline \\[-12pt]
                0 & 0 & \mystrut{3.7mm} \eta\phi\la^2 & \eta\phi(\la + 1) & \phi\la^5(\la + 1) & 0 \\
                0 & 0 & \eta\phi  & \eta\phi\la^2  & \phi\la^7 & 0 \\
            \hline \\[-12pt]
                0 & 0 & 0 & 0 & \mystrut{3.7mm} \eta\phi (\la^3 + \la) & \eta\phi^2 \\
                0 & 0 & 0 & 0 &  \eta\phi^2 & \eta\phi^2\la
\end{array}\right].
\]
where $\eta := \MT_R(\chi) = \rev_4(\chi)$.

Observe that $\grade (\wt{Q}) = 10$ by construction,
and since $\alpha_1$ (the first partial multiplicity at $\infty$) is zero,
by Lemma~\ref{lem.grade-equal-degree} we know that $\deg \wt{Q}$
must be the same as $\grade \wt{Q}$.
Indeed $\deg \wt{Q}$ is $10$, but only barely.
There is \emph{only one} entry, the $(2,6)$ entry,
that has degree $10$.
Consequently the leading coefficient of $\wt{Q}(\la)$ has rank one,
or equivalently rank deficiency $5$,
which is consistent with the eigenvalue at $\infty$
having geometric multiplicity $5$,
as specified in the given spectral data.

Observe also that both the $(1,1)$ and $(3,3)$ blocks have degree $9$.
So it would be possible to use Lemma~\ref{LemOffDiagClean}
and either of these diagonal blocks to reduce the degree
of the off-diagonal $(1,3)$-block to be strictly less than $9$
via unimodular equivalence.
This would preserve the finite spectral structure,
but would have the unwanted side effect of spoiling
the infinite spectral structure.
Since degree and grade would no longer be equal,
by Lemma~\ref{lem.grade-equal-degree} we would have $\alpha_1 > 0$,
contrary to the given spectral data at $\infty$.
\end{ejem}

%%%%%%%%%%%%%%%%%%%%%%%%%%%%%%%%%%%%%%%%%%%%%%%%%%%%%%%%%%%%%%%%%
%%%%%%%%%%%%%%%%%%%%%%%%%%%%%%%%%%%%%%%%%%%%%%%%%%%%%%%%%%%%%%%%%

\subsection{Quasi-triangularization of regular matrix polynomials}
   \label{subsect.quasi-triangularization}

In the following signature result of this paper,
the goal is to quasi-triangularize a given polynomial matrix,
rather than to construct a quasi-triangular realization of some given data.
That is, we will start with an arbitrary regular polynomial matrix $P(\la)$,
and show that there must always be a quasi-triangular matrix $Q(\la)$
with the same degree, grade, and complete spectral data as $P(\la)$.

\begin{theorem}[Quasi-Triangularization] \label{thm.QT-regular} \quad \\
 	Suppose $P(\la)$ is a regular $n\times n$ matrix polynomial
 	of grade $g$ and degree $d$,
 	over an arbitrary field $\,\bF$.
 	Define $k$ to be the maximum degree among all
 	of the $\bF$-irreducible divisors of $P(\la)$.
 	If $P(\la)$ is \emph{not} strictly regular,
 	further suppose that there is some constant $\omega \in \bF$ such that $s_n(\omega) \ne 0$,
 	where $s_n(\la)$ is the $n^{\text{th}}$ invariant polynomial of $P(\la)$.
%  	for $i = 1,\dotsc,n$.
 	\parens{I.e., there is some $\omega \in\bF$ that is \emph{not} in the spectrum of $P$.}
 	Then there exists a regular $k$-quasi-triangular matrix polynomial $Q(\la)$ over $\,\bF$
 	that has exactly the same size, grade, degree, and complete spectral data as $P(\la)$.
 	When $P(\la)$ is strictly regular,
    then the $k$-quasi-triangularization $Q(\la)$ is strictly regular,
    and may be chosen to have the additional property
    that all off-diagonal blocks have degree strictly less than $d$.
\end{theorem}

\begin{proof}
  From $P(\la)$, extract the complete spectral data, size, degree, and grade.
  If $P(\la)$ is strictly regular, then use this data
  together with Theorem~\ref{thm.QTR-strictly-regular} to construct the desired $Q(\la)$
  with $\grade Q = \deg Q = \deg P = \grade P$.

  If $P(\la)$ is not strictly regular,
  then define a new matrix polynomial $\wt{P}(\la)$
  that is entrywise identical to $P(\la)$,
  but with $\grade \wt{P}$ chosen to be equal to $d = \deg P$.
  (If $g=d$ to begin with, then $\wt{P}$ is identical to $P$ in every way.)
  Note that the degrees of $P$ and $\wt{P}$
  as well as their finite spectral structures
  are the same,
  even though their grades may be different.
  Now if $\PM (P,\, 0\la + 1) = (\alpha_1, \dotsc, \alpha_n)$
  is the partial multiplicity sequence of $P$ at $\infty$,
  then by Lemma~\ref{lem.grade-equal-degree}
  the partial multiplicity sequence of $\wt{P}$ at $\infty$
  is shifted by $\alpha_1 = g-d$ from that of $P$,
  i.e.,
  \[
    \PM (\wt{P},\, 0\la + 1) \,=\; \PM (P,\, 0\la + 1) - \alpha_1 \cdot (1, 1, \dotsc, 1)
                            \,=\, (0, \alpha_2 - \alpha_1, \dotsc, \alpha_n - \alpha_1) \,.
  \]
  Now we can use Theorem \ref{thm.QTR-regular} to construct
  a $k$-quasi-triangular realization $\wt{Q}(\la)$
  for the complete spectral data, size, degree, and grade of this $\wt{P}(\la)$,
  for which we have
  $\deg \wt{Q} = \grade \wt{Q} = \grade \wt{P} = \deg \wt{P} = \deg P = d$.
  Finally, we define a $k$-quasi-triangular matrix polynomial $Q(\la)$
  that is entrywise identical to $\wt{Q}(\la)$,
  and so has $\deg Q = \deg \wt{Q} = d$,
  but now has grade chosen to be $g$.
  The infinite spectral data of $Q$ is shifted by $\alpha_1 = g - d$
  from that of $\wt{Q}$, so we have
  \begin{align*}
    \PM (Q,\, 0\la + 1) & =\, \PM (\wt{Q},\, 0\la + 1) + \alpha_1 \cdot (1, 1, \dotsc, 1) \\
                        & =\, \PM (\wt{P},\, 0\la + 1) + \alpha_1 \cdot (1, 1, \dotsc, 1)
                         \,=\, (\alpha_1, \alpha_2, \dotsc, \alpha_n)
                         \,=\, \PM (P,\, 0\la + 1) \,.
  \end{align*}
  Thus we see that $Q(\la)$
  has exactly the same size, degree, grade, and complete spectral data as $P(\la)$,
  and so is the desired quasi-triangularization of $P(\la)$.
\end{proof}

\begin{remark} \label{rem.spectral-equiv}
\rm
  It is worth emphasizing that the relationship between $ P(\la) $
  and the quasi-triangularization $ Q(\la) $
  is stronger than just unimodular equivalence.
%   (which is guaranteed simply by having the same finite spectral structure),
%   (which only entails having the same finite spectral structure),
%   because they also share the same infinite spectral structure.
%   In~\cite{spectralequivalence}, this stronger relation
%   is called \emph{spectral equivalence}.
  Theorem~\ref{thm.QT-regular} guarantees the existence
  of a \emph{spectrally equivalent} $k$-quasi-triangularization
  for any regular matrix polynomial over an arbitrary field.
\end{remark}

%%%%%%%%%%%%%%%%%%%%%%%%%%%%%%%%%%%%%%%%%%%%%%%%%%%%%%%%%%%%%%%%%%%%%%%
%%%%%%%%%%%%%%%%%%%%%%%%%%%%%%%%%%%%%%%%%%%%%%%%%%%%%%%%%%%%%%%%%%%%%%%
\section{More on diagonal block sizes} \label{sect.diagonal-block-size-bounds}

In this final section we explore %do a little probing %of the two ends of
the range of possibilities for diagonal block sizes
in degree-preserving quasi-triangularizations.
We have shown that every regular matrix polynomial $P(\la)$ over an arbitrary field
admits a spectrally equivalent degree-preserving $k$-quasi-triangularization,
where $k$ is the highest degree
among the irreducible divisors of $P(\la)$.
But is that really the best possible general result?
Could it be that there is a smaller bound on diagonal block sizes
of quasi-triangularizations that holds for all regular matrix polynomials?
Section~\ref{subsect.sharp-bound} addresses this issue,
exhibiting a family of examples that shows
that the $k$ in Theorem~\ref{thm.QT-regular}
is indeed the best possible general bound for diagonal block sizes.

By contrast, in Section~\ref{subsect.triangularizing}
we probe the opposite end of the size range of diagonal block sizes,
trying to determine when it is possible to achieve diagonal blocks
that are all $1 \times 1$,
i.e., when it is possible to just plain \emph{triangularize}
in a spectrally equivalent and degree-preserving way.
Although we have not even come close to completely settling this question,
we are at least able to identify some scenarios
where a necessary and sufficient condition for triangularizability can be found,
and some other more general scenarios where a condition
sufficient to guarantee the existence
of a triangularization can be given.

%%%%%%%%%%%%%%%%%%%%%%%%%%%%%%%%%%%%%%%%%%%%%%%%%%%%%%%%%%%%%%%%%%%%%%%%%%
\subsection{Sharpness of the upper bound $k$} \label{subsect.sharp-bound}

The matrix polynomials described in Example~\ref{ex.sharp-bound}
show that Theorem~\ref{thm.QT-regular} provides
the best possible general bound on diagonal block sizes.
This is done via an infinite family of examples where it can be proved
that every possible quasi-triangularization
has all of its diagonal blocks of size $k \times k$ or larger,
where $k$ is the largest degree among all of the irreducible divisors.

% We have shown that every regular polynomial matrix
% can be unimodularly quasi-triangularized
% in such a way that the degree and the spectrum
% (including possible eigenvalue at infinity) are preserved.
% It was also shown that there is a predictable upper bound
% on the size of the diagonal blocks,
% namely the highest degree among the irreducible divisors of the given polynomial.
% This upper bound is sharp, as the following example will show.

\begin{ejem} \label{ex.sharp-bound}
Consider any strictly regular $\nbyn$ matrix polynomial $P(\la)$ with degree $d \ge 2$,
such that $P(\la)$ has exactly one irreducible divisor $\chi(\la)$,
and $k = \deg(\chi) \ge 2$ \emph{is coprime to $d$}.
Since $mk = dn$ for some $m \in \bN$ by the index sum theorem,
we must then also have $k|n$.

Suppose $\bF$ is any field that supports
such an $\bF$-irreducible polynomial $\chi(\la)$ with $k = \deg(\chi) \ge 2$;
e.g., $\bF = \bQ$ has such a $\chi(\la)$ for any $k \ge 2$ at all.
Then there are infinitely many choices of $d$ and $n$
that satisfy the conditions mentioned above,
i.e., that $k$ is coprime to $d$ and $k|n$.
 For any of these choices of $k, d, n, \chi(\la)$, and the field $\bF$,
the Fundamental Realization Theorem~\ref{Teorprescribed}
guarantees the existence of a matrix polynomial $P(\la)$ over $\bF$
as described above,
i.e., one that is strictly regular, $\nbyn$, degree $d$,
and with just one irreducible divisor $\chi$.
Thus there are infinitely many matrix polynomials
encompassed by the discussion in this Example.

Now suppose that $Q(\la)$ is \emph{any} degree-$d$ quasi-triangularization of $P(\la)$;
that is, $Q$ is block-upper-triangular, has degree $d$,
and is unimodularly equivalent to $P$.
Suppose $Q$ has $\ell$ diagonal blocks $ Q_{11}(\la),Q_{22}(\la),\dotsc,Q_{\ell\ell}(\la) $
with sizes $n_i \times n_i$, respectively,
so that $n = \sum_{i=1}^{\ell} n_i$.
Now each block $Q_{ii}(\la)$ has degree at most $d$,
so that $\deg \bigl( \det Q_{ii} \bigr) \le d n_i$.
Thus we have
\[
  dn \,=\, \deg \bigl(\det P \bigr) \,=\, \deg \bigl( \det Q \bigr)
     \,=\, \deg \Bigl( \prod_{i=1}^{\ell} \det Q_{ii} \Bigr)
     \,=\, \sum_{i=1}^{\ell} \deg \bigl( \det Q_{ii} \bigr)
     \,\le\, \sum_{i=1}^{\ell} d n_i \,=\, dn \,.
\]
But this %sandwiching
means that each inequality $\deg \bigl( \det Q_{ii} \bigr) \le d n_i$
must actually be an equality,
so $\deg \bigl( \det Q_{ii} \bigr) = d n_i \ge 2$,
and hence $Q_{ii}(\la)$ is not a constant block.
Since $\det Q = \alpha \chi^m$ for some nonzero scalar $\alpha \in \bF$,
% and each block $Q_{ii}$ is non-constant,
that means that $\chi$ must divide each $\det Q_{ii}(\la)$.
Thus $k$ divides each $d n_{ii}$,
and hence also divides each $n_{ii}$
since $k$ and $d$ are coprime.
But this means that $k \le n_{ii}$ for each $1 \le i \le \ell$,
and so $Q(\la)$ is at best $k$-quasi-triangular.

%  Finally, note that this example is consistent with the discussion
% in Remark~\ref{rem.HP-proof-analysis}.
 Finally, note that the discussion in Remark~\ref{rem.HP-proof-analysis}
is relevant to this example.
In that remark, it was pointed out that there is only one scenario
in which the homogeneous partitioning procedure
forces there to be sublists with exactly $k$ elements,
which then later lead to $\kbyk$ diagonal blocks
in the quasi-triangularization.
It is not hard to show that this example
falls exactly under this scenario,
so our procedure will necessarily produce a quasi-triangularization
with \emph{all} of its diagonal blocks being $\kbyk$.
This, of course, is not equivalent to \emph{proving}
that no quasi-triangularization with any diagonal block of size
smaller than $\kbyk$ can exist, as we have done above.
But it certainly is completely consistent with that result.
\end{ejem}

When $d$ and $k$ are coprime,
then Example~\ref{ex.sharp-bound} has shown that the ``best''
quasi-triangularization that can be attained may sometimes
be forced to have all of its diagonal blocks with size $\kbyk$.
However, there are many matrix polynomials that have quasi-triangularizations
with much smaller diagonal blocks than the general upper bound of $\kbyk$.
Indeed we have seen this already in Example~\ref{ex.partition},
where we had $k=4$,
but were able to construct a $2$-quasi-triangularization.
The next example gives a whole family of matrix polynomials
that show that the gap between this general upper bound $k$
and the actual smallest realizable diagonal block size for quasi-triangularizations
can be arbitrarily large.
The discussion in Section~\ref{subsect.triangularizing}
provides further examples of this phenomenon.

\begin{ejem} \label{ex.small-diag-blocks}
%   Consider strictly regular $12 \times 12$ matrix polynomials $P(\la)$
%   with degree $d$ and only one
  Consider an arbitrary target degree $d$ and irreducible polynomial
  $\chi(\la)$ with $k = \deg(\chi) = 2d$.
  Note that there are many fields $\bF$
  that support the presence of such
  high degree $\bF$-irreducible polynomials, e.g., $\bF = \bQ$.
  For any such $\bF$-irreducible $\chi(\la)$ of degree $k$,
  there is a unique way to express it
  in the form $\chi(\la) = \la^d p(\la) + q(\la)$,
  where $\deg p = d$ and $\deg q = d-1$.
  Then letting
  \[
    X(\la) \;=\; \mat{cr} p(\la) & -1 \\ q(\la) & \la^d \rix \,,
  \]
  we see that $\deg X = d$, $\det X = \chi(\la)$,
  and the Smith form of $X(\la)$ is just $\diag \bigl( 1, \chi(\la) \bigr)$.
  This implies that the Smith form of $X^6(\la)$
  must be of the form $\diag(\chi^{\ell}, \chi^m)$,
  where $\ell + m = 6$.

  Now let $P(\la)$ be any strictly regular
  $12 \times 12$ matrix polynomial over $\bF$ with degree $d$,
  that has the Smith form $\diag(I_{10},\chi^{\ell}, \chi^m)$.
  By the Fundamental Realization Theorem~\ref{Teorprescribed}
  such matrix polynomials must exist.
  Consider next the matrix polynomial
  \[
    Q(\la) \;=\; \mat{cccccc} X(\la) & I_2 & & & & \\
                                     & X(\la) & I_2    & & & \\
                                     &        & X(\la) & I_2 & &  \\
                                     &        &        & X(\la) & I_2 & \\
                                     &        &        &        & X(\la) & I_2 \\
                                     &        &        &        &        & X(\la)
                 \rix
    \,\in\; \bF[\la]^{12 \times 12} \,.
  \]
  This $Q(\la)$ has degree $d$, and it is not hard to show
  that it is unimodularly equivalent to $\diag \bigl( I_{10}, X^6(\la) \bigr)$,
  and hence also to $P(\la)$.
  Thus $Q(\la)$ is a degree-preserving 2-quasi-triangularization of $P(\la)$.
  This $Q(\la)$ has very much smaller diagonal blocks
  than what is guaranteed by the general %$k$-quasi-triangularization
  result in Theorem~\ref{thm.QT-regular},
  with a gap (of $k-2$) in the size of diagonal blocks
  between the general upper bound and those actually occurring in $Q(\la)$,
  a gap that can be arbitrarily large.
\end{ejem}

%%%%%%%%%%%%%%%%%%%%%%%%%%%%%%%%%%%%%%%%%%%%%%%%%%%%%%%%%%%%%%%%%%%%%%%%%%%
\subsection{When triangularizing is possible} \label{subsect.triangularizing}

A natural question to ask is ``When is it possible to triangularize?",
or in other words, when can we guarantee the existence
of a degree-preserving quasi-triangularization in which all diagonal blocks are 1$\times$1?
To answer this, let's begin with an example
where all of the irreducible divisors have degree 2 or less,
as in a real matrix polynomial.
\begin{ejem}\label{ex.2x2-triang}
  Suppose that we are trying to build a \emph{triangularization}
  of a strictly regular matrix polynomial with degree $d=7$ and size $n=8$,
  and there are a total of $m_1 = 18$ degree-1 irreducible factors
  and $m_2 = 19$ degree-2 irreducible factors in the Smith form,
  for a total degree sum of $56$.
  One strategy to build an appropriate target diagonal
  is to first spread the degree-2 factors out as much as possible,
  and then try to fill in the resulting gaps with the degree-1 factors.
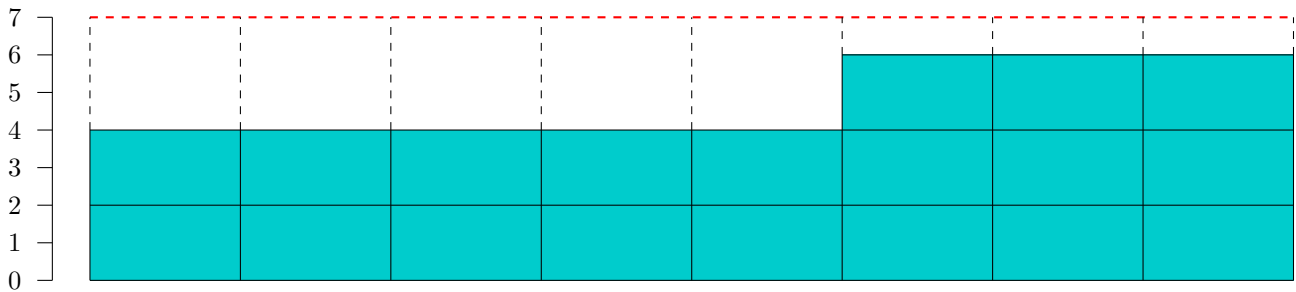
\begin{center}
\begin{tikzpicture}
\draw (-0.5,0) -- (-0.5,3.5);
\draw (-0.7,0) -- (-0.5,0);
\node at (-1,0) {0};
\draw (-0.7,0.5) -- (-0.5,0.5);
\node at (-1,0.5) {1};
\draw (-0.7,1) -- (-0.5,1);
\node at (-1,1) {2};
\draw (-0.7,1.5) -- (-0.5,1.5);
\node at (-1,1.5) {3};
\draw (-0.7,2) -- (-0.5,2);
\node at (-1,2) {4};
\draw (-0.7,2.5) -- (-0.5,2.5);
\node at (-1,2.5) {5};
\draw (-0.7,3) -- (-0.5,3);
\node at (-1,3) {6};
\draw (-0.7,3.5) -- (-0.5,3.5);
\node at (-1,3.5) {7};

\fill[cyan] (0,0) rectangle (16,2);
\fill[cyan] (10,2) rectangle (16,3);

\draw (0,0) -- (16,0);
\draw (0,1) -- (16,1);
\draw (0,2) -- (16,2);
\draw (10,3) -- (16,3);
\draw[thick,dashed,red] (0,3.5) -- (16,3.5);

\draw (0,0) -- (0,2);
\draw[dashed] (0,2) -- (0,3.5);
\draw (2,0) -- (2,2);
\draw[dashed] (2,2) -- (2,3.5);
\draw (4,0) -- (4,2);
\draw[dashed] (4,2) -- (4,3.5);
\draw (6,0) -- (6,2);
\draw[dashed] (6,2) -- (6,3.5);
\draw (8,0) -- (8,2);
\draw[dashed] (8,2) -- (8,3.5);
\draw (10,0) -- (10,3);
\draw[dashed] (10,3) -- (10,3.5);
\draw (12,0) -- (12,3);
\draw[dashed] (12,3) -- (12,3.5);
\draw (14,0) -- (14,3);
\draw[dashed] (14,3) -- (14,3.5);
\draw (16,0) -- (16,3);
\draw[dashed] (16,3) -- (16,3.5);
\end{tikzpicture}
\end{center}
Let $\mathbf{f}_1$ and $\mathbf{f}_2$ denote
the degree-1 and degree-2 factor-counting vectors
of the Smith form, respectively.
If $\mathbf{v}_2$ is the homogenization of $\mathbf{f}_2$,
then we know that it is possible to spread out
the $19$ degree-2 factors along the diagonal
to realize this $\mathbf{v}_2$ via unimodular transformations,
using Lemma~\ref{lem.conversion-to-homogen-version}
and Corollary~\ref{cor.unimod-trans}.
This homogenization is visualized in the diagram above,
where each column displays the contents of a diagonal entry location,
each box stands for an irreducible factor,
and the height of each box displays the degree of that factor,
in this case a height/degree of $2$.
 To achieve a triangularization,
we need to have each diagonal entry have degree $7$,
so our goal is to populate each column in the diagram
with boxes up to exactly height $7$.
The amount of remaining space between the top of the current stack of cyan blocks
and the red dashed line in each column will be called a degree gap,
and the vector containing all of the degree gaps will be called the gap vector $\mathbf{g}$.
In this example the gap vector is $\mathbf{g} = (3, 3, 3, 3, 3, 1, 1, 1)$.
So what remains is to try to distribute the $18$ degree-1 irreducible factors
(height-1 boxes) so as to exactly fill these gaps.
In other words, we need to try to convert $\mathbf{f}_1$
into the gap vector $\mathbf{g}$.
Now our only means to move these degree-1 factors around
is to use the tools from Corollary~\ref{cor.unimod-trans},
which correspond to compressions and interchanges
of the entries of the factor-counting  vector.
% But by the Muirhead-like result for natural vectors,
But such actions can only convert $\mathbf{f}_1$ into a vector that it majorizes.
(Recall the classical Muirhead theorem discussed
 right after Remark~\ref{rem.spread-out}.)
Thus this strategy will succeed in producing
an appropriate diagonal for an achievable triangularization
whenever $\mathbf{f}_1 \succeq \mathbf{g}$.
Consequently we see that this majorization condition is a sufficient condition
to guarantee the triangularizability of a strictly regular matrix polynomial
whenever all irreducible divisors are of degree at most two.
\end{ejem}

In fact, though, the condition $\mathbf{f}_1 \succeq \mathbf{g}$
is also a necessary condition for triangularizability,
and so gives a characterization in this scenario,
as will be seen as an immediate consequence of the following development.
We begin with some background lemmas.
The first of these lemmas uses an alternative definition
for majorization of vectors, that is equivalent
to the one given earlier in Definition~\ref{def.majorization}.

\begin{defi}[Majorization~\cite{Marshall}] \label{def.majorization-alt} \quad \\
  For vectors $\mathbf{x}, \mathbf{y}$ in $\bR^n$ (or in $\bZ^n$),
  let $\mathbf{x}^{\prime\prime}$ and $\mathbf{y}^{\prime\prime}$
  denote the permutations of those vectors
  in which the entries have been arranged in \emph{increasing} order.
  We say that \emph{$\mathbf{x}$ majorizes $\mathbf{y}$},
  or \emph{$ \mathbf{y} $ is majorized by $ \mathbf{x} $},
  and write $\, \mathbf{x} \succeq \mathbf{y} \,$, if
\begin{align}\label{eqn.majorization-alt}
  \sum_{i=1}^{\ell} x_i^{\prime\prime} \,\le\, \sum_{i=1}^{\ell} y_i^{\prime\prime}
    \quad \text{ for } \quad \ell = 1,2,\dotsc, n \;,
\end{align}
with equality when $ \ell = n $.
\end{defi}

\begin{lema} \label{lem.triang-majorizations}
  Suppose $T(\la)$ is a regular triangular matrix polynomial
  over an arbitrary field $\,\bF$,
  and has the Smith form $S(\la)$.
  Let $\mathcal{F} \sqcup \mathcal{G}$ be \emph{any} coprime partition
  of the multiset %$\cM$
  of all of the $\mathbb{F}$-irreducible factors
  of the invariant polynomials in $S(\la)$,
  equivalently of all of the $\mathbb{F}$-irreducible factors
  of the diagonal entries of $T(\la)$.
  Then the majorization relations
  \begin{equation} \label{eqn.triang-majorization-relations}
    \mathbf{d}_{\scF}(S) \,\succeq\, \mathbf{d}_{\scF}(T)
  \end{equation}
  hold for the diagonal factor-counting vectors of $S(\la)$ and $T(\la)$,
  with respect to \emph{every} $\cF$ from \emph{every} such coprime partition.
\end{lema}
\begin{proof}
  For convenience, let us introduce some notation to ease the discussion.
  Let
  \[
    \bm{\sigma} \;=\; (\sigma_1, \sigma_2, \dotsc, \sigma_n) \,:=\; \mathbf{d}_{\scF}(S)
  \]
  be an abbreviation for the diagonal factor-counting vector $\mathbf{d}_{\scF}(S)$.
  Note that $\sigma_1 \le \sigma_2 \le \dotsb \le \sigma_n$
  because of the divisibility chain property of invariant polynomials,
  but the entries of $\mathbf{d}_{\scF}(T)$ may not be in any such order.
  So let $\bm{\tau} = (\tau_1, \tau_2, \dotsc, \tau_n)$
%                     = (\kappa_{j_1}, \kappa_{j_2}, \dotsc, \kappa_{j_n})$
  be the permutation of $\mathbf{d}_{\scF}(T)$
  such that $\tau_1 \le \tau_2 \le \dotsb \le \tau_n$.
  Now from \eqref{eqn.gcd-characterization}
  we know that $s_{11}(\la) \vert \,t_{jj}(\la)$ for $j=1,\dotsc,n$,
  so $|s_{11}(\la)|_{\scF} \le |t_{jj}(\la)|_{\scF}$ for $j=1,\dotsc,n$,
  and thus $\sigma_1 \le \tau_1 = \min_j |t_{jj}(\la)|_{\scF}$.

  Using \eqref{eqn.gcd-characterization} again,
  we have that $\bigl[ s_{11}(\la) s_{22}(\la) \bigr] \vert \,\bigl[ t_{ii}(\la)t_{jj}(\la) \bigr]$
  for $i,j=1,\dotsc,n$ with $i \ne j$.
  Hence
  \[
    |s_{11}(\la) s_{22}(\la)|_{\scF} \;\le\; |t_{ii}(\la) t_{jj}(\la)|_{\scF}
  \]
  for $i,j=1,\dotsc,n$ with $i \ne j$,
  and consequently $\sigma_1 + \sigma_2 \le \tau_1 + \tau_2$.
  In a similar manner, we see from \eqref{eqn.gcd-characterization} that
  for each $1 \le \ell \le n-1$ we have
  \[
    |s_{11}(\la) s_{22}(\la) \dotsb s_{\ell \ell}(\la)|_{\scF}
      \;\le\;
    |t_{{i_1}{i_1}}(\la) t_{{i_2}{i_2}}(\la) \dotsb t_{{i_\ell}{i_\ell}}(\la) |_{\scF}
  \]
  for all $\ell$-tuples $(i_1, i_2, \dotsc, i_\ell)$ with distinct entries
  and $i_1, i_2, \dotsc, i_\ell = 1, \dotsc, n$.
  Thus we have
  \[
    \sigma_1 + \sigma_2 + \dotsb + \sigma_{\ell} \;\le\; \tau_1 + \tau_2 + \dotsb + \tau_{\ell}
  \]
  for each $1 \le \ell \le n-1$.
  Finally, since $\det S = c \cdot \det T$ for some nonzero scalar $c \in \bF$,
  we have
  \[
    |s_{11}(\la) s_{22}(\la) \dotsb s_{nn}(\la)|_{\scF}
      \;=\;
    |t_{11}(\la) t_{22}(\la) \dotsb t_{nn}(\la) |_{\scF} \,,
  \]
  and hence
  \[
    \sigma_1 + \sigma_2 + \dotsb + \sigma_n \;=\; \tau_1 + \tau_2 + \dotsb + \tau_n \,.
  \]
  By Definition~\ref{def.majorization-alt}, then,
  we have $\bm{\sigma} \succeq \bm{\tau}$,
  and hence that \eqref{eqn.triang-majorization-relations} holds.
\end{proof}

The next result shows that if the irreducible divisors of a matrix polynomial
have \emph{only two different degrees},
and one of those is degree $1$,
then having a triangularization of any kind
implies that there must also exist a triangularization
in which the highest degree irreducible factors are indeed
''spread out as much as possible'' as in Example~\ref{ex.2x2-triang},
i.e., where their factor-counting vector is $1$-homogeneous.

\begin{lema} \label{lem.triang-conversion}
  Suppose a strictly regular matrix polynomial $P(\la)$ of degree $d$
  has a triangularization $Q_0(\la)$ of degree $d$.
  Further suppose that the multiset $\cM$ of all of the $\,\bF$-irreducible factors
  in the Smith form for $P(\la)$ contains \emph{only two} distinct degrees,
  $1$ and $k$, for some $k \ge 2$.
  Let $\cM = \cF_1 \sqcup \cF_k$ be the coprime partition
  in which $\cF_j$ contains all of the irreducible factors in $\cM$ of degree $j$ for $j = 1,k$.
  Then $P(\la)$ has a triangularization $T(\la)$ of degree $d$
  in which the diagonal factor-counting vector $\mathbf{d}_{\scFk}(T)$
  is $1$-homogeneous.
\end{lema}
\begin{proof}
  If $\mathbf{d}_{{\cF}_k}(Q_0)$ is already $1$-homogeneous,
  take $T(\la) = Q_0(\la)$ and then of course we are done.
  So suppose that $\bm{\kappa} = (\kappa_1, \kappa_2, \dotsc, \kappa_n) := \mathbf{d}_{{\cF}_k}(Q_0)$
  is not $1$-homogeneous.
  The argument will consist of a procedure showing how to convert $Q_0(\la)$
  by a finite sequence of triangularizations for $P(\la)$
  into a degree-$d$ triangularization $T(\la)$
  that has the desired $1$-homogeneity property.

  Let $\kappa_i$ and $\kappa_j$ be the minimum and maximum entries in $\bm{\kappa}$,
  so $\kappa_j - \kappa_i \ge 2$.
  Now by a finite sequence of interchanges we can arrange
  that the corresponding diagonal entries of $Q_0(\la)$ are adjacent,
  say in the $(i,i)$ and $(i+1,i+1)$ locations,
  and these interchanges can be implemented by unimodular transformations
  using Corollary~\ref{cor.unimod-trans}.
  This gives us a new degree $d$ triangularization $\wt{Q}_0(\la)$ for $P(\la)$,
  with a new $\bm{\wt{\kappa}} = \mathbf{d}_{{\cF}_k}(\wt{Q}_0)$
  with $\wt{\kappa}_i = \kappa_i$ and $\wt{\kappa}_{i+1} = \kappa_j$.
  Now we do a compression of the degree $k$ factors in these two adjacent entries,
  again via a unimodular transformation from Corollary~\ref{cor.unimod-trans},
  decreasing the maximum $\wt{\kappa}_{i+1}$ by one
  and increasing the minimum $\wt{\kappa}_i$ by one.
  This gives us a triangular matrix polynomial $\wh{Q}_0(\la)$
  that is no longer degree $d$, although it is still unimodularly equivalent to $P(\la)$;
  all diagonal entries have degree $d$, except for the $i^{th}$ and $(i+1)^{th}$,
  which now have degrees $d+k$ and $d-k$, respectively.
  We can now restore degree $d$ on these adjacent diagonal entries,
  by doing a compression of the degree $1$ factors,
  again using a unimodular transformation from Corollary~\ref{cor.unimod-trans}.
  The $i^{th}$ diagonal entry of $\wh{Q}_0(\la)$ has $d - k \kappa_i$ degree-$1$ factors,
  while the $(i+1)^{th}$ diagonal entry has $d - k \kappa_j$ degree-$1$ factors.
  Since $(d - k \kappa_i) > (d - k \kappa_j)$,
  with a degree difference of at least $2k$,
  we can do a compression of the degree-$1$ factors
  where the $i^{th}$ diagonal entry loses $k$ degree-$1$ factors,
  and the $(i+1)^{th}$ diagonal entry gains $k$ degree-$1$ factors.
  This gives a triangular matrix polynomial $\breve{Q}_0(\la)$
  that is unimodularly equivalent to $P(\la)$,
  and has all diagonal entries with degree $d$ again,
  but the off-diagonal entries may now have degree larger than $d$.
  This is remedied by using Lemma~\ref{LemOffDiagClean},
  with all diagonal blocks taken to be of size $1 \times 1$.
  This finally gives us a degree-$d$ triangularization $Q_1(\la)$ of $P(\la)$,
  with $\mathbf{d}_{{\cF}_k}(Q_1)$ one step closer to being $1$-homogeneous
  than $\mathbf{d}_{{\cF}_k}(Q_0)$ was.

  Of course if $\mathbf{d}_{{\cF}_k}(Q_1)$ is now $1$-homogeneous,
  then we take $T(\la) = Q_1(\la)$ and we are done.
  If not, we repeat the above procedure on $Q_1(\la)$
  to produce a new degree-$d$ triangularization $Q_2(\la)$
  with diagonal factor-counting vector $\mathbf{d}_{{\cF}_k}(Q_2)$
  that is even closer to being $1$-homogeneous.
  Continuing this, we generate a sequence $Q_0, Q_1, Q_2, \dotsc$
  of degree-$d$ triangularizations for $P(\la)$,
  which in finitely many steps must eventually produce
  a degree-$d$ triangularization $T(\la)$
  for which $\mathbf{d}_{\scFk}(T)$ is $1$-homogeneous.
\end{proof}

\begin{prop} \label{prop.triang-characterization}
  Let $P(\la)$ be a strictly regular $\nbyn$ polynomial matrix
  of degree $d$ over a field $\,\bF$.
  Let $S(\la)$ be the Smith form of $P(\la)$,
  and assume that all irreducible divisors of $P(\la)$
  are degree $1$ or degree $k$, where $k \ge 2$.
  Let $\cF_1 \sqcup \cF_k$ be the same coprime partition
  of the multiset of all $\,\bF$-irreducible factors in $S(\la)$
  as in Lemma~\ref{lem.triang-conversion}.
  Organize the degree-$k$ factors into a vector of $n$ polynomials
  $\mathbf{q} (\la) := \bigl( q_1(\la), q_2(\la), \dotsc, q_n(\la) \bigr)$
  in the same way as in the proof of Lemma~\ref{Stacking}
  \parens{and as in the diagram for Example~\ref{ex.2x2-triang}\,},
  i.e., so that $|\mathbf{q}(\la)|_{\scFk}$ is $1$-homogeneous.
  Define the degree gaps $g_i := d - \deg(q_i)$,
  and the corresponding gap vector $\mathbf{g} := (g_1,g_2,\dotsc,g_n)$.
  \parens{Note that some of the $g_i$ may be negative.}
  Then $P(\la)$ has a triangularization of degree $d$
  if and only if the majorization condition
  $\mathbf{d}_{\scFo}(S) \,\succeq\, \mathbf{g}$ holds.
\end{prop}
\begin{proof}
  $(\Rightarrow)$ \,\, Suppose $P(\la)$ has a triangularization of degree $d$.
  Then by Lemma~\ref{lem.triang-conversion},
  $P(\la)$ has a triangularization $T(\la)$ of degree $d$
  in which the diagonal factor-counting vector $\mathbf{d}_{\scFk}(T)$
  is $1$-homogeneous.
  Applying Lemma~\ref{lem.triang-majorizations} to the coprime partition $\cF_1 \sqcup \cF_k$,
  we then have that $\mathbf{d}_{\scFo}(S) \,\succeq\, \mathbf{d}_{\scFo}(T)$.
  But in this triangularization $T(\la)$,
  it is easy to see that
  the diagonal factor-counting vector $\mathbf{d}_{\scFo}(T)$
  is exactly the same as the gap vector $\mathbf{g}$.
  Thus $\mathbf{d}_{\scFo}(S) \,\succeq\, \mathbf{g}$, as desired.

  \medskip
  \noindent
  $(\Leftarrow)$ \,\, Now conversely, suppose that the majorization condition
  $\mathbf{d}_{\scFo}(S) \,\succeq\, \mathbf{g}$ holds.
  Starting from $S(\la)$,
  we know from Corollary~\ref{cor.unimod-trans}
  and Lemma~\ref{lem.conversion-to-homogen-version}
  that we can spread out the irreducible factors in $\cF_k$ along the diagonal
  via unimodular transformations so as to form
  a triangular matrix polynomial $\wt{T}(\la)$
  such that $\mathbf{d}_{\scFk}(\wt{T}) = |\mathbf{q}(\la)|_{\scFk}$ is $1$-homogeneous,
  and $\mathbf{d}_{\scFo}(\wt{T}) = \mathbf{d}_{\scFo}(S)$.
  Since $\mathbf{d}_{\scFo}(S) \,\succeq\, \mathbf{g}$,
  there exists a finite sequence of interchanges and compressions of adjacent diagonal entries
  that will turn $\mathbf{d}_{\scFo}(S)$ into $\mathbf{g}$.
  Implementing this sequence via unimodular transformations from Corollary~\ref{cor.unimod-trans}
  will produce a triangular matrix polynomial $T(\la)$
  such that $\mathbf{d}_{\scFk}(T) = \mathbf{d}_{\scFk}(\wt{T}) = |\mathbf{q}(\la)|_{\scFk}$
  is $1$-homogeneous,
  $\mathbf{d}_{\scFo}(T) = \mathbf{g}$,
  $\deg T = d$,
  and $T(\la)$ is unimodularly equivalent to $S(\la)$,
  and hence also to $P(\la)$.
  In other words, $T(\la)$ is the desired degree-$d$ triangularization of $P(\la)$.
\end{proof}

\begin{remark}
\rm
  Necessary and sufficient conditions for the triangularizability
  of strictly regular \emph{real} matrix polynomials
  were given in \cite[Theorem 4.9]{TasTisZab}.
  Note that Proposition~\ref{prop.triang-characterization} recovers this result
  from \cite{TasTisZab}
  as the special case when $k=2$.
\end{remark}

The result of Proposition~\ref{prop.triang-characterization}
can be extended to a slightly more general scenario,
still involving irreducible divisors with only two degrees,
but no longer tied to requiring one of those degrees to be $1$.
This more general scenario is essentially
just a ``scaled'' version of the one discussed in the Proposition.

\begin{coro}\label{cor.triang-characterization}
  Let $P(\la)$ be a strictly regular $\nbyn$ polynomial matrix
  of degree $d$ over a field $\,\bF$.
  Let $S(\la)$ be the Smith form of $P(\la)$,
  and assume that all irreducible divisors of $P(\la)$
  are degree $\ell$ or degree $k$, where $k > \ell \ge 1$.
  Also assume that $\ell @\vert@ k$ and $\ell @\vert@ d$.
  Let $\cF_{\ell} \sqcup \cF_k$ be the coprime partition
  of the multiset of all $\,\bF$-irreducible factors in $S(\la)$
  where $\cF_{\ell}$ contains all of the degree-$\ell$ factors,
  and $\cF_k$ contains all of the degree-$k$ factors.
  Organize the degree-$k$ factors into a vector of $n$ polynomials
  $\mathbf{q} (\la) := \bigl( q_1(\la), q_2(\la), \dotsc, q_n(\la) \bigr)$
  in the same way as in the proof of Lemma~\ref{Stacking},
%   \parens{and as in the diagram for Example~\ref{ex.2x2-triang}\,},
  i.e., so that $|\mathbf{q}(\la)|_{\scFk}$ is $1$-homogeneous.
  Define the degree gaps $g_i := d - \deg(q_i)$,
  and the corresponding gap vector $\mathbf{g} := (g_1,g_2,\dotsc,g_n)$.
  \parens{Note that some of the $g_i$ may be negative.}
  Then $P(\la)$ has a triangularization of degree $d$
  if and only if the majorization condition
  $\;\ell \cdot\mathbf{d}_{\scFell}(S) \,\succeq\, \mathbf{g}$ holds.
\end{coro}
\begin{proof}
  The scenario described in this corollary can be viewed as a scaled version
  of the one handled in Proposition~\ref{prop.triang-characterization}.
  From the divisibility assumptions $\ell @\vert@ k$ and $\ell @\vert@ d$,
  let $k = \kappa \ell$ and $d = \delta \ell$, with $1 < \kappa < \delta$.
  Then by viewing $\ell$ as the basic ``unit'' of degree,
  the scenario of this corollary is just like
  that of Proposition~\ref{prop.triang-characterization}
  with $1$, $k$ and $d$ replaced by $1$, $\kappa$, and $\delta$.
\end{proof}

% \tcg{\underline{Example}: Consider $d=24$, $k=9$, and $\ell =3$.
%             This could be viewed as just a scaling of $d=8$, $k=3$, and $\ell = 1$,
%              which puts us right back into the scenario of the previous Proposition.}

The results in Proposition~\ref{prop.triang-characterization}
and its Corollary~\ref{cor.triang-characterization}
show that the triangularization question is still somewhat tractable
when there are no more than two different degrees
among all of the irreducible divisors that are present.
However, when irreducible divisors have three or more degrees,
the picture gets much more involved,
with some significant combinatorial complexity now possible.
The strategy guiding Example~\ref{ex.2x2-triang} is still viable, though,
and sometimes is able to provide sufficient conditions
for guaranteeing that a triangularization is possible,
although these conditions may no longer be necessary.
To see why this is the case, we consider a few more examples,
this time with irreducible divisors of degrees 1, 2, and 3.

\begin{ejem}\label{ex.3x3-triang}
 For this example we aim for degree $d = 10$ and size $n=8$,
but this time with $m_3 = 9$ degree-3 factors,
$m_2 = 18$ degree-2 factors, and $m_1 = 17$ degree-1 factors,
for a total degree sum of $80$.
The sufficient condition is the same as before,
i.e., the gap vector must be majorized by the degree-1 factor-counting  vector,
but the gap vector is defined slightly differently
to how it was done in Example~\ref{ex.2x2-triang}.
To determine the gap vector,
begin by spreading out the degree-3 and degree-2 factors
in a way similar to Lemma~\ref{Stacking},
as pictured in the following diagram.
Note that we will definitely be able to do this
by unimodular transformations,
since both the degree-3 and the degree-2 factor-counting vectors
in the diagram are $1$-homogeneous,
and we have already seen that factor-counting vectors
can always be homogenized.

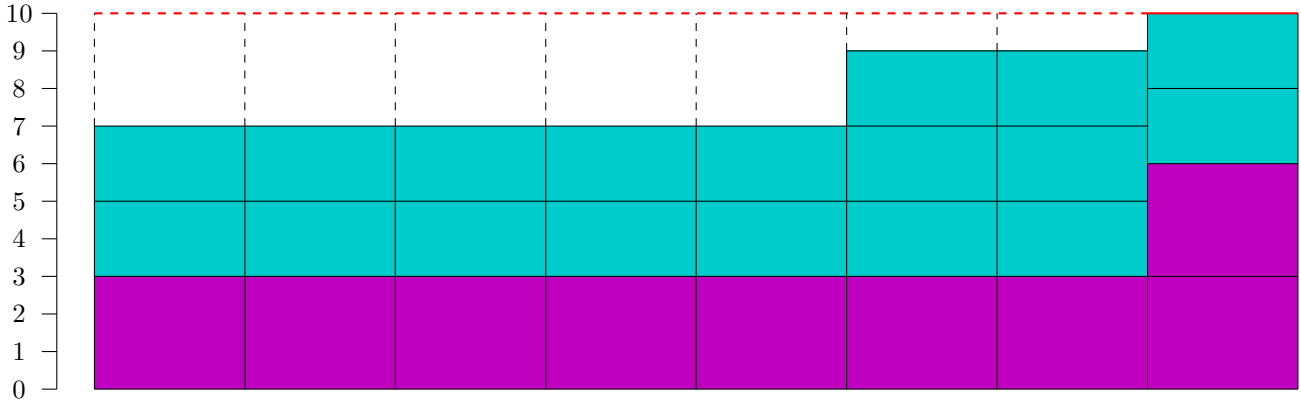
\begin{center}
\begin{tikzpicture}
\draw (-0.5,0) -- (-0.5,5);
\draw (-0.7,0) -- (-0.5,0);
\node at (-1,0) {0};
\draw (-0.7,0.5) -- (-0.5,0.5);
\node at (-1,0.5) {1};
\draw (-0.7,1) -- (-0.5,1);
\node at (-1,1) {2};
\draw (-0.7,1.5) -- (-0.5,1.5);
\node at (-1,1.5) {3};
\draw (-0.7,2) -- (-0.5,2);
\node at (-1,2) {4};
\draw (-0.7,2.5) -- (-0.5,2.5);
\node at (-1,2.5) {5};
\draw (-0.7,3) -- (-0.5,3);
\node at (-1,3) {6};
\draw (-0.7,3.5) -- (-0.5,3.5);
\node at (-1,3.5) {7};
\draw (-0.7,4) -- (-0.5,4);
\node at (-1,4) {8};
\draw (-0.7,4.5) -- (-0.5,4.5);
\node at (-1,4.5) {9};
\draw (-0.7,5) -- (-0.5,5);
\node at (-1,5) {10};

\fill[magenta] (0,0) rectangle (16,1.5);
\fill[magenta] (14,1.5) rectangle (16,3);
\fill[cyan] (0,1.5) rectangle (14,3.5);
\fill[cyan] (10,3.5) rectangle (16,4.5);
\fill[cyan] (14,3) rectangle (16,5);

\draw (0,0) -- (16,0);
\draw (0,1.5) -- (16,1.5);
\draw (0,2.5) -- (14,2.5);
\draw (14,3) -- (16,3);
\draw (0,3.5) -- (14,3.5);
\draw (14,4) -- (16,4);
\draw (10,4.5) -- (14,4.5);
\draw[thick,dashed,red] (0,5) -- (14,5);
\draw[thick,red] (14,5) -- (16,5);

\draw (0,0) -- (0,3.5);
\draw[dashed] (0,3.5) -- (0,5);
\draw (2,0) -- (2,3.5);
\draw[dashed] (2,3.5) -- (2,5);
\draw (4,0) -- (4,3.5);
\draw[dashed] (4,3.5) -- (4,5);
\draw (6,0) -- (6,3.5);
\draw[dashed] (6,3.5) -- (6,5);
\draw (8,0) -- (8,3.5);
\draw[dashed] (8,3.5) -- (8,5);
\draw (10,0) -- (10,4.5);
\draw[dashed] (10,4.5) -- (10,5);
\draw (12,0) -- (12,4.5);
\draw[dashed] (12,4.5) -- (12,5);
\draw (14,0) -- (14,5);
%\draw[dashed] (14,3) -- (14,3.5);
\draw (16,0) -- (16,5);
%\draw[dashed] (16,3) -- (16,3.5);
\end{tikzpicture}
\end{center}
The gap vector for this configuration is $\mathbf{g} = (3, 3, 3, 3, 3, 1, 1, 0)$,
so in order for a triangularization
(with this particular degree-$3$ and degree-$2$ configuration)
to be guaranteed to exist,
this vector must be majorized by the degree-1 factor-counting  vector
for the diagonal of the Smith form.
 For instance, if the degree-1 factor-counting vector in the Smith form
is $\mathbf{f}_1 = (0, 0, 0, 0, 2, 5, 5, 5)$, %$\mathbf{f}_1 = (5,5,5,2,0,0,0,0)$,
then there is a triangularization, since $\mathbf{f}_1 \succeq \mathbf{g}$.
On the other hand if the Smith form's degree-1 factor-counting  vector
is $\wt{\mathbf{f}}_1 = (1, 1, 1, 2, 3, 3, 3, 3)$, %$\wt{\mathbf{f}}_1 = (3,3,3,3,2,1,1,1)$,
then a triangularization may still exist,
but it cannot be guaranteed to exist by this pathway
since $\wt{\mathbf{f}}_1 \not\succeq \mathbf{g}$.

However, if we modify the layout of the degree-2 factors just a little bit,
then we can see that the degree-1 factor-counting  vector $\wt{\mathbf{f}}_1$
will admit a triangularization.
Let us shift one (cyan) height-2 block from the eighth column
to the fifth column, as in the diagram.
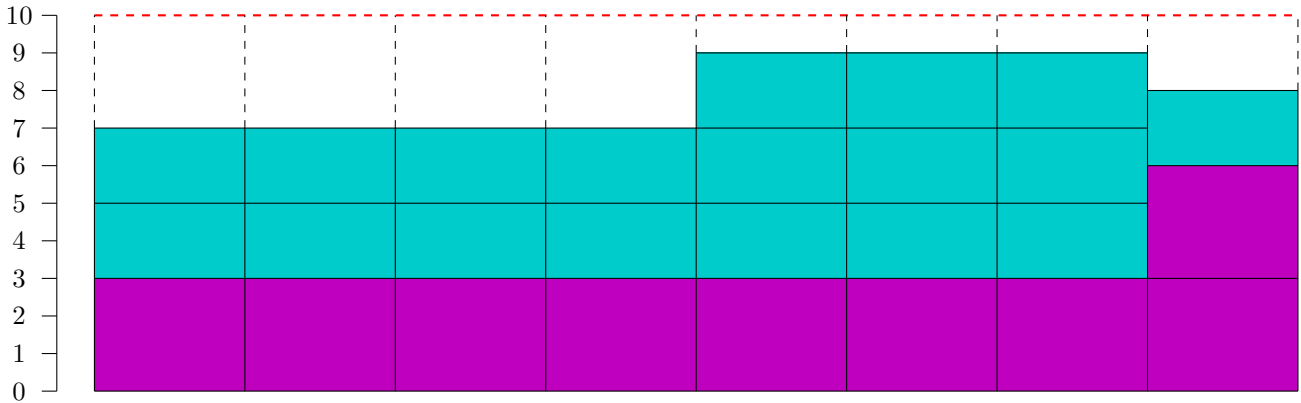
\begin{center}
\begin{tikzpicture}
\draw (-0.5,0) -- (-0.5,5);
\draw (-0.7,0) -- (-0.5,0);
\node at (-1,0) {0};
\draw (-0.7,0.5) -- (-0.5,0.5);
\node at (-1,0.5) {1};
\draw (-0.7,1) -- (-0.5,1);
\node at (-1,1) {2};
\draw (-0.7,1.5) -- (-0.5,1.5);
\node at (-1,1.5) {3};
\draw (-0.7,2) -- (-0.5,2);
\node at (-1,2) {4};
\draw (-0.7,2.5) -- (-0.5,2.5);
\node at (-1,2.5) {5};
\draw (-0.7,3) -- (-0.5,3);
\node at (-1,3) {6};
\draw (-0.7,3.5) -- (-0.5,3.5);
\node at (-1,3.5) {7};
\draw (-0.7,4) -- (-0.5,4);
\node at (-1,4) {8};
\draw (-0.7,4.5) -- (-0.5,4.5);
\node at (-1,4.5) {9};
\draw (-0.7,5) -- (-0.5,5);
\node at (-1,5) {10};

\fill[magenta] (0,0) rectangle (16,1.5);
\fill[magenta] (14,1.5) rectangle (16,3);
\fill[cyan] (0,1.5) rectangle (14,3.5);
\fill[cyan] (8,3.5) rectangle (14,4.5);
\fill[cyan] (14,3) rectangle (16,4);

\draw (0,0) -- (16,0);
\draw (0,1.5) -- (16,1.5);
\draw (0,2.5) -- (14,2.5);
\draw (14,3) -- (16,3);
\draw (0,3.5) -- (14,3.5);
\draw (14,4) -- (16,4);
\draw (8,4.5) -- (14,4.5);
\draw[thick,dashed,red] (0,5) -- (16,5);
%\draw[thick,red] (14,5) -- (16,5);

\draw (0,0) -- (0,3.5);
\draw[dashed] (0,3.5) -- (0,5);
\draw (2,0) -- (2,3.5);
\draw[dashed] (2,3.5) -- (2,5);
\draw (4,0) -- (4,3.5);
\draw[dashed] (4,3.5) -- (4,5);
\draw (6,0) -- (6,3.5);
\draw[dashed] (6,3.5) -- (6,5);
\draw (8,0) -- (8,4.5);
\draw[dashed] (8,4.5) -- (8,5);
\draw (10,0) -- (10,4.5);
\draw[dashed] (10,4.5) -- (10,5);
\draw (12,0) -- (12,4.5);
\draw[dashed] (12,4.5) -- (12,5);
\draw (14,0) -- (14,4.5);
\draw[dashed] (14,4.5) -- (14,5);
\draw (16,0) -- (16,4);
\draw[dashed] (16,4) -- (16,5);
\end{tikzpicture}
\end{center}
But is this new configuration actually reachable by unimodular transformations?
Now we have a new condition;
the degree-2 factor-counting vector in the Smith form
must majorize $(2, 2, 2, 2, 3, 3, 3, 1)$, %$(3,3,3,2,2,2,2,1)$,
the degree-2 factor-counting vector in this new configuration.
Assuming that this new condition is satisfied,
we still need the degree-1 factor-counting vector
to majorize the new gap vector
$\wt{\mathbf{g}} = (3, 3, 3, 3, 1, 1, 1, 2)$.
 For the particular degree-1 factor-counting vector $\wt{\mathbf{f}}_1$
that failed before, though, everything is now fine,
since $\wt{\mathbf{f}}_1$ and $\wt{\mathbf{g}}$
are just permutations of each other.

So in order to guarantee the existence of a triangularization
using this new configuration of degree-3 and degree-2 factors,
in general we will need \emph{two} majorization conditions to be satisfied.
One can now easily imagine the combinatorial nightmare
that will almost certainly accompany any effort to devise general conditions
that are necessary for triangularization in the arbitrary field setting.
This is why we have contented ourselves with only a brief discussion
of simple sufficient conditions for triangularizability
when there are at least three degrees of irreducible divisor present.
We leave the investigation of necessary conditions
for triangularizability for further research.
\end{ejem}

This brings us now to our final result,
which gives a generalized sufficient condition
for guaranteeing the existence of a triangularization over an arbitrary field.

\begin{prop} \label{prop.triang}
  Let $P(\la)$ be a strictly regular $\nbyn$ polynomial matrix
  of degree $d$ over a field $\,\bF$.
  Let $S(\la)$ be the Smith form of $P(\la)$,
  and assume that all irreducible divisors are degree $k$ or less.
  Organize the degree-$2$ through degree-$k$ factors
  into $n$ polynomials $q_1(\la), q_2(\la), \dotsc, q_n(\la)$
  in the same way as in the proof of Lemma~\ref{Stacking},
  and define the degree gaps $g_i := d - \deg(q_i)$.
  \parens{Note that some of the $g_i$ may be negative.}
  If the degree-1 factor-counting vector for the diagonal of $S(\la)$
  majorizes the gap vector $\mathbf{g} := (g_1,g_2,\dotsc,g_n)$,
  then $P(\la)$ has a triangularization of degree $d$.
\end{prop}
\begin{proof}
% Based on the hypotheses given,
% Because of the majorization hypothesis,
It is possible to employ the techniques
pictured in Examples \ref{ex.2x2-triang} and \ref{ex.3x3-triang}
(i.e., following the pattern of the proof of Lemma~\ref{Stacking}
in distributing all irreducible factors of degree two and higher,
and then filling in the rest of the available spaces
with all of the remaining degree-$1$ factors)
to design a target diagonal in which all of the entries have degree $d$.
Note that the degree-$\ell$ factor-counting vectors for $2 \le \ell \le k$
in this target diagonal are all $1$-homogeneous,
and thus are definitely all realizable by spreading out
the irreducible divisors in the Smith form
using Corollary~\ref{cor.unimod-trans}
and Lemma~\ref{lem.conversion-to-homogen-version}.
The majorization hypothesis
about the degree-$1$ factor-counting vector for the diagonal of $S(\la)$
then suffices to imply
that the degree-$1$ factor-counting vector for the target diagonal,
i.e., the gap vector $\mathbf{g}$,
can also be realized using Corollary~\ref{cor.unimod-trans}.
Once all of these factor-counting vectors for the target diagonal are realized,
we will have attained the desired degree-$d$ triangularization for $P(\la)$.
\end{proof}

In the statement of Proposition~\ref{prop.triang} it was noted
that under the given conditions,
it is possible for the gap vector $\mathbf{g}$ to have negative entries.
Whenever this occurs,
then it is impossible for any conceivable degree-1 factor-counting vector
for the Smith form
to majorize $\mathbf{g}$, since all entries of a factor-counting vector are non-negative.
In this scenario, then, Proposition~\ref{prop.triang}
% gives no information
tells us nothing about the existence or non-existence of a triangularization.
Other arrangements of the higher degree irreducible factors along the diagonal
may still lead to a triangularization,
as illustrated in Example~\ref{ex.3x3-triang}.

\begin{remark}
\rm
  Note that the condition in Proposition~\ref{prop.triang}
  for ensuring triangularizability
  can be adapted to regular matrix polynomials $P(\la)$
  having nontrivial infinite spectral structure.
  First apply a M\"{o}bius transformation
  to transform $P(\la)$ into a matrix polynomial $Q(\la)$
  with only finite spectral structure,
  i.e., into a strictly regular matrix polynomial,
  as was done in the proof of Theorem~\ref{thm.QTR-regular}.
  Since any M\"obius transformation preserves the degree
  of any irreducible divisor of degree two or higher by Lemma~\ref{lem.degree-preservation},
  all of the degree-$\ell$ factor-counting vectors of $Q(\la)$
  will be exactly the same as those of $P(\la)$,
  except for $\ell = 1$.
  The partial multiplicities at $\infty$ for $P(\la)$
  will turn into partial multiplicities for $Q(\la)$
  at some degree-1 irreducible $\la - \omega$,
  hence the degree-$1$ factor-counting vector
  for the Smith form of $Q(\la)$
  will be equal to the sum of the degree-$1$ factor-counting vector
  for the Smith form of $P(\la)$
  together with the partial multiplicity sequence for $P(\la)$ at $\infty$.
  In other words, the infinite partial multiplicities
  effectively get included with all of the degree-1 irreducible factors.
  We can now apply the majorization condition in Proposition~\ref{prop.triang}
  (or Proposition~\ref{prop.triang-characterization})
  to determine if a triangularization for $Q(\la)$ is guaranteed.
  If it is, then the inverse M\"{o}bius transformation
  applied to this triangularization for $Q(\la)$
  provides a spectrally equivalent triangularization for $P(\la)$.
  Note that this generalization to all regular matrix polynomials
%   \tcb{(was hinted at, but)}
%   did not explicitly appear in~\cite{TasTisZab}.
  appears in \cite{TasTisZab, TisZab} for real matrix polynomials.
\end{remark}

\section{Conclusion} \label{sect.conclusion}

% \tcb{(Para reprising main result, giving some context with previous work)} \,\,
This work has shown that any regular matrix polynomial $P(\la)$
over an arbitrary field $\bF$
is spectrally equivalent
to a $k$-quasi-triangular matrix polynomial over $\bF$
of the same size and degree,
where $k$ is the largest degree
among all of the irreducible divisors of $P(\la)$.
This extends and generalizes the earlier work in \cite{TasTisZab},
which found triangularizations and $2$-quasi-triangularizations
for regular matrix polynomials over algebraically closed fields
and the real field $\bR$.
We have also shown that for any field $\bF$,
this $k$ is the best possible general bound on the diagonal block sizes of quasi-triangularizations
that holds for all regular matrix polynomials over $\bF$.

% \tcb{(Para highlighting some of the new tools developed)}
Several new tools and results were developed
in order to achieve this extension to arbitrary fields.
Among these are:
\begin{itemize}
  \item a technique to allow the flexible but controlled movement of individual irreducible factors
        up and down the diagonal of a triangular polynomial matrix via %(constructive?)
        unimodular transformations, %(Unimodular Transfer Lemma)
  \item %a new combinatorial property %(homogeneous partitioning property)
        a new homogeneous partitioning property
        of ``tightly packed'' integer multisets,
%   \item as in \cite{TasTisZab}, majorization also plays a role (DELETE??)
%   \item off-diagonal degree reduction, sweep by block-diagonals (DELETE??)
  \item %a reformulation of one of the fundamental properties of M\"obius transformations in \cite{Mobius}
        a reformulation of the interaction of M\"obius transformations with spectral data,
        in a way that makes it easier to work with higher degree irreducible divisors.
\end{itemize}

% \medskip
% The goal of this work was to show that every regular matrix polynomial over any field
% can be quasi-triangularized in a spectrally equivalent way that preserves the degree.
% This was achieved with the help of several new tools
% that were used to rearrange the factors along the diagonal of the Smith form
% in such a way that the difference in degree between any two diagonal entries is bounded above.
% The new matrix, which is now triangular, could then be partitioned
% so that the diagonal blocks could be transformed to have degree $ d $.
% The resulting quasi-triangularization has diagonal block sizes
% bounded by the highest degree among the irreducible divisors.
% When the infinite spectrum is not empty,
% some additional techniques involving M\"{o}bius transformations are required,
% and an additional assumption that the field
% contains an element disjoint from the spectrum is needed.

% \tcb{(Closing para on partially resolved issues and remaining open questions,
%       e.g., triangularizability)}
A number of issues remain to be settled,
especially ones related to the size of diagonal blocks
in quasi-triangularizations.
Although we know that these block sizes need never be any larger than $k$,
and that sometimes they are all forced to be of size exactly $k$,
very often quasi-triangularizations
can be found with diagonal block sizes much smaller than the upper bound $k$.
 For given spectral data, can one predict how small
the diagonal blocks can be made in a quasi-triangularization,
and indeed when these blocks can all be made $1 \times 1$,
i.e., when can we actually triangularize?
Some limited results were given along these lines,
but much about this question still remains open.

% \medskip
% Some additional discussion and examples were then given
% regarding the upper bound on the block sizes.
% It was shown that the upper bound is sharp,
% but it was also shown that in many cases, the upper bound is a huge over estimate.
%  Finally, we gave a sufficient condition
% for guaranteeing that the given polynomial matrix can be triangularized,
% diagonal blocks all size 1.
% While the condition given is sufficient,
% it is not necessary, and it is the view of the authors
% that obtaining a necessary condition will likely be a combinatorial nightmare.

%%%%%%%%%%%%%%%%%%%%%%%%%%%%%%%%%%%%%%%%%%%%%%%%%%%%%%%
%%%%%%%%%%%%%%%%%%%%%%%%%%%%%%%%%%%%%%%%%%%%%%%%%%%%%%%

\end{document}